\documentclass[11pt]{article}
\usepackage{amsmath,amsthm,mathrsfs,bm}

\usepackage[initials,alphabetic]{amsrefs}
\usepackage[utf8]{inputenc}
\usepackage{amsfonts,amstext,amssymb,verbatim,epsfig,psfrag,stmaryrd,extarrows}
\usepackage{pgf,tikz}
\usepackage{float}
\usetikzlibrary{arrows}
\usepackage{hyperref}
\usepackage{color}
\usepackage{transparent}
\usepackage{subfigure}
\usepackage{prettyref}
\newrefformat{pr}{Proposition \ref{#1}}
\newrefformat{co}{Corollary \ref{#1}}
\newrefformat{le}{Lemma \ref{#1}}
\newrefformat{th}{Theorem \ref{#1}}
\newrefformat{ex}{Example \ref{#1}}
\newrefformat{re}{Remark \ref{#1}}
\newrefformat{se}{Section \ref{#1}}

\usepackage{esint}

\usepackage{bbm,mathrsfs}

\usepackage{xcolor}
\usepackage[normalem]{ulem}

\hypersetup{colorlinks=true,pdfborder=000,  citecolor  = blue, citebordercolor= {magenta}}
\hypersetup{linkcolor=purple}
\makeatletter
\let\reftagform@=\tagform@
\def\tagform@#1{\maketag@@@{(\ignorespaces\textcolor{purple}{#1}\unskip\@@italiccorr)}}
\renewcommand{\eqref}[1]{\textup{\reftagform@{\ref{#1}}}}
\makeatother

\makeatletter
\allowdisplaybreaks

\DeclareUrlCommand\ULurl@@{%
  \def\UrlLeft{\uline\bgroup}%
  \def\UrlRight{\egroup}}
\def\ULurl@#1{\hyper@linkurl{\ULurl@@{#1}}{#1}}
\DeclareRobustCommand*\ULurl{\hyper@normalise\ULurl@}
\makeatother

\def\lessim{\ \lower4pt\hbox{$
		\buildrel{\displaystyle <}\over\sim$}\ }
\def\gessim{\ \lower4pt\hbox{$\buildrel{\displaystyle >}
		\over\sim$}\ }

\def\si{\sigma}

\def\eps{{\varepsilon}}

\newcommand{\indi}{\ensuremath{\boldsymbol 1}}

\newcommand{\bt}{\boldsymbol{t}}

\newcommand{\pref}{\prettyref}

\newtheorem{lemma}{\bf Lemma}[section]

\newtheorem{theorem}[lemma]{\bf Theorem}

\newtheorem{example}{\bf Example}

\newtheorem{proposition}[lemma]{\bf Proposition}

\theoremstyle{remark}

\newtheorem{remark}{Remark}[section]

\numberwithin{equation}{section}

\newcommand{\8}{\infty}

\newcommand{\ex}{\mathcal{E}}

\newcommand{\px}{\mathcal{P}}

\newcommand{\nz}{\mathbb{N}}
\newcommand{\rz}{\mathbb{R}}

\newcommand{\ez}{\mathbb{E}}

\newcommand{\pz}{\mathbb{P}}

\newcommand{\Ga}{\Gamma}

\newcommand{\sfT}{\mathsf T}

\newcommand{\sfa}{\mathsf a}
\newcommand{\sfb}{\mathsf b}

\newcommand{\al}{\alpha}
\newcommand{\de}{\delta}
\renewcommand{\si}{\sigma}
\newcommand{\ga}{\gamma}
\newcommand{\la}{\lambda}
\renewcommand{\bt}{\beta}

\newcommand{\Crt}{\mathrm{Crt}}
\newcommand{\Var}{\mathrm{Var}}
\newcommand{\Cov}{\mathrm{Cov}}
\newcommand{\GOE}{\mathrm{GOE}}
\newcommand{\sgn}{\mathrm{sgn}}

\newcommand{\dd}{\mathrm{d}}

\newenvironment{Proof of lemma}{\noindent{\bf Proof of Lemma}}{\hfill$\Box$\newline}
\newenvironment{Proof of theorem}{\noindent{\bf Proof of Theorem}}{\hfill{\footnotesize${\square}$}\newline}
\newenvironment{Proof of theorems}{\noindent{\bf Proof of Theorems}}{\hfill$\Box$\newline}
\newenvironment{Proof of proposition}{\noindent{\bf Proof of Proposition}}{\hfill$\Box$\newline}
\newenvironment{Proof of propositions}{\noindent{\bf Proof of Propositions}}{\hfill$\Box$\newline}
\newenvironment{Proof of exercise}{\noindent{\it Proof of Exercise:}}{\hfill$\Box$}
\begin{document}
\title{Complexity of Gaussian random fields with isotropic increments}
\author{Antonio Auffinger \thanks{Department of Mathematics, Northwestern University, tuca@northwestern.edu, research partially supported by NSF Grant CAREER DMS-1653552 and NSF Grant DMS-1517894.} \\
	\small{Northwestern University}\and Qiang Zeng \thanks{Department of Mathematics, University of Macau, qzeng.math@gmail.com, research partially supported by SRG 2020-00029-FST and FDCT 0132/2020/A3.}\\
	\small{University of Macau}}
\date{\today}

%We provide
%\end{abstract}
\maketitle
\abstract{We study the energy landscape of a model of a single particle on a random potential, that is, we investigate the topology of level sets of smooth random fields on $\mathbb R^{N}$ of the form $X_N(x) +\frac\mu2 \|x\|^2,$ where $X_{N}$ is a Gaussian process with isotropic increments. We derive asymptotic formulas for the mean number of critical points with critical values in an open set as the dimension $N$ goes to infinity. In a companion paper, we provide the same analysis for the number of critical points with a given index.}

\section{Introduction}
In this paper we provide asymptotics for the number of critical points of Gaussian random fields with isotropic increments (a.k.a.~locally isotropic Gaussian random fields) in the high dimensional limit. The definition of locally isotropic fields was first formulated by Kolmogorov about 80 years ago \cite{Ko41} for the application in statistical theory of turbulence; see \cite{Ya57} for an account of background and early history.

The model is defined as follows. Let $B_N \subset \mathbb R^N$ be a sequence of subsets and  let $H_N : B_N \subset \mathbb R^N \to \mathbb R$ be given by
\begin{align}\label{hamil}
  H_N(x) = X_N(x) +\frac\mu2 \|x\|^2,
\end{align}
where $\mu \in \rz$, $\|x\|$ is the Euclidean norm of $x$, and $X_N$ is a Gaussian random field that satisfies
\begin{align*}
  \ez[(X_N(x)-X_N(y))^2]=N D\Big(\frac1N \|x-y\|^2\Big), \ \ x,y\in \rz^N.
\end{align*}
Here the function $D:\mathbb R_+ \to \mathbb R_+$ is called the correlator (or structure) function and $\rz_+=[0,\8)$. It determines the law of $X_N$ up to an additive shift by a Gaussian random variable. Complete characterization of all correlators was given in the work of  Yaglom \cite{Ya57} (see also the general form of a positive definite kernel due to Schoenberg \cite{Sch38}). In short, if $D$ is the correlator function for all $N\in \nz$, then $X_N$ must belong to one of the following two classes (see also \cite{Kli12}*{Theorem A.1}):
\begin{enumerate}
  \item \textbf{Isotropic fields.} There exists a function $B:\rz_+\to \rz$ such that
      \begin{align*}
        \ez[X_N(x)X_N(y)]=N B\Big(\frac1N\|x-y\|^2\Big)
      \end{align*}
      where $B$ has the representation
      \begin{align*}
      B(r) = c_0+\int_{(0,\8)} e^{-r t^2} \nu(\dd t),
      \end{align*}
      and $c_0\in \rz_+$ is a constant and $\nu$ is a finite measure on $(0,\8)$. In this case,
      \begin{align*}
        D(r)=2(B(0)-B(r)).
      \end{align*}
  \item \textbf{Non-isotropic field with isotropic increments.} The correlator $D$ can be written as
      \begin{align}\label{eq:drep}
        D(r) = \int_{(0,\8)} (1-e^{-rt^2})\nu(\dd t) + Ar, \ \ r\in \rz_+,
      \end{align}
      where $A\in \rz_+$ is a constant and $\nu$ is a $\si$-finite measure with
      \begin{align*}
        \int_{(0,\8)} \frac{t^2}{1+t^2}\nu(\dd t) <\8.
      \end{align*}
\end{enumerate}
See \cite{Ya87}*{Section 25.3} for more details on locally isotropic fields. Case 1 is known as short-range correlation (SRC)  processes and case 2 as long-range correlation (LRC) in the physics literature.

Here is a special example of $B(r)$ and $D(r)$, which we learned from Yan Fyodorov.

\begin{example}\rm
We  assume $c_0=0$ and $A=0$. For fixed $\eps>0$ and $\ga>0$, let
\begin{align*}
  \nu(\dd x) = 2 e^{-\eps x^2}x^{2\ga-3} \dd x.
\end{align*}
The case $\ga>1$ corresponds to SRC while the case $0<\ga\le1$ is LRC field. Indeed, if $\ga>1$,
\begin{align*}
  B(r)&=\int_0^\8 2e^{-r t^2}e^{-\eps t^2} t^{2\ga-3} \dd t = \frac{\Ga(\ga-1)}{(r+\eps)^{\ga-1}},
\end{align*}
while if $0< \ga<1$, using integration by parts,
\begin{align*}
  D(r)=\int_0^\8(e^{-\eps y}- e^{-(r+\eps)y})y^{\ga-2} \dd y = \frac{\Ga(\ga)}{1-\ga}[(r+\eps)^{1-\ga}-\eps^{1-\ga}].
\end{align*}
The case $\ga=1$ can be obtained by sending $\ga\uparrow 1$ and using the dominated convergence theorem with the control function $f(y)=(e^{-\eps y}- e^{-(r+\eps)y})y^{-1}$ for $y\le1$ and $=(e^{-\eps y}-e^{-(r+\eps)y})y^{-1/2}$ for $y>1$. Then if $\ga=1$, we have
\begin{align*}
  D(r)=\log(1+r/\eps).
\end{align*}
In the LRC case, we see that the long range covariance behaves like a high dimensional analog of fractional Brownian motions. %However, in dimension one, isotropic increments force the stochastic process to have constant mean.
\end{example}

\begin{remark}
Observe that any Bernstein function vanishing at 0 is a structure function. This is a consequence of the L\'evy--Khintchine representation of Bernstein functions; see e.g.~the monograph \cite{SSV}, which also contains a comprehensive list of complete Bernstein functions. Conversely, any structure function is a Bernstein function. It follows that any correlator function $D$ must be concave, infinitely differentiable, and non-decreasing on $(0,+\8)$. Moreover, we have $D'(r)\ge 0, D''(r)\le 0, D'''(r)\ge 0$ for $r>0$.
\end{remark}

\begin{remark}
One should not confuse SRC/LRC with short-range/long-range dependence. SRC here refers to the fact that $\ez(X_N(x)X_N(y))$ decays as $\|x-y\|\to \8$ while for LRC it may not. Short-range dependence requires the autocovariance function to have exponential decay.
\end{remark}

\subsection{Previous results}

The Hamiltonian \eqref{hamil} has been considered in many papers, from physics to mathematics, since 1950s. In particular, the model was introduced by Mezard--Parisi \cite{MP91} and Engel \cite{En93} among others as a model for a classical particle confined to an impenetrable spherical box or a toy model describing elastic manifolds propagating in a random potential \cite{Fy04}. A nice historical account can be found in \cite{FS07} which also contains the phase diagram ($T-\mu$ relation) for the model at positive temperature. At zero temperature, in the seminal paper \cite{Fy04}, Fyodorov considered the case of isotropic fields (SRC) and computed the mean total number of critical points, finding a phase transition for different values of $\mu$ and $D''(0)$. In a subsequent and impressive paper,  \cite{FW07} computes the mean number of saddles and minima for SRC fields. This paper considered a more general model where $\mu\|u\|_2^2/2$ is replaced by $NU(\|u\|_2^2/N)$ for suitable confining potential $U$.

Still in the case of isotropic fields, \cite{FN12} computed the mean number of minima and studied the phenomena of topology trivialization and the relation of this quantity with the Tracy--Widom distribution. More recently, \cite{CS18} considered the mean number of critical points of a fixed index and for finite $N$. The reader is also invited to take a look at  \cites{BD07, YV18, Kli12}.

For a similar Hamiltonian defined on the $N$ dimensional sphere, known as the spherical $p$-spin model, the rigorous study of the complexity of saddles and minima started in \cite{ABC13} and now has solid body of work including \cites{ABA13, subag2017complexity, Mihai}. For the physics predictions of this model, the reader should consult \cites{CL04, MPV} and the references therein.

 All of the rigorous work above only considered isotropic Gaussian fields (SRC case) or spherical spin glasses. In the physics literature, the complexity of LRC Gaussian fields was studied in a sequence of two remarkable papers \cites{FS07, FB08}. However, the lack of symmetry in this model imposes a difficult obstacle and no rigorous results on the complexity are currently known. 

The main purpose of this article and its companion paper is to close this gap by providing a comprehensive rigorous study of the complexity of LRC Gaussian fields. We extend and recover the predictions made by Fyodorov, Bochaud and Sommers \cites{FS07, FB08}. In this first paper, we focus on the high dimensional limit of the expected number of critical points for when the domain and value of the fields are constrained to any particular set. In the companion paper \cite{AZ22}, we will provide information on local minima and saddles with given indices.

A  word of comment is needed here. One of the main differences between LRC and SRC fields is the fact that the variance of an LRC field may change from location and the gradient $\nabla H_{N}$ is no longer independent of $H_{N}$. The main novelty of our two papers is the development of techniques to overcome this difficulty. Another set of important techniques to deal with ``non-invariant'' fields  was also recently developed in \cites{BBMsd, BBMsd2}; these do not seem to apply to the model we consider. 

\subsection{Main results}
To state our results, let $B_N\subset \rz^N$ and $E\subset \rz$ be (a sequence of) Borel sets. We define
\begin{align*}
%\mathrm{Crt}_{N,k}(E,B_N) &= \#\{x\in B_N: \nabla H_N(x)=0, \frac1NH_N(x)\in E, i(\nabla^2 H_N(x))=k\},\\
	\mathrm{Crt}_N(E,B_N) &= \#\{x\in B_N: \nabla H_N(x)=0, \frac1N H_N(x)\in E \}.
\end{align*}
%and write $\Crt_{N,k}(u,B_N)=\mathrm{Crt}_{N,k}(E,B_N) $ and $\Crt_{N}(u,B_N)=\mathrm{Crt}_{N}(E,B_N)$ if $E=(-\8,u)$.
%Here $i(\nabla^2 H_N(x))$ is the index (or number of negative eigenvalues) of the Hessian $\nabla^2 H_N(x)$.
Throughout the paper we will consider the following extra assumptions on $X_N$.

\textbf{Assumption I} (Smoothness). The function $D$ is four times differentiable at $0$ and it satisfies
\begin{align}\label{eq:asmp}
	0<|D^{(4)}(0)|<\8.
\end{align}

\begin{remark}
    %It follows that $D$ is a Bernstein function; see the comprehensive monograph \cite{SSV}.
    By Kolmogorov's criterion, Assumption I ensures that almost surely the field $H_{N}$ is twice differentiable. Moreover, Assumption I guarantees $D'(0), D''(0)$ and $D'''(0)$ exist and are non-zero. This implies that for $ r>0$
\begin{align*}
  D(r)>0, \ \ D'(r)>0,\ \ D''(r)<0,\ \ D'''(r)> 0,
\end{align*}
and in particular all these functions are strictly monotone. From here we also know that $D(r)\le D'(0) r$ and when $\nu$ in the representation \pref{eq:drep} is not a finite measure (or equivalently in case 2), $\lim_{r\to \8} D(r)=\8$.
\end{remark}

\textbf{Assumption II} (Pinning). We have $$X_N (0) = 0.$$

\begin{remark}
Random fields with isotropic increments are high dimensional generalizations of stochastic processes with stationary increments in dimension one. It is a common practice to assume such processes (like Brownian motion or Poisson processes) to start from 0. Therefore, Assumption II is a natural choice for studying random fields with isotropic increments. Note that only the trivial isotropic field ($X_{N}=0$) satisfies Assumption II. %In other words, under assumption II, we only consider non-isotropic fields.
\end{remark}

We first consider the average of the {\it total} number of critical points of $H_N$. Then we count the average number of critical points  of $H_N$ with a given fixed critical value. Although the first result can be essentially obtained from the second, the formula and proof for the first are clearer, thus we state them separately. We hope this organization provides a gentle introduction to the reader to appreciate the second result, where most of the novelty (and difficulty of the paper) resides.

%We are now ready to state the main results in this paper.
The following condition is only needed when the critical value is not restricted.

\textbf{Assumption III} (Domain growth). Let $z_N$ be a standard $N$ dimensional Gaussian random variable. There exist $\Xi$ or $\Theta$ such that the sequence of sets $B_N$ satisfies
\begin{align}
\lim_{N\to \8} \frac1N \log \pz(  z_N \in  |\mu| B_N/\sqrt{D'(0)} ) &= -\Xi\le 0, &\mu\neq 0,\label{eq:bnasp1}\\
\lim_{N\to \8} \frac1N \log |B_N|& = \Theta, &\mu=0.\label{eq:bnasp2}
\end{align}

\begin{remark}
Assumption III serves to select domains in the right scale and it is less restrictive. As seen in the proof of our main theorems, the reader could consider other sequence of sets $B_{N}$ provided some knowledge of their volumes. %
\end{remark}

\begin{theorem}%(Total number of critical points).
\label{th:ttcpx}
Under Assumptions I, II, and III, we have
\begin{align*}
\lim_{N\to\8}&\frac1N\log\ez \Crt_N(\rz,B_N)\\
&=
\begin{cases}
-\Xi, & |\mu|>\sqrt{-2D''(0)},\\
-\log\frac{|\mu|}{\sqrt{-2D''(0)}}+\frac{\mu^2}{-4D''(0)} -\frac12-\Xi, & 0<|\mu|\le \sqrt{-2D''(0)},\\
\log\sqrt{-2D''(0)}-\frac12-\frac12\log(2\pi)- \frac12\log[D'(0)] +\Theta, & \mu=0.
\end{cases}
\end{align*}
\end{theorem}
\begin{remark}
If we let $J=\sqrt{-2D''(0)}$ and $\Xi=0$ as in \cite{Fy04}, the second case can be rewritten as
\begin{align}
  \Sigma_{\mu,D} = \frac12\Big(\frac{\mu^2}{J^2}-1\Big)-\log\frac{\mu}J \ge 0.
\end{align}
which matches Fyodorov's result for isotropic Gaussian random fields.
\end{remark}

Next, we state our main result on the number of critical points with values in an open set $E\subset \mathbb R$ and confined to a shell $B_{N}(R_{1},R_{2})=\{ x\in \mathbb R^{N}: R_{1}< \frac{\| x\|}{\sqrt N} < R_{2} \}$. This is a natural choice, as the isotropy assumption implies rotational invariance. To emphasize the dependence on $R_1$ and $R_2$, we also write
\begin{align*}
  \Crt_{N}(E, (R_1,R_2)) & =\Crt_{N}(E, B_{N}(R_1,R_2)).
%  \Crt_{N,k}(E, (R_1,R_2)) & =\Crt_{N,k}(E, B_{N}(R_1,R_2)).
\end{align*}
We will assume the following technical assumption:

\textbf{Assumption IV}. \pref{eq:asmp1} and \pref{eq:asmp2} hold for $x\in\rz^N\setminus\{0\}$.

This assumption is rather mild, and is satisfied by e.g.~the so called Thorin--Bernstein functions; see \pref{se:3} for more details.

\begin{theorem}\label{th:cpsublevel}
Let $0\le R_1<R _2\le \8$ and $E$ be an open set of $\rz$. Assume Assumptions I, II, IV, and $|\mu|+\frac1{R_2} > 0$.  Then
 	\[
 	\lim_{N\to\8} \frac1N \log\ez \Crt_{N}(E, B_{N}(R_1,R_2)) = \frac12 \log[-4D''(0)] -\frac12\log D'(0) +\frac12+\sup_{(\rho,u,y)\in F}\psi_*(\rho,u,y)
 	\]
 	where  $F=\{(\rho,u,y):y\in\rz, \rho \in (R_1, R_2), u\in    E\}$, and the function $\psi_*$ is given explicitly in \eqref{eq:psifunction}.
 \end{theorem}

The condition $|\mu|+\frac1{R_2} > 0$ merely says $R_2<\8$ if $\mu=0$, which is necessary to get non-trivial asymptotics as we saw in \pref{th:ttcpx}. In \pref{ex:2} at the end of Section \ref{se:4}, we provide details on how to recover Theorem \ref{th:ttcpx} from Theorem \ref{th:cpsublevel} when $B_N$ is a shell, which also provides some insight on the location of the majority of critical points.

Let us end this section with a brief description of the proofs, highlighting the main difference from previous results that also computed the mean number of critical points. Similar to many results in this area, we use the Kac--Rice formula as a starting point. Since our fields do not have constant variance and in particular $H_{N}$ is correlated to $\nabla H_{N}$, we are unable to trace a direct parallel to random matrix theory as done in \cites{ABC13, ABA13, subag2017complexity} where the Hessian is distributed as a matrix from the Gaussian Orthogonal Ensemble (GOE) plus a scalar matrix. This small difference actually provides major obstacles. To go around this difficulty, we first find out the conditional distribution of the Hessian  after some matrix manipulations. The GOE matrix appears as a summand of a principal submatrix which itself is correlated to the other element on diagonal. Then we estimate from above and below the conditional expectation of  the Hessian given $H_{N}$. Matching upper and lower bounds only come after long and non-trivial calculations and careful asymptotic analysis.

The rest of the paper is organized as follows. In Section \ref{se:whole}, we fix our notation and provide some preliminary facts before giving the proof of \pref{th:ttcpx}. We find the (conditional) distribution of the Hessian with some of the tools from random matrix theory in \pref{se:3} and establish various results on exponential tightness in \pref{se:exptt}, both of which will serve as the starting point for computing complexity functions in this paper and the companion paper \cite{AZ22}. We prove Theorem \ref{th:cpsublevel} in \pref{se:4} .

\subsection{Acknowledgments}
We would like to thank Yan Fyodorov for suggesting the study of fields with isotropic increments and providing several references. %Julian Gold for providing the reference \cite{lazutkin1988signature}. A.A. thanks the hospitality of the International Institute of Physics where part of this work was developed. Q.Z. thanks Jiaoyang Huang for showing him the identity \pref{eq:elak}.

\section{Preliminary facts and proof of \pref{th:ttcpx}}\label{se:whole}
Throughout, we regard a vector to be a column vector. We write e.g.~$C_{\mu,D}$  for a constant depending on $\mu$ and $D$ which may vary from line to line. For $N\in \nz$, let us denote $[N]=\{1,2,...,N\}$. For a vector $(y_1,...,y_N)\in \rz^N$, we write $L(y_1^N)=\frac1N \sum_{i=1}^N \de_{y_i}$ for its empirical measure. Recall that an $N \times N$ matrix $M$ in the Gaussian Orthogonal Ensemble (GOE) is a symmetric matrix with centered Gaussian entries that satisfy
\begin{align}\label{eq:goemat}
  \ez(M_{ij})=0,\ \ \ez(M_{ij}^2) =\frac{1+\de_{ij}}{2N}.
\end{align}
We will simply write $\GOE_N$ or $\GOE(N)$ for the matrix $M$. Denoting by $\lambda_1 \leq \dots \leq \lambda_N$ the eigenvalues of $M$, we write $L_N =L(\la_1^N)= \frac{1}{N} \sum_{k=1}^N \delta_{\lambda_k}$ for its empirical spectral measure. From time to time, we may also use $\la_k$ to denote the $k$th smallest eigenvalue of $\GOE_{N+1}$  or $\GOE_{N-1}$. This should be clear from context and should not affect any results as we only care about the large $N$ behavior eventually. For a closed set $F\subset \rz$, we denote by  $\px(F)$ the set of probability measures with support contained in $F$. We equip the space $\px(\rz)$ with the weak topology, which is compatible with the distance
\begin{align}\label{eq:measd}
d(\mu,\nu):=\sup\Big\{\Big|\int f \dd \mu-\int f \dd\nu \Big|: \|f\|_\8\le1, \|f\|_L\le 1\Big\},\quad \mu,\nu \in\px(\rz),
\end{align}
where $\|f\|_\8$ and $\|f\|_L$ denote the $L^\8$ norm and Lipschitz constant of $f$, respectively. Let $B(\nu,\de)$ denote the open ball in the space $\px(\rz)$ with center $\nu$ and radius $\de$ w.r.t.~to the distance $d$ given in \pref{eq:measd}. Similarly, we write $B_K(\nu,\de)=B_K(\nu,\de) \cap \px([-K,K])$ for some constant $K>0$. We denote by $\si_{{\rm sc}}$ the semicircle law scaled to have support $[-\sqrt2,\sqrt2]$.

We will frequently use the following facts which are consequences of large deviations. Using the large deviation principle (LDP) of empirical measures of GOE matrices \cite{BG97}, for any $\de>0$, there exists $c=c(\de)>0$ and $N_\de>0$ such that for all $N>N_\de$,
\begin{align}\label{eq:conineq}
\pz(L(\la_1^N)\notin B(\si_{\rm sc}, \de) )\le e^{-cN^2}.
\end{align}
On the other hand, the LDP of the smallest eigenvalue of GOE matrices \cite{BDG01} states that $\la_1$ satisfies an LDP with speed $N$ and  a good rate function
\begin{align}
J_{1}(x)&=
\begin{cases}
k\int_{x}^{-\sqrt{2}}\sqrt{z^2-2}\dd z, & x\le -\sqrt 2,\\
\8, & x>-\sqrt2,
\end{cases}\notag \\
&=\begin{cases}
\frac12\log2 -\frac12 x\sqrt{x^2-2}-\log(-x+\sqrt{x^2-2}),& x\le -\sqrt2,\\
\8,& x>-\sqrt2.
\end{cases}\label{eq:rf1ev}
\end{align}
In particular, writing $\la_N^*=\max_{i\in [N]} |\la_i|$ for the operator norm of an $N\times N$ GOE matrix, by \cite{BDG01}*{Lemma 6.3}, there exists $N_0>0$ and $K_0>0$ such that for $K>K_0$ and $N>N_0$,
\begin{align}\label{eq:ladein}
\pz(\la_N^*>K)\le e^{-NK^2/9}.
\end{align}
This can also be seen directly from the LDP of $\la_1$, even though it was originally proved as a technical input for the LDP of $\la_1$. It follows that there exists an absolute constant $C>0$ such that
 \begin{align}\label{eq:goenorm}
 \ez[{\la_N^*}^{k}] \le C^{k}
 \end{align}
for any $k\ge0$ and $N>N_0$. For a probability measure $\nu$ on $\rz$, let us define
\begin{align}\label{eq:psidef0}
  \Psi(\nu,x)=\int_\rz \log |x-t| \nu(\dd t), \qquad \Psi_*(x)=\Psi(\si_{\rm sc}, x).
\end{align}
By calculation,
\begin{align}\label{eq:phi*}
\Psi_*(x) &= \frac12 x^2-\frac12 -\frac12\log2-\int_{\sqrt2}^{|x|} \sqrt{y^2-2}\dd y \indi\{|x|\ge\sqrt2\} \notag \\
&=
\begin{cases}
\frac12 x^2-\frac12 -\frac12\log2,& |x|\le \sqrt2,\\
\frac12x^2 -\frac12-\log2 - \frac12 |x| \sqrt{x^2-2}+\log(|x|+\sqrt{x^2-2}) ,& |x|>\sqrt2.
\end{cases}
\end{align}
Note that $\Psi_*(x)-\frac{x^2}{2}\le -\frac12-\frac12\log2$.

Let $z$ be a standard Gaussian r.v.~and $\Phi$ the c.d.f.~of $z$. For $a\in\rz, b>0$, we have
% \begin{align}\label{eq:gau2}
%   \ez[(a+bz)\indi\{a+bz>0\}]=a\Phi(\frac{a}b) +\frac{b}{\sqrt{2\pi}} e^{-\frac{a^2}{2b^2}},
% \end{align}
% which is strictly increasing as $a$ increases, and
% \begin{align}\label{eq:gau3}
%   \ez[-(a+bz)\indi\{a+bz<0\}]= -a\Phi(-\frac{a}b) +\frac{b}{\sqrt{2\pi}} e^{-\frac{a^2}{2b^2}},
% \end{align}
% which is strictly decreasing as $a$ increases. These two functions are 1-Lipschitz convex in $a$. Also,
\begin{align}\label{eq:absgau}
\sqrt{\frac2\pi}b\le \ez|a+bz|=\frac{\sqrt2 b }{\sqrt{\pi}}e^{-\frac{a^2}{2b^2}}
 +a(2\Phi(\frac{a}{b})-1)\le \sqrt{\frac2\pi}b+|a|.
\end{align}
%We will also use repeatedly $\ez (z/\sqrt N)^k \le C^k$ for any $k\ge0$.
Unless specified otherwise, we always assume Assumptions I and II throughout. %Moreover, Assumption IV is assumed in Sections \ref{se:4}, \ref{se:5} and \ref{se:div}.

Let us prove the result for the total number of critical points. The strategy we employ is well-known and similar to the one developed in \cite{ABC13}: We start by applying the Kac--Rice formula and we derive the asymptotics in high dimensions with the use of random matrix theory and large deviation principles. The proof is somewhat straight-forward since we do not face the main obstacle of the next sections, i.e., the dependence of $H_{N}$ and $\nabla H_{N}$.

\begin{proof}[Proof of Theorem \ref{th:ttcpx}]
Let $E$ be a Borel subset of $\rz$. By the Kac--Rice formula \cite{AT07}*{Theorem 11.2.1},
\begin{align} \label{eq:startingpoint}
  &\ez \Crt_{N}(E,B_N)=\int_{B_N} \ez[|\det \nabla^2 H_N(x)| \indi\{\frac1N H_N(x)\in E\}| \nabla H_N(x)=0] p_{\nabla H_N(x)}(0)\dd x,
%  &=\int_{B_N} \int_{E} \ez[|\det \nabla^2 H_N(x)|| \nabla H_N(x)=0, H_N(x)=N u] p_{\nabla H_N(x)}(0) \pz(H_N(x)/N \in \dd u)\dd x,
\end{align}
where $p_{\nabla H_N(x)}(t)$ is the p.d.f.~of ${\nabla H_N(x)}$ at $t$.

When $E=\rz$, the restriction on the range of $H_N(x)$ disappears. By independence of $\nabla H_N$ and $\nabla^2 H_N$ (see Lemma \ref{le:cov}) and dropping the restriction on index, the above formula simplifies to
\begin{align}\label{eq:KRSimple}
  \ez \Crt_N(\rz,B_N)=\int_{B_N} \ez[|\det \nabla^2 H_N(x)| ] p_{\nabla H_N(x)}(0)\dd x.
\end{align}
The following lemma is a random matrix computation.
\begin{lemma} Let $M$ be an $N \times N$ GOE matrix and set
$$P=aM -\left(b+\frac{\sigma}{\sqrt N}Z\right)I,$$
where $Z$ is a standard Gaussian random variable independent of $M$, $I$ is the identity matrix and $a,b, \sigma \in \mathbb R.$  Then
\begin{equation*}
\mathbb E |\det P| = \frac{\Gamma(\frac{N+1}{2})(N+1)a^{N+1}}{\sqrt{\pi}\sigma N^{\frac{N}{2}}e^{\frac{Nb^2}{2\sigma^2}}} \mathbb E \int \exp \left[ \frac{(N+1)x^2}{2} \left( 1- \frac{a^2}{\sigma^2} \right) + \frac{\sqrt{N(N+1)}axb}{\sigma^2} \right]  L_{N+1}(\dd x).
\end{equation*}
\end{lemma}
\begin{proof}
Use \cite{ABC13}*{Lemma 3.3} with $m=\frac{b}{a}$, $t=\frac{\sigma}{\sqrt{N}a}$ and sum over the eigenvalues.
\end{proof}
From Lemma \ref{le:cov}, $\nabla^2 H_N(x)$ and $\sqrt{-4D''(0)}M- (\sqrt{\frac{-2D''(0)}N}Z-\mu)I$ have the same distribution. Then with $$m=-\mu/\sqrt{-4D''(0)},$$ from the Lemma above with $a=\sqrt{-4D''(0)}, b = -\mu, \sigma = \sqrt{-2D''(0)}$ we obtain
% \begin{align*}
%   \ez&|\det\nabla^2 H_N(x)| =  \ez\prod_{i=1}^N |2\sqrt{-D''(0)} \la_i-(\sqrt{-2D''(0)} z-\mu)| \\
%   &=[-4D''(0)]^{N/2} \frac{\sqrt{2N}}{\sqrt{2\pi} }\int_{-\8}^\8 e^{-N(x-m)^2} \frac1{Z_N}\int_{\Lambda_N} \prod_{i=1}^N |\la_i-x| \prod_{i=1}^N e^{-\frac{N\la_i^2}{2}} \prod_{1\le i<j\le N} |\la_i-\la_j| \dd \la_1 \cdots \dd \la_N \dd x\\
%   &=[-4D''(0)]^{N/2} \sqrt{\frac{N}\pi} \frac{e^{-Nm^2}}{Z_N}\int_{\Lambda_N}\int_{-\8}^\8 \sum_{k=0}^N \indi_{(\la_k,\la_{k+1})}(x) e^{-Nx^2+2Nmx} \\
%   &\qquad \  \prod_{i=1}^N |\la_i-x| \prod_{i=1}^N e^{-\frac{N\la_i^2}{2}} \prod_{1\le i<j\le N} |\la_i-\la_j| \dd x \dd \la_1 \cdots \dd \la_N,
%   \end{align*}
% where $Z_N$ is a normalizing constant. For $k=0,...,N$, if $\la_k<x<\la_{k+1}$, writing $\la_i=\sqrt{\frac{N+1}N}y_i$ for $i=1,...,k$, $\la_i=\sqrt{\frac{N+1}N}y_{i+1}$ for $i=k+1,...,N$, and $x=\sqrt{\frac{N+1}N}y_{k+1}$, we have
%  \begin{align*}
%   \ez&|\det\nabla^2 H_N(x)|=[-4D''(0)]^{N/2} \sqrt{\frac{N}\pi} \Big(\frac{N+1}{N}\Big)^{\frac{(N+2)(N+1)}{4}} \frac{e^{-N m^2}Z_{N+1}}{Z_N} \\
%   &\qquad \ \sum_{k=1}^{N+1} \int_{\Lambda_{N+1}} e^{-\frac12(N+1)y_k^2+2\sqrt{N(N+1)}m y_{k}} \frac1{Z_{N+1}}\prod_{i=1}^{N+1} e^{-\frac{(N+1)y_i^2}{2}} \prod_{1\le i<j\le N+1}|y_i-y_j| \dd y_1\cdots \dd y_{N+1}.
% \end{align*}
% Since $Z_N=\frac1{N!}(2\sqrt{2})^N N^{-N(N+1)/4}\prod_{i=1}^N\Ga(1+i/2)$, we have
% \begin{align*}
% \frac{Z_{N+1}}{Z_N}&=\sqrt2 \Big(\frac{N}{N+1}\Big)^{(N+1)(N+2)/4} N^{-(N+1)/2}\Ga((N+1)/2),\\
 \begin{align*}
  \ez|\det\nabla^2 H_N(x)|&= \frac{\sqrt{2}[-4D''(0)]^{N/2} \Ga(\frac{N+1}2)  (N+1) }{\sqrt{\pi}N^{N/2} e^{Nm^2 }} \mathbb E \int e^{-\frac12(N+1)w^2+2\sqrt{N(N+1)} m w}  L_{N+1}(\dd w).
\end{align*}
% where $\rho_N(x)$ is the normalized one-point correlation function of GOE such that $\int_\rz f(x) \rho_N(x)\dd x=\frac1N \ez_{\GOE(N)} \sum_{i=1}^N f(\la_i)$ for a bounded continuous function $f$. Note that $\ez|\det\nabla^2 H_N(x)|$ does not dependent on $x$ and

From \pref{le:repl}, we see that for the asymptotic analysis we can replace the above $\sqrt{N(N+1)}$ in the exponent by $N+1$, leaving us to compute asymptotics of

\begin{align*}
 I_N= \mathbb E \int e^{(N+1) \phi(x)}   L_{N+1}(\dd x),
\end{align*}
where
$$\phi(x)=-\frac12x^2-\frac{\mu x}{\sqrt{-D''(0)} }.$$
This is obtained in the following Lemma.

\begin{lemma}\label{dasodaskpow}
 If $|\mu|>\sqrt{-2D''(0)}$ then
\[
\lim_{N\to\8} \frac1N\log I_N =\frac{\mu^2}{-4D''(0)} +\log\frac{|\mu|}{\sqrt{-2D''(0)}}+\frac12,
\]
while if $|\mu| \le \sqrt{-2 D''(0)}$ we have
\[
\lim_{N\to\8}\frac1N \log I_N = \frac{\mu^2}{-2D''(0)}.
\]
\end{lemma}

Assuming the above Lemma, we note that
\begin{align*}
  \int_{B_N}p_{\nabla H_N(x)}(0) \dd x =
  %&= \frac1{\mu^N} \int_{\mu B_N/\sqrt{D'(0)})} \frac1{(2\pi )^{N/2}} e^{-\frac{\|y\|^2}{2}} \dd y \\
  \begin{cases}
   \frac1{|\mu|^N} \pz(  z_N \in  |\mu| B_N/\sqrt{D'(0)} ),& \mu\neq 0,\\
   \frac1{(2\pi)^{N/2}D'(0)^{N/2}} |B_N| , & \mu=0,
  \end{cases}
\end{align*}
where $|B_N|$ is the Lebesgue measure of $B_N$ and $z_N$ is a standard $N$ dimensional Gaussian vector. It follows from \eqref{eq:KRSimple} that
\begin{align*}
\lim_{N\to \8} \frac{1}{N} \log \ez \Crt_N(\rz,B_N)&= \lim_{N\to\8} \frac{1}{N} \bigg(\log C_N + \log I_N\bigg),
\end{align*}
where
\begin{align}\label{eq:stir1}
  C_N = \begin{cases}
  \frac{\sqrt{2}[-4D''(0)]^{N/2} \Ga(\frac{N+1}2)  (N+1) }{\sqrt{\pi}N^{N/2} e^{Nm^2 }|\mu|^N }  \pz(  z_N \in  |\mu| B_N/\sqrt{D'(0)} ), & \mu\neq 0,\\
 \frac{\sqrt{2}[-4D''(0)]^{N/2} \Ga(\frac{N+1}2)  (N+1)|B_N| }{\sqrt{\pi}N^{N/2} (2\pi)^{N/2}D'(0)^{N/2}} , & \mu=0.
  \end{cases}
\end{align}
% From the concentration property of Gaussian measure, we know that for $\eps>0$
% \[
% \lim_{N\to \8} \frac1N \log \pz((1-\eps)\sqrt N \le \|z_N\| \le (1+\eps)\sqrt N) =0.
% \]
From Assumption III and Stirling's formula,
\begin{align*}
\lim_{N\to\8} \frac1N \log C_N = \begin{cases}
\log\frac{\sqrt{-2D''(0)}}{|\mu|} +\frac{\mu^2}{4D''(0)}- \frac12 -\Xi, & \mu\neq 0,\\
\log\sqrt{-2D''(0)}-\frac12-\frac12\log(2\pi)- \frac12\log[D'(0)]+\Theta, & \mu=0.
\end{cases}
\end{align*}
The above computation combined with Lemma \ref{dasodaskpow} finishes the proof of the Theorem.
\end{proof}

We finish this section with the proof of Lemma \ref{dasodaskpow}.

\begin{proof}[Proof of Lemma \ref{dasodaskpow}]
The proof follows from the large deviation principle for the smallest eigenvalues of GOE. In short, in the latter case, the maximum of $\phi$ is attained in the bulk while in the former case, the smallest eigenvalue contributes to the asymptotics of $I_N$.
We argue the first case $|\mu|>\sqrt{-2D''(0)}$. By symmetry, we only consider $\mu>\sqrt{-2D''(0)}$. Since $\phi(x)$ is bounded from above, by the LDP for $\la_1$ as in \pref{eq:rf1ev} and Varadhan's Lemma,
\begin{align}\label{eq:vrd1}
\sup_{x\in \rz} \phi(x)-J_{1}(x)&\le \liminf_{N\to\8} \frac1{N+1}\log \ez_{\GOE(N+1)} e^{(N+1)\phi(\la_{1})}\notag \\
&\le \limsup_{N\to\8} \frac1{N+1}\log \ez_{\GOE(N+1)} e^{(N+1)\phi(\la_{1})}\le \sup_{x\in \rz} \phi(x)-J_{1}(x).
\end{align}
Note that $\arg\max_x [\phi(x)-J_{1}(x) ]= -\frac{\mu}{\sqrt{-4D''(0)}} -\frac{\sqrt{-D''(0)}}{\mu}<-\sqrt2$. It follows that
\begin{align*}
    \liminf_{N\to \8} \frac1N\log I_N\ge\liminf_{N\to \8} \frac1N\log \frac1{N+1}\ez_{\GOE(N+1)} e^{(N+1)\phi(\la_1)}\\
\ge  \frac{\mu^2}{-4D''(0)}+\log\frac{\mu}{\sqrt{-2D''(0)}}+\frac12.
\end{align*}
On the other hand,
\begin{align*}
I_N\le \ez_{\GOE(N+1)} e^{(N+1)\phi(\la_1)}\indi\{\la_1\ge -\frac{\mu}{\sqrt{-D''(0)}}\}+ e^{(N+1)\phi(-\frac{\mu}{\sqrt{-D''(0)}})} \pz\Big(\la_1< -\frac{\mu}{\sqrt{-D''(0)}}\Big).
\end{align*}
For an upper bound for the first term on the right-hand side, we have by \pref{eq:vrd1},
\begin{align*}
\lim_{N \to\8}& \frac1N  \log \ez_{\GOE(N+1)} e^{(N+1)\phi(\la_1)} \\
&= \phi\Big(-\frac{\mu}{\sqrt{-4D''(0)}} -\frac{\sqrt{-D''(0)}}{\mu}\Big)-J_1\Big(-\frac{\mu}{\sqrt{-4D''(0)}} -\frac{\sqrt{-D''(0)}}{\mu}\Big).
\end{align*}
And for the second term, we find by \pref{eq:rf1ev}
\begin{align*}
&\limsup_{N \to\8} \frac1N \log \Big[e^{(N+1)\phi(-\frac{\mu}{\sqrt{-D''(0)}})} \pz\Big(\la_1< -\frac{\mu}{\sqrt{-D''(0)}}\Big)\Big]\\
&\le \phi\Big(-\frac{\mu}{\sqrt{-D''(0)}}\Big) -J_1\Big(-\frac{\mu}{\sqrt{-D''(0)}}\Big)\\
&\le \phi\Big(-\frac{\mu}{\sqrt{-4D''(0)}} -\frac{\sqrt{-D''(0)}}{\mu}\Big)-J_1\Big(-\frac{\mu}{\sqrt{-4D''(0)}} -\frac{\sqrt{-D''(0)}}{\mu}\Big).
\end{align*}
It follows that
\[
\limsup_{N \to\8} \frac1N \log I_N\le  \frac{\mu^2}{-4D''(0)}+\log\frac{\mu}{\sqrt{-2D''(0)}}+\frac12.
\]
We have proved the claim.

For the second case $|\mu| \le \sqrt{-2 D''(0)}$, the maximum of $\phi(x)$ on $[-\sqrt2,\sqrt2]$ is achieved at $x=-\frac{\mu} {\sqrt{-D''(0)}}$. Then for $\eps>0$ and $N$ large enough,
\[
\mathbb E \int_{-\frac{\mu}{\sqrt{-D''(0)}}}^{-\frac{\mu} {\sqrt{-D''(0)}}+\eps} e^{(N+1) \phi(-\frac{\mu} {\sqrt{-D''(0)}}+\eps)} L_{N+1}(\dd x) \le  I_N\le  e^{ (N+1) \phi\Big(-\frac{\mu} {\sqrt{-D''(0)}}\Big)}.
\]
Since $\lim_{N\to\8} \mathbb E L_{N+1}\Big(-\frac{\mu}{\sqrt{-D''(0)}}, -\frac{\mu}{\sqrt{-D''(0)}}+\eps\Big) >0$, it follows that
\begin{align*}
\frac{\mu^2}{-2D''(0)} -\frac{\eps^2}{2}\le \liminf_{N\to\8} \frac1N \log I_N \le \limsup_{N\to\8}\frac1N\log I_N \le \frac{\mu^2}{-2D''(0)}.
\end{align*}
The claim follows by sending $\eps\to 0+$.
\end{proof}

\section{Conditional law of $\nabla^2 H_N$ with constrained critical values}\label{se:3}

In this section, we provide the initial steps for computing complexity functions. Our main result is a relation between a conditional Hessian $\nabla^{2} H_{N}$ and the GOE given in \pref{eq:ayg} which implies \eqref{eq:martin} in the Kac--Rice representation for structure functions $D$ that satisfy Assumptions I, II and IV.

Recall the Kac--Rice formula \pref{eq:startingpoint}.
%  that for a Borel set $E\in \rz$, we have
% \begin{align}\label{eq:startingpoint2}
%   \ez \Crt_{N}(E,B_N) =  &\int_{B_N} \ez[|\det \nabla^2 H_N(x)| \indi\{\frac1N H_N(x)\in E\} | \nabla H_N(x)=0] p_{\nabla H_N(x)}(0)\dd x.
% \end{align}
Note that $(H_N(x),\partial_i H_N(x), \partial_{kl} H_N(x))_{1\le i\le N, 1\le k\le l\le N}$ is a Gaussian field. From \pref{le:cov}, we have $ \Var(H_N(x))  = ND(\frac1N \|x\|^2)$ and the means
\begin{align*}
  \ez(H_N(x))&=\frac{\mu}2\|x\|^2,\ \
  \ez(\nabla H_N(x))=\mu x, \ \  \ez(\nabla^2 H_N(x)) = \mu I_N.
\end{align*}
Let $\Sigma_{01}= \Cov(H_N(x), \nabla H_N(x))= D'(\frac{\|x\|^2}N)x^\mathsf T$ and $\Sigma_{11}=\Cov(\nabla H_N(x))= D'(0) I_N$. By the conditional distribution of Gaussian vectors, we know $$Y:=\frac1N[H_N(x)-\Sigma_{01}\Sigma_{11}^{-1}\nabla H_N(x)]
 = \frac{H_N(x)}N- \frac{D'(\frac{\|x\|^2}N)\sum_{i=1}^{N} x_i \partial_i H_N(x)}{N D'(0)}$$
is independent from $\nabla H_N(x)$. Since $\nabla H_N(x)$ is independent from $\nabla^2H_N(x)$, by conditioning, we may rewrite \pref{eq:startingpoint} as
\begin{align}\label{eq:kr1}
  &\ \ \ez \Crt_{N}(E,B_N)\notag\\
  &=\int_{B_N} \ez[|\det \nabla^2 H_N(x)| \indi{\{ Y+ \frac1N\Sigma_{01}\Sigma_{11}^{-1}\nabla H_N(x) \in  E \}}| \nabla H_N(x)=0]  p_{\nabla H_N(x)}(0)\dd x \notag\\
  &=\int_{B_N} \ez[|\det \nabla^2 H_N(x)| \indi\{Y\in E\} ] p_{\nabla H_N(x)}(0)\dd x \notag\\
  &=\int_{B_N}\int_{E} \ez(|\det \nabla^2 H_N(x)| |Y=u) \frac1{\sqrt{2\pi}\si_Y} e^{-\frac{(u-m_Y)^2}{2\si_Y^2}} p_{\nabla H_N(x)}(0) \dd u \dd x,
\end{align}
where
\begin{align*}
  m_Y &= \ez(Y)=\frac{\mu\|x\|^2}{2N}-\frac{\mu D'(\frac{\|x\|^2}{N}) \|x\|^2}{D'(0) N}, \\
  \si_Y^2 & =\Var(Y)=\frac1N\Big(D(\frac{\|x\|^2}{N})-\frac{D'(\frac{\|x\|^2}N)^2}{D'(0)} \frac{\|x\|^2}{N}\Big).
\end{align*}
To proceed, we need the conditional distribution of $\nabla^2 H_N(x)$ given $Y=u$. A crucial difficulty arises here, however. Namely, one can check that the off-diagonal entries of $\nabla^2 H_N(x)$ given $Y=u$ may have negative covariance, for example,
\[
\Cov[(\partial_{ij} H_N(x), \partial_{kl} H_N(x))|Y=u]=  - \frac1N \frac{\al x_i x_j}{N} \frac{\al x_k x_l}{N}, %\frac{(D''(\|x\|^2/N)x_ix_j/N) (D''(\|x\|^2/N)x_k x_l/N)}{ND(\frac1N \|x\|^2)-\frac1{D'(0)}D'(\frac{\|x\|^2}N)^2\|x\|^2},
\ i\neq j,  k\neq l, \{i,j \}\neq\{k,l\},
\]
for some $\al$ defined below, which prevents using GOE directly.

To overcome this difficulty, let us define
\begin{align}\label{eq:albt0}
   \alpha=\al(\|x\|^2/N) &= \frac{2D''(\|x\|^2/N)}{ \sqrt{ D(\frac{\|x\|^2}N)-\frac{D'({\|x\|^2}/N)^2}{D'(0)}\frac{\|x\|^2}N}}, \notag\\
   \beta=\bt(\|x\|^2/N) & =\frac{D'(\|x\|^2/N)-D'(0)}{\sqrt{ D(\frac{\|x\|^2}N)-\frac{D'({\|x\|^2}/N)^2}{D'(0)}\frac{\|x\|^2}N}}.
\end{align}
Note that $\al\le0$ and $\beta\le 0$. One should think of $\al$ and $\bt$ as $O(1)$ quantities.
Let us define $A=A_N=U(x)\nabla^2H_N(x) U(x)^\mathsf T$ where $U(x)$ is an $N\times N$ orthogonal matrix such that
\begin{align}\label{eq:umat}
U( \frac{\al x x^\mathsf T}N+\bt I_N )U^\mathsf T =
\begin{pmatrix}
 \frac{\al\|x\|^2}N +\bt &0  &\cdots  &0  \\
 0& \bt &  \cdots&0  \\
 \vdots&  \vdots& \ddots& \vdots  \\
 0&  0&  \cdots& \bt
 \end{pmatrix}.
\end{align}
In other words, we have for $U=(u_{ij})$,
\begin{align}\label{}
  \sum_{k,l} u_{ik}( \frac{\al x_k x_l}N +\bt \de_{kl}) u_{jl}&  =\al\de_{i1}\de_{j1} \frac{\|x\|^2}N +\bt\de_{ij}.
%  \sum_{k,l} u_{ik}(\al x_k x_l +\bt \de_{kl}) u_{il} &=\bt, \ \ i\neq 1, \\
%  \sum_{k,l} u_{ik}(\al x_k x_l +\bt \de_{kl}) u_{jl} & =0, \ \ i\neq j.
\end{align}
Indeed, such a $U(x)$ can be found by imposing the first row to be $\frac{x^\sfT}{\|x\|}$ for $x\neq 0$; and if $x=0$, $U(x)$ can be arbitrary orthogonal matrix. It follows that $\ez(A)=\mu I_N$, and by \pref{le:cov},
\begin{align*}
  \Cov(A_{ij}, A_{i'j'})&=\sum_{k,l,k',l'} u_{ik}u_{jl}u_{i'k'} u_{j'l'} \Cov(\partial_{kl} H_N(x),\partial_{k'l'} H_N(x))\\
  &=\frac{-2D''(0)}{N} (\de_{ij}\de_{i'j'}+\de_{ii'}\de_{jj'} +\de_{ij'}\de_{i'j}),\\
  \Cov(A_{ij}, \partial_l H_N(x)) &= \sum_{a,b} u_{ia} u_{jb}\Cov(\partial_{ab} H_N(x), \partial_l H_N(x))=0,\\
  \Cov(A_{ij}, H_N(x))& =\sum_{a, b} u_{ia}u_{jb} (\frac{2D''(\|x\|^2/N)x_a x_b}N +[D'(\|x\|^2/N)-D'(0)]\de_{ab})\\
  &=\frac{2D''(\|x\|^2/N) \de_{i1}\de_{j1} \|x\|^2}N +[D'(\|x\|^2/N) -D'(0)] \de_{ij}.
\end{align*}
Since $A$ and $\nabla^2 H_N(x)$ have the same eigenvalues, by \pref{eq:kr1},
\begin{align}\label{eq:kr2}
 \ez \Crt_{N}(E,B_N)=\int_{B_N}\int_{E} \ez(|\det A| |Y=u) \frac1{\sqrt{2\pi}\si_Y} e^{-\frac{(u-m_Y)^2}{2\si_Y^2}} p_{\nabla H_N(x)}(0) \dd u \dd x.
\end{align}
We need the conditional distribution of $A$ given $Y=u$. Note that
\[
\Cov(A_{ij}, Y)= \Cov(A_{ij},\frac{H_N}N) =\frac{2D''(\|x\|^2/N) \de_{i1}\de_{j1} \|x\|^2}{N^2 } +\frac{[D'(\|x\|^2/N) -D'(0)] \de_{ij}}{N}.
\]
Then conditioning on $Y=u$ we have
\begin{align}
  &\ez(A_{ij}|Y=u) = \ez(A_{ij})+\Cov(A_{ij},Y)\si_Y^{-2}(u-\ez(Y)) \nonumber \\
  & = \mu\de_{ij} + \frac{(\frac{2D''(\frac{\|x\|^2}N) \de_{i1}\de_{j1} \|x\|^2}N +[D'(\frac{\|x\|^2}N) -D'(0)] \de_{ij}) (u -\frac{\mu\|x\|^2}{2N} +\frac{\mu D'(\frac{\|x\|^2}N)\|x\|^2}{D'(0)N})}{ D(\frac{\|x\|^2}{N}) -\frac{D'(\frac{\|x\|^2}{N})^2 \|x\|^2}{D'(0) N}} , \nonumber \\
& m_{A|u}:=\ez(A|Y=u) =\mu I_N + \frac{u-\frac{\mu\|x\|^2}{2N} +\frac{\mu D'(\frac{\|x\|^2}N)\|x\|^2}{D'(0)N}}{D(\frac{\|x\|^2}{N}) -\frac{D'(\frac{\|x\|^2}{N})^2 \|x\|^2}{D'(0) N}}\nonumber \\
  & \times \begin{pmatrix} \
     \frac{2D''(\frac{\|x\|^2}N)\|x\|^2}{N} +D'(\frac{\|x\|^2}N) -D'(0)  & 0 \label{mean:A} \\
     0 &  [D'(\frac{\|x\|^2}N) -D'(0) ]I_{N-1}
   \end{pmatrix}, \\
 & \Cov[(A_{ij}, A_{i'j'})^\mathsf T |Y=u]=\Cov[(A_{ij}, A_{i'j'})^\mathsf T]-\Cov[(A_{ij}, A_{i'j'})^\mathsf T ,Y]\si_Y^{-2}\Cov[Y,(A_{ij}, A_{i'j'})^\mathsf T]  \nonumber \\
 &= \Cov[(A_{ij}, A_{i'j'})^\mathsf T] - \frac1N\nonumber \\
 & \times\begin{pmatrix}
          (\frac{\al \de_{i1}\de_{j1} \|x\|^2}{N}+\bt\de_{ij})^2 & (\frac{\al \de_{i1}\de_{j1} \|x\|^2}{N}+\bt\de_{ij}) (\frac{\al \de_{i'1}\de_{j'1} \|x\|^2}{N}+\bt\de_{i'j'} )\\
           (\frac{\al \de_{i1}\de_{j1} \|x\|^2}{N}+\bt\de_{ij}) (\frac{\al \de_{i'1}\de_{j'1} \|x\|^2}{N}+\bt\de_{i'j'} ) & (\frac{\al \de_{i'1}\de_{j'1} \|x\|^2}{N}+\bt\de_{i'j'})^2
        \end{pmatrix},\nonumber
\end{align}
where $\Cov[(A_{ij}, A_{i'j'})^\mathsf T] $ denotes the $2\times 2$ covariance matrix of $A_{ij}$ and $A_{i'j'}$ while $\Cov[(A_{ij}, A_{i'j'})^\mathsf T ,Y] $ denotes the $2\times 1$ covariance matrix of $(A_{ij}, A_{i'j'})^\mathsf T$ and $Y$. From here we see conditioning on $Y=u$,
\begin{align*}
  &\Cov[(A_{ij}, A_{i'j'})|Y=u] \\
  &=\frac{-2D''(0)(\de_{ij}\de_{i'j'}+\de_{ii'}\de_{jj'} +\de_{ij'}\de_{i'j})}{N}-\frac1N(\frac{\al \de_{i1}\de_{j1} \|x\|^2}{N}+\bt\de_{ij}) (\frac{\al \de_{i'1}\de_{j'1} \|x\|^2}{N}+\bt\de_{i'j'} )\\
  &=\begin{cases}
     \frac{-6D''(0)}{N}-\frac1N (\frac{\al \|x\|^2}N+\bt)^2, & i=j=i'=j'=1, \\
     \frac{-2D''(0)}{N} - \frac1N(\frac{\al \|x\|^2}N+\bt)\bt, & i=j=1\neq i'=j', \mbox{ or } i'=j'=1\neq i=j,\\
    % -(\bt+\al\|x\|^2)\al \|x\|^2, &  i=i'=j=1, j'\neq 1\\
     %\frac{-2D''(0)}{N}-\al^2\|x\|^4, & i=i'=1, j=j'\neq 1\\
     %-\al^2\|x\|^4, &  i=i'=1, j\neq j', j\neq 1, j'\neq 1\\
     %-\al \|x\|^2 \bt , & i=1\neq j,i'=j'\neq 1\\
     \frac{-6D''(0)}{N} -\frac{ \bt^2}N,& i=j=i'=j'\neq 1,\\
     \frac{-2D''(0)}{N} -\frac{\bt^2}N , & 1\neq i=j\neq i'=j'\neq 1,\\
   \frac{-2D''(0)}{N} ,& i=i'\neq j=j', \mbox{ or } i=j'\neq j=i', \\
   0, &\text{otherwise}.
   \end{cases}
\end{align*}

Alternatively, one can find the above conditional covariances using spherical coordinates, which could avoid the matrix function $U(x)$. In order to draw connection with GOE, we first have to check that all the quantities above are positive. Note that $\al$ and $\bt$ depend on $\|x\|^2$ and $N$ through $\|x\|^2/N$. Let us write $\rho=\rho_{N}(x)=\frac{\|x\|}{\sqrt N}$ so that $\al=\al(\rho^2)$ and $\bt=\bt(\rho^2)$.

\begin{lemma}\label{le:albtd}
  We have $\lim_{\rho \to0+} \frac{D(\rho^2)}{\rho^4}-\frac{D'(\rho^2)^2}{D'(0)\rho^2} =-\frac32D''(0)$ and
  \begin{align*}
    \lim_{\rho\to0+} \bt(\rho^2)^2&=-\frac23 D''(0),\quad
    \lim_{\rho\to 0+} \al(\rho^2) \bt(\rho^2)\rho^2  = -\frac43 D''(0), \quad
    \lim_{\rho\to 0+} [\al(\rho^2) \rho^2]^2  = -\frac83 D''(0).
  \end{align*}
\end{lemma}
\begin{proof}
  Using l'Hospital's rule together with $D(0)=0$,
  \begin{align*}
    \lim_{\rho\to 0+} \frac{D(\rho^2)}{\rho^4}-\frac{D'(\rho^2)^2}{D'(0)\rho^2} & =\lim_{\rho \to 0+} \frac{D'(\rho^2)\rho^2-D(\rho^2)}{\rho^4}-\frac{2D'(\rho^2)D''(\rho^2)}{D'(0)} = -\frac32D''(0).
  \end{align*}
  It follows that
  \begin{align*}
  \lim_{\rho \to 0+}& \bt(\rho^2)^2= \lim_{\rho\to 0+} \frac{ [\frac{D'(\rho^2)-D'(0)}\rho^2]^2}{\frac{D(\rho^2)}{\rho^4} -\frac{D'(\rho^2)^2}{D'(0)\rho^2}}= -\frac23 D''(0),\\
  \lim_{\rho\to 0+}& \al(\rho^2)\bt(\rho^2) \rho^2 =\lim_{\rho\to0+}\frac{[2D''(\rho^2)] \frac{D'(\rho^2)-D'(0)}{\rho^2}} {\frac{D(\rho^2)}{\rho^4}-\frac{D'(\rho^2)^2}{D'(0)\rho^2}} =-\frac43 D''(0),\\
  \lim_{\rho\to 0+}&[\al(\rho^2)\rho^2]^2 = \lim_{\rho\to0+}\frac{[2D''(\rho^2)]^2} {\frac{D(\rho^2)}{\rho^4}-\frac{D'(\rho^2)^2}{D'(0)\rho^2}}=-\frac83 D''(0).\qedhere
  \end{align*}
\end{proof}

In light of \pref{le:albtd}, we make the following observation. Following \cite{SSV}*{Theorem 8.2}, a function $f: (0,\8)\to(0,\8)$ is a Thorin--Bernstein function if and only if $\lim_{x\to0+}f(x)$ exists and its derivative has a representation
\begin{align}\label{eq:tbfcn}
f'(x)=\frac{a}{x}+b +\int_{(0,\8)} \frac1{x+t}\si (\dd t),
\end{align}
where $a,b\ge0$ and $\si$ is a measure on $(0,\8)$ satisfying $\int_{(0,\8)} \frac1{1+t}\si(\dd t)<\8$. In particular, the functions $D(r)=\log(1+r/\eps)$ and $D(r)=(r+\eps)^\ga-\eps^{\ga}$ are Thorin--Bernstein functions. Recall the definitions of $\al$ and $\bt$ as in \pref{eq:albt0}. The proof of the following analytical result is deferred to Appendix \pref{se:aux}.
\begin{lemma}\label{le:dgeab}
For any $x\in \rz^N\setminus\{0\}$, we have
\begin{align}
  -2D''(0)&> \left(\frac{\al\|x\|^2}{N}+\bt\right)\bt,\label{eq:asmp1} \\
  -4D''(0)&> \left(\frac{\al\|x\|^2}{N}+\bt\right)\frac{\al \|x\|^2}N,\label{eq:asmp2}
  %\frac{\al\bt\|x\|^2}{N} & >0
\end{align}
provided anyone of the following conditions holds:
\begin{enumerate}
  \item For all $x\neq 0$,
  \begin{align}\label{eq:btbd}
\bt^2\le -\frac23 D''(0).
\end{align}
  \item For all $y\ge0$,
  \begin{align}\label{eq:btinc}
2D'(0)D''(y)[D(y)-D'(y)y]+D'(y)[D'(y)-D'(0)]^2 \ge 0.
\end{align}
\item For all $y\ge0$
\begin{align}\label{eq:btbd2}
\frac{D'(y) y}{D'(0)}-\frac{D'(y)-D'(0)}{D''(0)}\ge 0.
\end{align}
\item For all $y\ge 0$,
\begin{align}\label{eq:btbd3}
-\frac{D'(y)}{D''(y)} +\frac{D'(0)}{D''(0)}\ge y.
\end{align}
\item   For all $y\ge 0$,
\begin{align}\label{eq:btbd4}
\frac{-D''(y)^2+D'''(y)D'(y)}{D''(y)^2}\ge 1.
\end{align}
\item $D$ is a Thorin--Bernstein function with $a=0$ in \pref{eq:tbfcn}.
\end{enumerate}
\end{lemma}

From now on, we always assume \pref{eq:asmp1} and \pref{eq:asmp2}, thus $\Cov[(A_{ij}, A_{i'j'})|Y=u]\ge0$ for all $i,i',j,j'$. Recalling \eqref{mean:A}, let us write
\begin{align}
  m_1 & =m_1(\rho,u)= \mu + \frac{(u-\frac{\mu\rho^2}{2} +\frac{\mu D'(\rho^2)\rho^2}{D'(0)}) ( 2D''(\rho^2)\rho^2+D'(\rho^2) -D'(0) )}{D(\rho^2) -\frac{D'(\rho^2)^2 \rho^2}{D'(0) }}, \notag \\
  m_2&=m_2(\rho,u)= \mu + \frac{(u-\frac{\mu \rho^2}{2} +\frac{\mu D'(\rho^2) \rho^2}{D'(0)}) (D'(\rho^2) -D'(0) )}{D( \rho^2) -\frac{D'(\rho^2)^2 \rho^2}{D'(0) }}, \notag\\
\si_1 & =\si_1(\rho)= \sqrt{\frac{-4D''(0)-(\al \rho^2 +\bt)\al\rho^2}{N}}, \ \ \
  \si_2  =\si_2(\rho)= \sqrt{\frac{-2D''(0)-(\al\rho^2 +\bt)\bt}{N}}, \notag \\
  m_Y&=m_Y(\rho)= \frac{\mu\rho^2}{2}-\frac{\mu D'(\rho^2) \rho^2}{D'(0) }, \ \ \ \si_Y =\si_Y(\rho) =\sqrt{\frac1N\Big(D(\rho^2)-\frac{D'(\rho^2)^2\rho^2}{D'(0)} \Big)},\notag \\
  \alpha &=\alpha(\rho^2)= \frac{2D''(\rho^2)}{ \sqrt{ D(\rho^2)-\frac{D'(\rho^2)^2 \rho^2}{D'(0)}}},  \ \ \
   \beta =\beta(\rho^2)=\frac{D'(\rho^2 )-D'(0)}{\sqrt{ D(\rho^2 )-\frac{D'(\rho^2)^2 \rho^2}{D'(0)}}}, \label{eq:msialbt}
\end{align}
where $\rho=\frac{\|x\|}{\sqrt N}$. From time to time, we also use the following change of variable
\begin{align}\label{eq:uvcov}
v=\frac{u-\frac{\mu\rho^2}{2}+\frac{\mu D'(\rho^2)\rho^2}{D'(0)}}{\sqrt{D(\rho^2)-\frac{D'(\rho^2)^2\rho^2}{D'(0)}}} = \frac{u-m_Y}{\sqrt{N}\si_Y}
\end{align}
so that
\begin{align}
m_1&=\mu +v(\al \rho^2+\bt), \ \
m_2=\mu+ v\bt.\label{eq:m12cov}
\end{align}
Let 
\begin{align*}
  G= G(u) =
  \begin{pmatrix}
    z_1'& \xi^\mathsf T \\
     \xi & \sqrt{-4D''(0)} (\sqrt{\frac{N-1}{N}}\GOE_{N-1}-z_3'I_{N-1})
  \end{pmatrix} ,
\end{align*}
where with $z_1,z_2,z_3$ being independent standard Gaussian random variables,
\begin{align*}
  z_1'&=\si_1 z_1 - \si_2  z_2 + m_1, \quad
   z_3'=\frac1{\sqrt{-4D''(0)}}\Big(\si_2 z_2+ \frac{  \sqrt{\al\bt}\rho }{\sqrt N} z_3 - m_2\Big),
\end{align*}
and $\xi$ is a centered column Gaussian vector with covariance matrix $\frac{-2D''(0)}{N}I_{N-1}$ which is independent from $z_1,z_2,z_3$ and the GOE matrix $\GOE_{N-1}$. 
The above discussion yields our main result of this section.
\begin{proposition}
  Assume Assumptions I, II and IV. Then we have in distribution
\begin{align}\label{eq:ayg}
  (U\nabla^2 H_N U^\sfT | Y=u)  & \stackrel{d}{=}
  G.
\end{align}
\end{proposition}
In the following we write frequently
$$G_{**}=\sqrt{-4D''(0)} \Big(\sqrt{\frac{N-1}{N}}\GOE_{N-1}-z_3'I_{N-1}\Big).$$
To connect with \pref{eq:kr2}, we have
\begin{equation}\label{eq:martin}
\ez(|\det A||Y=u) = \int |\det a| p_{A|Y}(a|u) \dd a = \ez(|\det G|).
\end{equation}

\section{Exponential tightness}\label{se:exptt}
The purpose of this section is to prove several exponential tightness results so that our future analysis will be reduced to the compact setting. Let $E \subset \rz$ be a Borel set. Hereafter, for simplicity, let us assume $B_N$ is a shell $B_{N}(R_{1},R_{2})=\{ x\in \mathbb R^{N}: R_{1}< \frac{\| x\|}{\sqrt N} < R_{2} \}$, $0\le R_1<R_2\le \8$. Recall that in this case we write
$\Crt_{N}(E, (R_1, R_2))= \Crt_{N}(E,B_N(R_1,R_2)).$
 Using spherical coordinates and writing $\rho=\frac{\|x\|}{\sqrt N}$, by the Kac--Rice formula we have
\begin{align}
&\ez\Crt_{N}(E,(R_1,R_2))=\int_{B_N} \int_E \ez[|\det A| |Y=u] \frac1{\sqrt{2\pi}\si_Y} e^{-\frac{(u-m_Y)^2}{2\si_Y^2}} p_{\nabla H_N(x)}(0) \dd u \dd x \notag\\
&=S_{N-1} N^{(N-1)/2} \int_{R_1}^{R_2}\int_E \ez[|\det G| ] \frac1{\sqrt{2\pi}\si_Y} e^{-\frac{(u -m_Y)^2}{2\si_Y^2}}  \frac1{(2\pi)^{N/2} D'(0)^{N/2}} e^{-\frac{N \mu^2 \rho^2}{2D'(0)}} \rho^{N-1}   \dd u  \dd\rho.\label{eq:krerr}
\end{align}
Here  $S_{N-1} = \frac{2\pi^{N/2}}{\Gamma(N/2)}$ is the area of $N-1$ dimensional unit sphere, $G$ depends on $u$ implicitly.
Using the Stirling formula, we have
\begin{align}\label{eq:snlim}
\lim_{N\to\8} \frac1N \log (S_{N-1}N^{\frac{N-1}{2}})= \frac12\log(2\pi)+\frac12.
\end{align}

Recall the representation \pref{eq:ayg}. Let $\la_1\le \cdots \le \la_{N-1}$ be the eigenvalues of  $\GOE_{N-1}$.  The eigenvalues of $G_{**}$ can be represented as $\{\sqrt{-4D''(0)}((\frac{N-1}{N})^{1/2}\la_i-z_3')\}_{i=1}^{N-1}$. By the representation, we may find a random orthogonal matrix $V$ which is independent of the unordered eigenvalues $\tilde \la_j, j=1,...,N-1$ and $z_3'$, such that
\begin{align}\label{eq:goedc}
G_{**} = \sqrt{-4D''(0)} V^\mathsf{T} \begin{pmatrix}
  (\frac{N-1}{N})^{1/2}\tilde \la_1-z_3' &\cdots  &0  \\
\vdots& \ddots& \vdots  \\
0&  \cdots& (\frac{N-1}{N})^{1/2}\tilde\la_{N-1}-z_3'
\end{pmatrix} V.
\end{align}
By the rotational invariance of Gaussian measures, $V \xi$ is a centered Gaussian vector with covariance matrix $\frac{-2D''(0)}N I_{N-1}$ that is independent of $z_3'$ and $\tilde\la_j$'s. We can rewrite $V \xi \stackrel{d}{=}\sqrt{\frac{-2D''(0)}{N}} Z$, where $Z=(Z_1,..., Z_{N-1})$ is an $N-1$ dimensional standard Gaussian random vector.
%Since each row of $V=(V_{ij})_{i,j=1}^{N-1}$ is a random vector with uniform distribution on the unit sphere, it is well known that each component $V_{ij}$ has density
%\[
%f_{V_{ij}}(x)=\frac{\Gamma(\frac{N-1}{2})}{\sqrt\pi \Gamma(\frac{N-2}{2})}(1-x^2)^{\frac{N-4}{2}}.
%\]
%It follows that $\ez(V_{ij}^2)=\frac1{N-1}$ and by independence
%$$\ez({\xi'_{i}}^2)=\ez(\sum_{j,k=1}^{N-1} V_{ij}V_{ik} \xi_j\xi_k)=\frac{-2D''(0)}{N}.$$
Using the determinant formula for block matrices or the Schur complement formula,
\begin{align}
    \det G= \det (G_{**})(z_1' - \xi^\sfT G_{**}^{-1} \xi )& = [-4D''(0)]^{(N-1)/2} z_1'\prod_{j=1}^{N-1} ((\frac{N-1}{N})^{1/2}\la_j -z_3') \notag\\
&\ \ \ - \frac{[-4D''(0)]^{N/2}}{2N} \sum_{k=1}^{N-1} Z_k^2 \prod_{j\neq k}^{N-1} ((\frac{N-1}{N})^{1/2}\la_j -z_3').\label{eq:schur}
\end{align}
It follows from \pref{eq:krerr} that
\begin{align}
  &\ez\Crt_{N} (E,(R_1,R_2))
  = S_{N-1}N^{(N-1)/2}\int_{R_1}^{R_2} \int_{E} \ez\Big(\Big| [-4D''(0)]^{(N-1)/2} z_1'  \prod_{j=1}^{N-1} ((\frac{N-1}{N})^{1/2}\la_j-z_3') \notag\\
  &\ \ -  \frac{[-4D''(0)]^{N/2}}{2N} \sum_{k=1}^{N-1} Z_k^2 \prod_{j\neq k}^{N-1} ((\frac{N-1}{N})^{1/2}\la_j -z_3')\Big|\Big) \frac{e^{-\frac{(u-m_Y)^2}{2\si_Y^2}} }{\sqrt{2\pi}\si_Y} \frac{e^{-\frac{N \mu^2 \rho^2}{2D'(0)}} }{(2\pi)^{N/2} D'(0)^{N/2}}  \rho^{N-1}  \dd u  \dd\rho \notag \\
  &\le  S_{N-1}N^{(N-1)/2} [I_1(E, (R_1,R_2)) +I_2(E, (R_1,R_2))], \label{eq:ecnr12}
\end{align}
where
\begin{align}
 I_1(E, (R_1,R_2)) &= [-4D''(0)]^{\frac{N-1}{2}} \int_{R_1}^{R_2} \int_{E} \ez\Big[|z_1'| \prod_{i=1}^{N-1} |(\frac{N-1}{N})^{1/2}\la_i-z_3'| \Big]\notag\\
  &\frac{ e^{-\frac{(u-m_Y)^2}{2\si_Y^2}}}{\sqrt{2\pi}\si_Y} \frac{e^{-\frac{N \mu^2 \rho^2}{2D'(0)}}}{(2\pi)^{N/2} D'(0)^{N/2}}  \rho^{N-1}   \dd u  \dd\rho,\notag \\
 I_2(E,(R_1,R_2)) &= \frac{[-4D''(0)]^{\frac{N}{2}} }{2N} \sum_{i=1}^{N-1}\int_{R_1}^{R_2} \int_{E} \ez\Big[Z_i^2 \prod_{j\neq i} |(\frac{N-1}{N})^{1/2}\la_j-z_3'| \Big ] \notag \\
  &\frac{ e^{-\frac{(u-m_Y)^2}{2\si_Y^2}}}{\sqrt{2\pi}\si_Y} \frac{e^{-\frac{N \mu^2 \rho^2}{2D'(0)}}}{(2\pi)^{N/2} D'(0)^{N/2}}  \rho^{N-1}   \dd u  \dd\rho.\label{eq:ie12}
\end{align}

In the following we will employ hard analysis to derive various estimates that would reduce the problem to the compact setting.

\begin{lemma}\label{le:apest}
	For any $\rho>0$, $u\in \rz$, we have
	\begin{align}
		\frac{1}{D'(0)-D'(\rho^2)}&\le \frac{C_D (1+\rho^2)}{\rho^2}, \ \ \		\frac1{\sqrt{D(\rho^2)-\frac{D'(\rho^2)^2 \rho^2}{D'(0)} }} \le \frac{C_D (1+\rho^2)}{\rho^2}, \notag\\
		|m_i| &\le |\mu| + C_D \Big|\frac{u}{\rho^2}-\frac\mu2+\frac{\mu D'(\rho^2)}{D'(0)}  \Big| (1+\rho^2), \ \ i=1,2. \label{eq:mtest}
	\end{align}
\end{lemma}
\begin{proof}
	Since $\lim_{\rho\to0+} \frac{D'(\rho^2)-D'(0)}{\rho^2}=D''(0)$ and $D'(\rho^2)$ is strictly decreasing to 0 as $\rho^2$ tends to $\8$, we have the first assertion.  By \pref{eq:asmp1}, we have
	\[
	\frac1{\sqrt{D(\rho^2)-\frac{D'(\rho^2)^2 \rho^2}{D'(0)} }}\le \frac{\sqrt{-2D''(0)}}{D'(0)-D'(\rho^2)}\le \frac{C_D(1+\rho^2)}{\rho^2} .
	\]
	Using \pref{eq:asmp1} and \pref{eq:asmp2},
	\begin{align*}
	|m_1| & \le |\mu|+ \Big|\frac{u}{\rho^2}-\frac\mu2+\frac{\mu D'(\rho^2)}{D'(0)}\Big|\frac{ C_D \rho^2}{D'(0)-D'(\rho^2)} \le |\mu|+ C_D \Big|\frac{u}{\rho^2}-\frac\mu2+\frac{\mu D'(\rho^2)}{D'(0)}  \Big| (1+\rho^2),\notag \\
	|m_2|&\le |\mu|+ \Big|\frac{u}{\rho^2}-\frac\mu2+\frac{\mu D'(\rho^2)}{D'(0)}\Big| \frac{C_D \rho^2}{D'(0)-D'(\rho^2)}\le |\mu| + C_D \Big|\frac{u}{\rho^2}-\frac\mu2+\frac{\mu D'(\rho^2)}{D'(0)}  \Big| (1+\rho^2).\qedhere
	\end{align*}
\end{proof}

Recall $z_1'=\si_1 z_1 -\si_2 z_2 + m_{1}$, $z_3'=(\si_2 z_2 + \frac{\rho\sqrt{\al \bt}z_3}{\sqrt N}-m_{2}) /\sqrt{-4D''(0)}$. Note that the conditional distribution of $z_1'$ given $z_3'=y$ is given by
\begin{align}\label{eq:z13con0}
z_1'| z_3'=y \sim N \Big (\bar \sfa, \frac{\sfb^2}{N}\Big),
\end{align}
where
\begin{align*}
 \bar \sfa&=  m_1-\frac{\si_2^2(\sqrt{-4D''(0)}y + m_2)}{\si_2^2+\frac{\al \bt \rho^2}{N}}\\% =m_1-\frac{[-2D''(0)-(\al\rho^2 +\bt)\bt](\sqrt{-D''(0)}y +m_2)}{-2D''(0)-\bt^2}\\
 &=\frac{-2D''(0)\al\rho^2( u-\frac{\mu\rho^2}{2}+\frac{\mu D'(\rho^2) \rho^2}{D'(0) } )}{(-2D''(0)-\bt^2) \sqrt{D(\rho^2)-\frac{D'(\rho^2)^2 \rho^2}{D'(0)} }}
   +\frac{\al\bt\rho^2 \mu }{-2D''(0)-\bt^2}\notag \\
   &\quad  -\frac{(-2D''(0)-\bt^2-\al\bt\rho^2)\sqrt{-4D''(0)} y}{-2D''(0)-\bt^2},\\
 \frac{\sfb^2}{N}& =  \si_1^2+\si_2^2-\frac{\si_2^4}{\si_2^2+\frac{\al\bt \rho^2}{N}} =\frac{-4D''(0)}{N} +\frac{2D''(0)\al^2\rho^4}{N(-2D''(0)-\bt^2)}.
\end{align*}

\begin{lemma}\label{le:exptt}
Suppose $\mu\neq 0$. %Let $\Crt_{N} (E,F )=\Crt_{N} (E,B_N) $ when $B_N=\{x\in \rz^N: \frac{\|x\|}{\sqrt N}\in F\}$ where $E$ and $F$ are Borel sets of $\rz$.
Then
\begin{align*}
\limsup_{ T\to \8} \limsup_{N\to \8} \frac1N \log  \ez \Crt_{N} ([-T, T]^c, (0,\8)) &= -\8,\\
\limsup_{ R \to \8} \limsup_{N\to \8} \frac1N \log  \ez \Crt_{N} (\rz, (R,\8)) &= -\8,\\
\limsup_{ \eps \to 0+} \limsup_{N\to \8} \frac1N \log  \ez \Crt_{N} (\rz, (0,\eps)) &= -\8.
\end{align*}
\end{lemma}
\begin{proof}
%Since the argument is similar, we only give a detailed proof for the first claim.
(1) Note that $\sfb^2\le -4D''(0)$ and that
\begin{align}\label{eq:z3mom}
\ez [|z_3'|^{N-1}] \le C^{N-1}\Big[ m_2^{N-1}+\Big(\frac{-2D''(0)-\bt^2}{-4N D''(0)}\Big)^{\frac{N-1}{2}}\Big]\le C^{N-1}(1+m_2^{N-1}).
\end{align}
We write $m_u=|\mu|+ C_D |\frac{u}{\rho^2}-\frac\mu2+\frac{\mu D'(\rho^2)}{D'(0)}  | (1+\rho^2)$. Using the conditional distribution \pref{eq:z13con0}, \pref{eq:absgau}, \pref{eq:goenorm}, \pref{le:apest} and the elementary fact $m_u\le\max\{ 1, m_u^N\}$,
\begin{align*}
& \ez \Big[|z_1'| \prod_{i=1}^{N-1} |(\frac{N-1}{N})^{1/2}\la_i-z_3'| \Big] \\
&=\int_\rz \ez \Big[|z_1'| \prod_{i=1}^{N-1} |(\frac{N-1}{N})^{1/2}\la_i-y| \Big|z_3'=y\Big]  \frac{\sqrt {-4N D''(0)} \exp\{-\frac{N(\sqrt{ -4D''(0)}y+m_2)^2}{2(-2D''(0)-\bt^2)}\}} {\sqrt{2\pi(-2D''(0)-\bt^2)} } \dd y\\
 &\le \int_\rz \Big( \frac{\sqrt 2 \sfb}{\sqrt{\pi N}}+|\bar a| \Big) \ez(\la_{N-1}^*+|y|)^{N-1} \frac{\sqrt {-4N D''(0)} \exp\{-\frac{N(\sqrt{ -4D''(0)}y+m_2)^2}{2(-2D''(0)-\bt^2)}\}} {\sqrt{2\pi(-2D''(0)-\bt^2)} } \dd y\\
 &\le C^{N-1}\ez [ (\sfb+ |m_1|+|m_2| +\sqrt{-4D''(0)}|z_3'|) ({\la_{N-1}^*}^{N-1}+|z_3'|^{N-1})]\\
 &\le C_D^N(1+m_u^{N}),
\end{align*}
where $\la_{N-1}^*$ is the operator norm of $\GOE_{N-1}$. Similarly,
\begin{align*}
\ez\Big[Z_i^2 \prod_{j\neq i, 1\le j\le N-1} |(\frac{N-1}{N})^{1/2}\la_j-z_3'| \Big ] \le \ez (\la_{N-1}^*+|z_3'|)^{N-2}\le  C^N(1+|m_2|^{N-2}).
\end{align*}
Since $D(r)\le D'(0) r$, we have
$
D(\rho^2)-\frac{D'(\rho^2)^2 \rho^2}{D'(0)}\le D'(0)\rho^2.
$
Together with \pref{le:apest}, we obtain after a change of variable $u= \rho^2 s$,
\begin{align*}
& \ez \Crt_{N} ([-T, T]^c, (0,\8)) \\
&\le   C_D^N S_{N-1}\int_{\rz_+}  \int_{[-T,T]^c} (1 +m_u^N)  \frac1{\sqrt{2\pi}\si_Y} e^{-\frac{(u-m_Y)^2}{2\si_Y^2}} \frac1{(2\pi)^{N/2} D'(0)^{N/2}} e^{-\frac{N \mu^2 \rho^2}{2D'(0)}} \rho^{N-1} \dd u  \dd\rho \\
&\le  C_{\mu,D}^N S_{N-1} \int_{\rz_+}  \int_{[-{T/\rho^2},{T/\rho^2}]^c} \Big[1 + (1+\rho^{2N})|s-\frac{\mu}2+\frac{\mu D'(\rho^2)}{D'(0)} |^{N}\Big]   \\
&\ \  \frac{\sqrt N}{\sqrt{2\pi} \sqrt{D(\rho^2)-\frac{D'(\rho^2)^2 \rho^2}{D'(0)} }} \exp\Big(- \frac{N \rho^4 (s - \frac{\mu}2+\frac{\mu D'(\rho^2)}{D'(0)})^2}{2(D(\rho^2)-\frac{D'(\rho^2)^2 \rho^2}{D'(0)})}\Big) \frac1{(2\pi)^{N/2} D'(0)^{N/2}} e^{-\frac{N \mu^2 \rho^2}{2D'(0)}} \rho^{N+1}  \dd s \dd\rho   \\
&\le   \frac{C_{\mu,D}^N S_{N-1} \sqrt N}{ (2\pi)^{\frac{N+1}{2}}D'(0)^{N/2}} \Big(\int_{0}^\8  \int_{\sqrt{T/s}}^\8 + \int_{-\8}^0  \int_{\sqrt{-T/s}}^\8\Big ) [1+(1+\rho^{2N})(|s|+|\mu|)^{N}  ]  \\
& \ \ \ \ \ \frac{ (1+\rho^2)}{ \rho^2}   \exp\Big(-\frac{N[(s - \frac{\mu}2+\frac{\mu D'(\rho^2)}{D'(0)})^2+\mu^2] \rho^2}{2D'(0)}\Big)  \rho^{N+1}   \dd\rho \dd s.
\end{align*}
We need to find a good lower bound for $(s - \frac{\mu}2+\frac{\mu D'(\rho^2)}{D'(0)})^2$. To save space, let
\[
f(s,\rho^2)=[1+(1+\rho^{2N})(|s|+|\mu|)^{N}  ]  (\rho^{N-1}+\rho^{N+1})\exp\Big(-\frac{N[(s - \frac{\mu}2+\frac{\mu D'(\rho^2)}{D'(0)})^2+\mu^2] \rho^2}{2D'(0)}\Big).
\]
We will use the estimate $\int_{x}^\8 e^{-\frac{y^2}{2\si^2}} \dd y\le \frac{\si^2}x e^{-\frac{x^2}{2\si^2}}$ repeatedly in the following.

\emph{Case 1}: $s>0$. If $s>|\mu|$,  since $|\frac12-\frac{D'(\rho^2)}{D'(0)}|\le\frac12 $, we have
\[
\Big(s - \frac{\mu}2+\frac{\mu D'(\rho^2)}{D'(0)}\Big)^2\ge \Big(s- \Big|\frac12-\frac{D'(\rho^2)}{D'(0)}\Big| |\mu|\Big )^2\ge \frac{s^2}4.
\]
Then
\begin{align*}
  &\int_{|\mu|}^{\8} \int_{  \sqrt{T/s} }^{\8} f(s,\rho^2) \dd \rho \dd s\\
  &\le \int_{|\mu|}^\8  \int_{\sqrt{T/s}}^\8 [1+(s+|\mu|)^{N} +(s+|\mu|)^{N} \rho^{2N}] (\rho^{N-1}+\rho^{N+1}) e^{-\frac{N[\frac{s^2}4+\mu^2]\rho^2}{2D'(0)}} \dd \rho \dd s\\
  &\le C_{\mu,D}\Big( \frac{8D'(0)}{5\mu^2}\Big)^{N+1} \int_{|\mu| }^\8  \int_{\sqrt{\frac{T}{2D'(0)}(\frac{s}4+\frac{\mu^2}{s})} }^\8 \Big(\frac{2D'(0)}{\frac{s^2}4+\mu^2} \Big)^{(N-1)/2} (1+(s+|\mu|)^{N})r^{3N+1} e^{-Nr^2} \dd r \dd s\\
  &\le  \frac{ C_{\mu,D}^N}{N  \sqrt{T}} \int_{|\mu|}^\8  \frac{1+(s+|\mu|)^{N}} {(s^2+4\mu^2)^{(N-1)/2}} e^{-\frac{N T}{4D'(0)}(\frac{s}{4}+\frac{\mu^2}{s})}  \dd s\\
  &\le  \frac{ C_{\mu,D}^N}{N  \sqrt{T}} \int_{|\mu|}^\8  \frac{1} {s^2 } e^{-\frac{N Ts }{32 D'(0)}}   \dd s \le \frac{C_{\mu,D}^N}{N \sqrt{T}} e^{-\frac{|\mu| N T }{32 D'(0)}}.
\end{align*}
Here we have used the fact that $\sqrt{\frac{T}{2D'(0)}(\frac{s}4+\frac{\mu^2}{s})}\ge |\mu|\sqrt{\frac{T}{2D'(0)}}$ so that we can always choose $T$ large to guarantee $r>1$ and $r^4\le e^{r^2/2}$.

If $s\le |\mu|$, using the trivial bound $(s - \frac{\mu}2+\frac{\mu D'(\rho^2)}{D'(0)})^2\ge0$, we have
\begin{align*}
  &\int_{0}^{|\mu|}  \int_{\sqrt{T/s}}^{\8 } f(s,\rho^2) \dd \rho \dd s\\
  &\le \int_{0}^{|\mu|}  \int_{\sqrt{T/|\mu| }}^\8 [1+(s+|\mu|)^{N} +(s+|\mu|)^{N} \rho^{2N}] (\rho^{N-1}+\rho^{N+1}) e^{-\frac{N \mu^2\rho^2}{2D'(0)}} \dd \rho \dd s\\
  &\le C_D \int_{0}^{|\mu|}  \int_{ \sqrt{\frac{|\mu|T}{2  D'(0)}}}^\8 \Big[ \Big(\frac{2D'(0)}{\mu^2}\Big)^{3N/2}+1\Big] (1+(s+|\mu|)^{N})r^{3N+1} e^{-Nr^2} \dd r \dd s\\
  &\le \frac{C_{\mu,D}^N  }{N \sqrt{  T}} e^{-\frac{|\mu| N T}{4 D'(0)}}.
\end{align*}

\emph{Case 2}: $s<0$. After change of variable $s\to -s$, we can proceed in the same way as the case $s>0$ and find
\begin{align*}
  &\int_{-\8}^0\int_{\sqrt{-T/s}}^{\8}f(s,\rho^2)\dd \rho \dd s  =\int_{0}^\8 \int_{\sqrt{T/s}}^{\8}f(-s,\rho^2) \dd \rho \dd s\\
  &=\Big(\int_{0}^{|\mu|}\int_{\sqrt{T/s}}^{\8} + \int_{|\mu|}^{\8}\int_{\sqrt{T/s}}^{\8}  \Big) f(-s,\rho^2) \dd \rho \dd s\\
  &\le \frac{C_{\mu,D}^N}{N \sqrt T}\Big(   e^{-| \mu| NT/[32D'(0)]}+  e^{-|\mu|NT/[4 D'(0)]} \Big).
\end{align*}
Putting things together, we see that
\[
\ez \Crt_{N} ([-T, T]^c, (0,\8)) \le  \frac{C_{\mu,D}^N}{N \sqrt T}\Big(   e^{-| \mu| NT/[32D'(0)]}+  e^{-|\mu|NT/[4 D'(0)]} \Big).
\]
From here the first assertion follows.

(2) The last two claims follow somewhat different strategy. By conditioning and Young's inequality,
\begin{align*}
 &\ez \Big[|z_1'| \prod_{i=1}^{N-1} |(\frac{N-1}{N})^{1/2}\la_i-z_3'| \Big]\\
 &\le C^{N-1}\ez [ (\sfb+ |m_1|+|m_2| +\sqrt{-4D''(0)}|z_3'|) ({\la_{N-1}^*}^{N-1}+|z_3'|^{N-1})]\\
 &\le C_D^N(1+|m_1|^N+|m_2|^N).
\end{align*}
Using \pref{le:apest}, \pref{eq:asmp1} and \pref{eq:asmp2} together with the change of variable formulas \pref{eq:uvcov} and \pref{eq:m12cov},
\begin{align*}
& \ez \Crt_{N} (\rz, (R,\8)) \\
&\le   C_D^N S_{N-1}\int_{R}^\8  \int_{\rz} (1 +|m_1|^N+|m_2|^N)  \frac{e^{-\frac{(u-m_Y)^2}{2\si_Y^2}}}{\sqrt{2\pi}\si_Y}  \frac{e^{-\frac{N \mu^2 \rho^2}{2D'(0)}} }{(2\pi)^{N/2} D'(0)^{N/2}} \rho^{N-1} \dd u  \dd\rho \\
&\le  C_{\mu,D}^N S_{N-1} \int_{R}^\8  \int_{\rz} [1 +  |v|^{N}(\al\rho^2+\bt)^N]    \frac{e^{-\frac{N v^2}{2}}}{\sqrt{2\pi}}  \frac{e^{-\frac{N \mu^2 \rho^2}{2D'(0)}} }{(2\pi)^{N/2} D'(0)^{N/2}} \rho^{N-1} \dd v  \dd\rho  \\
&\le   C_{\mu,D}^N S_{N-1} \int_{R}^{\8} e^{-\frac{N \mu^2 \rho^2}{2D'(0)}} \rho^{N-1} \dd\rho\\
&\le \frac{C_{\mu,D}^N S_{N-1}}{NR} e^{-\frac{N \mu^2 R^2}{4D'(0)}}
\end{align*}
for $R$ large enough. Similarly,
\begin{align*}
& \ez \Crt_{N} (\rz, (0,\eps)) \le   C_{\mu,D}^N S_{N-1} \int_{0}^{\eps} e^{-\frac{N \mu^2 \rho^2}{2D'(0)}} \rho^{N-1} \dd\rho\le \frac{C_{\mu,D}^N S_{N-1}\eps^N}{N} .
\end{align*}
This completes the proof.
\end{proof}

We remark that we have actually proved the following stronger results with heavier notations from \pref{eq:ie12}:
\begin{align*}
\limsup_{ T\to \8} \limsup_{N\to \8} \frac1N \log  [I_1([-T, T]^c, (0,\8)) +I_2([-T, T]^c, (0,\8)) ]&= -\8,\\
\limsup_{ R \to \8} \limsup_{N\to \8} \frac1N \log  [I_1 (\rz, (R,\8))+I_2 (\rz, (R,\8))] &= -\8,\\
\limsup_{ \eps \to 0+} \limsup_{N\to \8} \frac1N \log  [I_1 (\rz, (0,\eps))+I_2 (\rz, (0,\eps))] &= -\8.
\end{align*}
The third claim also holds for $\mu=0$ with the same argument.
If $\mu=0$, observing the complexity function in \pref{se:whole}, it is reasonable to require $R_2<\8$.
\begin{lemma}\label{le:exptt2}
Let $\mu=0$ and $R<\8$. Then
\begin{align*}
\limsup_{ T\to \8} \limsup_{N\to \8} \frac1N \log  \ez \Crt_{N} ([-T, T]^c, [0,R)) &= -\8.
\end{align*}
\end{lemma}
\begin{proof}
The argument follows that of \pref{le:exptt} and is actually much easier. Indeed, we find
\begin{align*}
& \ez \Crt_{N} ([-T, T]^c, (0,R) ) \\
&\le   C_D^N S_{N-1}\int_{0}^{R} \int_{[-T,T]^c} (1 +m_u^N)  \frac1{\sqrt{2\pi}\si_Y} e^{-\frac{(u-m_Y)^2}{2\si_Y^2}} \frac1{(2\pi)^{N/2} D'(0)^{N/2}}  \rho^{N-1} \dd u  \dd\rho \\
&\le \frac{C_{D}^N S_{N-1} \sqrt N}{ (2\pi)^{\frac{N+1}{2}}D'(0)^{N/2}}  \int_{0}^R  \Big( \int_{T}^\8 +    \int_{-\8}^{-T} \Big ) [1+\rho^N+ (1+\rho^{2N})|u|^{N}  ]    \frac{ (1+\rho^2)}{ \rho^2}   e^{-\frac{N u^2 }{2D'(0) \rho^2}}  \rho^{N-1} \dd u  \dd\rho \\
&\le \frac{C_{R, D}^N S_{N-1}  \sqrt N}{  T }  e^{-\frac{N T^2}{ 4D'(0)  R^2}}.
\end{align*}
The proof is complete.
\end{proof}

We need the following fact.
\begin{lemma}\label{le:intbd}
Suppose $|\mu| +\frac1R>0$. Then for any $a>0,c>0,b,d\in\rz$ satisfying $aN+b<cN+d$, there exist constants $C_{\mu,D, a,b,c,d}>0, N_0>0$ such that for all $N>N_0$,
\begin{align*}
\int_0^{R}\int_{-\8}^\8  (1+|s|^{aN+b}) \exp\Big(-\frac{N(s^{2}+\mu^{2})
	\rho^{2}}{2D'(0)}\Big)\rho^{cN+d}\dd s\dd\rho \le C^N_{\mu,R, D, a,b,c,d}.
\end{align*}
\end{lemma}
\begin{proof}
If $\mu\neq 0$, changing the order of integration yields
\begin{align*}
&\int_0^{\8}\int_{\rz} (1+|s|^{aN+b}) \exp\Big(-\frac{N(s^{2}+\mu^{2}) \rho^{2}}{2D'(0)}\Big)\rho^{cN+d} \dd s\dd\rho\\
&= \int_{-\8}^\8 \int_0^\8 \Big(	\frac{D'(0)}{s^2+\mu^2}\Big)^{\frac{cN+d+1}{2}}( 1+|s|^{aN+b} ) r^{cN+d} e^{-Nr^2/2} \dd r\dd s\\
&\le C_{D,c,d}^N \int_{-\8}^\8 \frac{1+|s|^{aN+b}}{(s^2+\mu^2)^{\frac{cN+d+1}{2}}} \dd s\le C_{\mu,D,c,d}^N,
\end{align*}
where  in the last step we used the assumption $aN+b<cN+d$.  If $\mu=0$, then $R<\8$ and we have
\begin{align*}\
\int_0^{R}\int_{-\8}^\8  (1+|s|^{aN+b}) \exp\Big(-\frac{Ns^2
	\rho^{2}}{2D'(0)}\Big)\rho^{cN+d}\dd s \dd\rho \le   C^N_{a, b,D} \int_0^{R} (1+\rho^{-aN-b}) \rho^{cN+d} \dd\rho,
\end{align*}
which completes the proof.
\end{proof}
To save space, for an event $\Delta$ that may depend on the eigenvalues of GOE and other Gaussian random variables in question,  let us write
\begin{align*}
I_2(E,(R_1,R_2),\Delta) &= \frac{[-4D''(0)]^{\frac{N}{2}} }{2N} \sum_{i=1}^{N-1}\int_{R_1}^{R_2} \int_{E} \ez\Big[Z_i^2 \prod_{j\neq i} |(\frac{N-1}{N})^{1/2}\la_j-z_3'| \indi_\Delta \Big]  \\
  &\frac{ e^{-\frac{(u-m_Y)^2}{2\si_Y^2}}}{\sqrt{2\pi}\si_Y} \frac{e^{-\frac{N \mu^2 \rho^2}{2D'(0)}}}{(2\pi)^{N/2} D'(0)^{N/2}}  \rho^{N-1}   \dd u  \dd\rho.
%II(E,(R_1,R_2);\Delta) &= \int_{R_1}^{R_2}\int_{E }  \int_{\rz} \ez\Big(\Big| \sum_{k=1}^{N-1} {\xi_k'}^2 \prod_{j\neq k}^{N-1}\Big[\eta_j- \Big(\si_2 r+ \frac{\rho \sqrt{\al\bt}}{\sqrt N} z_3 -m_2 \Big)\Big]\Big|\indi_\Delta\Big)\\
%& \ \ \frac1{\sqrt{2\pi}} e^{-\frac{r^2}{2}}  \frac1{\sqrt{2\pi}\si_Y} e^{-\frac{(u-m_Y)^2}{2\si_Y^2}}  \frac1{(2\pi)^{N/2} D'(0)^{N/2}} e^{-\frac{N \mu^2 \rho^2}{2D'(0)}} \rho^{N-1} \dd r \dd u  \dd\rho.
\end{align*}
\begin{lemma}\label{le:goecpt}
Suppose $|\mu| +\frac1{R_2}>0$.  Then
\begin{align*}
&\limsup_{K\to\8} \limsup_{N \to\8} \frac1N\log I_2(E, (R_1,R_2), \{\la_{N-1}^*>K\} )=-\8,\\
&\limsup_{K\to\8} \limsup_{N \to\8} \frac1N\log I_2(E, (R_1,R_2), \{|z_3'-\ez(z_3')|>K\} )=-\8.
\end{align*}
\end{lemma}

\begin{proof}
Using \pref{eq:ladein} and choosing $K$ large so that $2t< e^{t^2/18}$ for $t\ge K$,
\begin{align}
&\ez[{(\la_{N-1}^*)}^{N-2}\indi\{\la_{N-1}^*>K\}]=\int_0^K Kt^{K-1}\pz(\la_{N-1}^*\ge K)\dd t+\int_K^\8 (N-2)t^{N-3} \pz(\la_{N-1}^*>t) \dd t \notag\\
&\le K^K e^{-(N-1)K^2/9}+ \int_K^\8 e^{-(N-1)t^2/18} \dd t \le  2 e^{-(N-1)K^2/18}.\label{eq:lantail}
\end{align}
%\begin{align*}
%&\ez[(\la_{N-1}^*)^{N-2}\indi\{\la_{N-1}^*>K\}]=\sum_{j=0}^\8 \ez[(\la_{N-1}^*)^{N-2}\indi\{K+j<\la_{N-1}^*\le K+j+1\}]\\
%&\le \sum_{j=0}^\8 (K+j+1)^{N-2} e^{-(N-1)(K+j)^2/9} \le \sum_{j=0}^\8 e^{-(N-1)(K+j)^2/18}\\
%&\le \frac12 e^{-(N-1)K^2/18}.
%\end{align*}
If $\mu\neq 0$, using \pref{eq:z3mom} and \pref{le:apest}, we obtain
\begin{align*}
&I_2(E,(R_1,R_2), \{\la_{N-1}^*>K\})\le C_D^N \int_0^\8\int_\rz  \ez [((\la_{N-1}^*)^{N-2}+z_3'^{N-2})\indi\{\la_{N-1}^*>K\}]\\
& \ \ \frac1{\sqrt{2\pi}\si_Y} e^{-\frac{(u-m_Y)^2}{2\si_Y^2}}  \frac1{(2\pi)^{N/2} D'(0)^{N/2}} e^{-\frac{N \mu^2 \rho^2}{2D'(0)}} \rho^{N-1}  \dd u  \dd\rho\\
&\le C_{\mu,D}^N e^{-(N-1)K^2/18} \int_0^\8\int_\rz  \Big[ 1+  |\frac{u}{\rho^2}-\frac\mu2+\frac{\mu D'(\rho^2)}{D'(0)}  |^{N-2} (1+\rho^{2(N-2)}) \Big] \\
& \ \   \frac1{\sqrt{2\pi}\si_Y} e^{-\frac{(u-m_Y)^2}{2\si_Y^2}}  \frac1{(2\pi)^{N/2} D'(0)^{N/2}} e^{-\frac{N \mu^2 \rho^2}{2D'(0)}} \rho^{N-1}   \dd u  \dd\rho\\
&\le C_{\mu,D}^N e^{-(N-1)K^2/18} \int_0^\8\int_\rz \Big[1 +s^{N-2}(1+\rho^{2(N-1)})\Big]  \exp\Big(-\frac{N(s^{2}+\mu^{2})
	\rho^{2}}{2D'(0)}\Big)\rho^{N-1}\dd s  \dd\rho.
\end{align*}
Here in the last step we used the observation $(1+\rho^{2(N-2)})(1+\rho^2)\le 4(1+\rho^{2(N-1)}).$ The assertion then follows from \pref{le:intbd}. Similarly, note that
\[
\pz(|z_3'-\ez(z_3')|>K)\le 2e^{-\frac{N(-4D''(0))K^2}{2(-2D''(0)-\bt^2)}}\le 2 e^{-NK^2}.
\]
It follows that for $K$ large enough,
\[
\ez(|z_3'-\ez(z_3')|^{N-2}\indi\{|z_3'-\ez(z_3')|>K\}) \le 4e^{-NK^2/2}.
\]
From here we deduce that
\begin{align*}
&I_2(E,(R_1,R_2), \{|z_3'-\ez(z_3')|>K\})\le C_D^N \int_0^\8\int_\rz  \ez [({(\la_{N-1}^*)}^{N-2}+ |\ez(z_3')|^{N-2}\\
&\ \ +|z_3'-\ez(z_3')|^{N-2})\indi\{|z_3'-\ez(z_3')|>K\}]  \frac1{\sqrt{2\pi}\si_Y} e^{-\frac{(u-m_Y)^2}{2\si_Y^2}}  \frac1{(2\pi)^{N/2} D'(0)^{N/2}} e^{-\frac{N \mu^2 \rho^2}{2D'(0)}} \rho^{N-1}  \dd u  \dd\rho\\
&\le C_{\mu,D}^N e^{-NK^2/2} \int_0^\8\int_\rz  \Big[ 1+  |\frac{u}{\rho^2}-\frac\mu2+\frac{\mu D'(\rho^2)}{D'(0)}  |^{N-2} (1+\rho^{2(N-2)}) \Big]\\
&\quad \frac1{\sqrt{2\pi}\si_Y} e^{-\frac{(u-m_Y)^2}{2\si_Y^2}}  \frac1{(2\pi)^{N/2} D'(0)^{N/2}} e^{-\frac{N \mu^2 \rho^2}{2D'(0)}} \rho^{N-1}  \dd u  \dd\rho .
\end{align*}
The rest of argument is the same as above.  The case $\mu=0$ and $R_2<\8$ follows the same steps and is omitted.
\end{proof}

\begin{lemma}\label{le:goecpt2}
Suppose $|\mu| +\frac1{R_2}>0$.  Then for any $\de>0$,
\[
\limsup_{N \to\8} \frac1N\log I_2(E, (R_1,R_2), \{L(\la_{1}^{N-1})\notin B(\si_{{\rm sc}},\delta) \} )=-\8.
\]
\end{lemma}
\begin{proof}
 We only argue for the harder case $\mu\neq 0$. Using \pref{eq:z3mom}, the Cauchy--Schwarz inequality and \pref{eq:conineq}, we have
\begin{align*}
 &\ez \Big[\prod_{i=1,i\neq j}^{N-1} |(\frac{N-1}{N})^{1/2}\la_i - z_3'  | \indi\{L(\la_{1}^{N-1})\notin B(\si_{{\rm
 sc}},\delta)\}\Big] \\
 &\le C^N \ez [((\la_{N-1}^*)^{N-2}+z_3'^{N-2}
 )\indi\{L(\la_{1}^{N-1})\notin B(\si_{{\rm sc}},\delta)\}
 ]\\
 &\le C^N [\ez ({(\la_{N-1}^*)}^{2(N-2)}+z_3'^{2(N-2)} )]^{1/2}
 \pz(L(\la_{1}^{N-1})\notin B(\si_{{\rm sc}},\delta))^{1/2}\\
 &\le C_{\mu,D}^{N}\Big[1  +\Big|\frac{u}{\rho^{2}}-\frac{\mu}{2}+\frac{\mu D'(\rho^{2})}{D'(0)}\Big|^{N-2}(1+\rho^{2(N-2)})\Big] e^{-\frac12c(N-1)^{2}}.
\end{align*}
Together with \pref{le:intbd}, we deduce that
\begin{align*}
  & I_2(E, (R_1,R_2), \{L(\la_{1}^{N-1})\notin B(\si_{{\rm sc}},\delta) \} ) \notag\\
 &\le C_{\mu,D}^{N} e^{-cN^{2}}
 \int_{R_{1}}^{R_{2}}\int_{E/\rho^{2}} \Big[1+\Big|s-\frac{\mu}{2}+\frac{\mu
 D'(\rho^{2})}{D'(0)}\Big|^{N-2}(1+\rho^{2(N-1)})\Big]\notag\\
 & \ \ \ \exp\Big(-\frac{N[(s-\frac{\mu}{2}
 +\frac{\mu D'(\rho^{2})}{D'(0)})^{2}+\mu^{2}]
 \rho^{2}}{2D'(0)}\Big)\rho^{N-1}\dd s\dd\rho\notag\\	
  &\le C_{\mu,D}^{N} e^{-cN^{2}}
 \int_{0}^{\8}\int_{\rz} \Big[1+|v|^{N-2}(1+\rho^{2(N-1)})\Big]  \exp\Big(-\frac{N[v^{2}+\mu^{2}]
 \rho^{2}}{2D'(0)}\Big)\rho^{N-1}\dd v\dd\rho\notag\\	
 & \le C_{\mu,D}^{N}e^{-cN^{2}}.
\end{align*}
From here the assertion follows.
\end{proof}

For an event $\Delta$,  let us write
\begin{align*}
I_1(E,(R_1,R_2),\Delta) &= [-4D''(0)]^{\frac{N-1}{2}} \int_{R_1}^{R_2} \int_{E} \ez\Big[|z_1'| \prod_{i=1}^{N-1} |(\frac{N-1}{N})^{1/2}\la_i-z_3'| \indi_\Delta\Big]\notag\\
  &\frac{ e^{-\frac{(u-m_Y)^2}{2\si_Y^2}}}{\sqrt{2\pi}\si_Y} \frac{e^{-\frac{N \mu^2 \rho^2}{2D'(0)}}}{(2\pi)^{N/2} D'(0)^{N/2}}  \rho^{N-1}   \dd u  \dd\rho.
\end{align*}
The argument in this part shares the same spirit as that for $I_2$. %However, now we have $N$ terms in the product to evaluate the expectation compared to $N-2$ terms for $I_2$. This prevents us using \pref{le:intbd} like above. Instead, we rely on compactness in this part.

\begin{lemma}\label{le:exptti11}
  Suppose $|\mu|+\frac1{R_2}>0$. Then we have
\begin{align*}
&\limsup_{K\to\8} \limsup_{N \to\8} \frac1N\log I_1(E, (R_1,R_2), \{\la_{N-1}^*>K\} )=-\8,\\
&\limsup_{K\to\8} \limsup_{N \to\8} \frac1N\log I_1(E, (R_1,R_2), \{|z_3'-\ez(z_3')|>K\} )=-\8.
\end{align*}
\end{lemma}

\begin{proof}
  The argument is similar to that of \pref{le:goecpt}. As there, we only provide details for the case $\mu\neq 0$. Note that $\sfb^2\le -4D''(0)$. By \pref{eq:lantail},  \pref{eq:z13con0}, \pref{eq:absgau}, Young's inequality and conditioning, we find
\begin{align*}
  &\ez\Big[|z_1'| \prod_{i=1}^{N-1} |(\frac{N-1}{N})^{1/2}\la_i-z_3'| \indi\{\la_{N-1}^*>K \} \Big]\\
  &\le C^N\ez\Big[\Big(\frac{\sqrt2 \sfb}{\sqrt{\pi N}}+|\bar \sfa|\Big) ({(\la_{N-1}^*)}^{N-1} +|z_3'|^{N-1}) \indi\{\la_{N-1}^*>K \}\Big]\\
  &\le C^{N}\ez [ (\sfb+ |m_1|+|m_2| +\sqrt{-4D''(0)}|z_3'|) ({(\la_{N-1}^*)}^{N-1}+|z_3'|^{N-1})\indi\{\la_{N-1}^*>K\}]\\
 &\le C_{D}^N e^{-(N-1) K^2/18}(1+|m_1|^N+|m_2|^N).
\end{align*}
Using \pref{eq:asmp1}, \pref{eq:asmp2} and the change of variable formulas \pref{eq:uvcov} and \pref{eq:m12cov},
\begin{align*}
&I_1(E, (R_1,R_2), \{\la_{N-1}^*>K\} )\\
&\le  C_{D}^N e^{-(N-1) K^2/18}\int_{0}^\8  \int_{\rz} (1 +|m_1|^N+|m_2|^N)  \frac{e^{-\frac{(u-m_Y)^2}{2\si_Y^2}}}{\sqrt{2\pi}\si_Y}  \frac{e^{-\frac{N \mu^2 \rho^2}{2D'(0)}} }{(2\pi)^{N/2} D'(0)^{N/2}} \rho^{N-1} \dd u  \dd\rho \\
&\le  C_{\mu,D}^N e^{-(N-1) K^2/18} \int_{0}^\8  \int_{\rz} [1 +  |v|^{N}(\al\rho^2+\bt)^N]    \frac{e^{-\frac{N v^2}{2}}}{\sqrt{2\pi}}  \frac{e^{-\frac{N \mu^2 \rho^2}{2D'(0)}} }{(2\pi)^{N/2} D'(0)^{N/2}} \rho^{N-1} \dd v  \dd\rho  \\
&\le   C_{\mu,D}^N e^{-(N-1) K^2/18}  \int_{0}^{\8} e^{-\frac{N \mu^2 \rho^2}{2D'(0)}} \rho^{N-1} \dd\rho\\
&\le C_{\mu,D}^N e^{-(N-1) K^2/18}.
\end{align*}
From here the first assertion follows.  The argument for the second one is in the same fashion after observing $|z_3'|\le |z_3'-\ez(z_3')|+|\ez(z_3')|$ and
\begin{align*}
  &\ez\Big[|z_1'| \prod_{i=1}^{N-1} |(\frac{N-1}{N})^{1/2}\la_i-z_3'| \indi\{ |z_3'-\ez(z_3')|>K \} \Big]\\
  &\le C^{N}\ez [ (\sfb+ |m_1|+|m_2| +\sqrt{-4D''(0)}|z_3'|) ({(\la_{N-1}^*)}^{N-1}+|z_3'|^{N-1})\indi\{|z_3'-\ez(z_3')|>K\}]\\
 &\le C_{D}^N e^{-N K^2/2}(1+|m_1|^N+|m_2|^N).\qedhere
\end{align*}
\end{proof}

\begin{lemma}\label{le:exptti12}
 Let $\de>0$.  Suppose $|\mu|+\frac1{R_2}>0$. Then we have
\begin{align*}
\limsup_{N \to\8} \frac1N\log I_1(E, (R_1,R_2), \{L(\la_1^{N-1})\notin B(\si_{\rm sc},\de)\} )=-\8.
\end{align*}
\end{lemma}

\begin{proof}
  The proof is similar to that of \pref{le:goecpt2} and we only provide the difference for the case $\mu\neq0$.  Conditioning as in the proof of \pref{le:exptti11}, using Young's inequality, the Cauchy--Schwarz inequality and \pref{eq:conineq}, we find
  \begin{align*}
    &\ez\Big[|z_1'| \prod_{i=1}^{N-1} |(\frac{N-1}{N})^{1/2}\la_i-z_3'| \indi\{L(\la_1^{N-1})\notin B(\si_{\rm sc},\de) \} \Big]\\
    &\le C_{D}^N (1+|m_1|^{2N}+|m_2|^{2N})^{1/2} e^{-c N^2}.
  \end{align*}
  The rest of argument follows verbatim that of \pref{le:exptti11}.
\end{proof}

\section{Proof of \pref{th:cpsublevel}}\label{se:4}

%To deal with $II$, since $\ez{\xi_k'}^2=\frac{-2D''(0)}{N}$, writing $z_3'= (\si_2 r+ \frac{\rho \sqrt{\al\bt}}{\sqrt N} z_3 -m_2 )/\sqrt {-4D''(0)}$, by symmetry,
%\begin{align*}
%  &\ez\Big(\Big| \sum_{k=1}^{N-1} {\xi_k'}^2 \prod_{j\neq k}^{N-1}\Big[\eta_j - \Big(\si_2 r+ \frac{ \rho\sqrt{\al\bt}}{\sqrt N} z_3 -m_2 \Big)\Big]\Big| |z_2=r\Big)\\
%  &\le \frac{[-4D''(0)]^{(N-1)/2}}{2N}  \sum_{k=1}^{N-1}\ez\prod_{j\neq k}^{N-1}\Big|\frac{\eta_j }{\sqrt {-4D''(0)}}- \frac1{\sqrt {-4D''(0)}} \Big(\si_2 r+ \frac{\rho \sqrt{\al\bt}}{\sqrt N} z_3 -m_2 \Big)\Big|\\
%  &=\frac{[-4D''(0)]^{(N-1)/2}(N-1)}{2N}  \ez_{z_3} \int_{\rz^{N-1}}\prod_{i=1}^{N-2} |x_i - z_3'  | p_\GOE (x_1,...,x_{N-1}) \dd x_1 ... \dd x_{N-1},
%\end{align*}
%where $p_\GOE (x_1,...,x_{N-1})$ is the p.d.f.~of the (unordered) eigenvalues of an $(N-1)\times (N-1)$ GOE matrix.

For a probability measure $\nu$ defined on $\rz$, recall the functions $\Psi(\nu,x)$ and $\Psi_*(x)$ as in \pref{eq:psidef0}. Let us define
\begin{align}
 \psi(\nu,\rho,u,y)  &= \Psi(\nu,y)-\frac{(u-m_Y)^2}{2\Big(D(\rho^2)-\frac{D'(\rho^2)^2 \rho^2}{D'(0)} \Big)}-\frac{-2D''(0)}{-2D''(0)-\bt^2}\Big(y+\frac{m_2}{\sqrt{-4D''(0)}}\Big)^2\notag \\
&\ \  -\frac{\mu^2\rho^2}{2D'(0)}+\log \rho, \label{eq:psidef} \\
\psi_*(\rho,u,y)&=\psi(\si_{\rm sc},\rho,u,y). \notag
% & \psi(\nu,\rho,u,y)  =\Psi(\nu, x)-\frac{(u-\frac{\mu\rho^2}{2}+\frac{\mu D'(\rho^2) \rho^2}{D'(0) } )^2}{ 2\Big(D(\rho^2)-\frac{D'(\rho^2)^2 \rho^2}{D'(0)} \Big)} -\frac{\mu^2\rho^2}{2D'(0)}+\log \rho \notag
%  & -\frac{-2D''(0)}{-2D''(0)-\frac{[D'(\rho^2 )-D'(0)]^2}{{ D(\rho^2 )-\frac{D'(\rho^2)^2 \rho^2}{D'(0)}}}} \Big(x + \frac{1} {\sqrt{-4D''(0)}}\Big[ \mu + \frac{(t-\frac{\mu \rho^2}{2} +\frac{\mu D'(\rho^2) \rho^2}{D'(0)}) (D'(\rho^2) -D'(0) )}{D( \rho^2) -\frac{D'(\rho^2)^2 \rho^2}{D'(0) }}\Big]\Big)^2 .
\end{align}
Recalling the notations \pref{eq:msialbt}, $\psi_*(\rho,u,y)$ can be written explicitly as
\begin{align}
 		&\psi_*(\rho,u,y)= \Psi_*( y)-\frac{(u-\frac{\mu\rho^2}{2}+\frac{\mu D'(\rho^2) \rho^2}{D'(0) } )^2}{ 2 (D(\rho^2)-\frac{D'(\rho^2)^2 \rho^2}{D'(0)} )} -\frac{\mu^2\rho^2}{2D'(0)}+\log \rho-\frac{-2D''(0)}{-2D''(0)-\frac{[D'(\rho^2 )-D'(0)]^2}{{ D(\rho^2 )-\frac{D'(\rho^2)^2 \rho^2}{D'(0)}}}}
 		\notag \\
 		& \times\Big(y+\frac{1} {\sqrt{-4D''(0)}}\Big[ \mu + \frac{(u-\frac{\mu \rho^2}{2} +\frac{\mu D'(\rho^2) \rho^2}{D'(0)}) (D'(\rho^2) -D'(0) )}{D( \rho^2) -\frac{D'(\rho^2)^2 \rho^2}{D'(0) }}\Big]\Big)^2.
\label{eq:psifunction} 	
\end{align}

\begin{lemma}\label{le:psirho0}
  For any $u$ and $y$ fixed, we have $\lim_{\rho\to0+}\psi_*(\rho,u,y) = -\8$. For any $\rho$ and $u$ fixed, we have $\lim_{|y|\to \8}\psi_*(\rho,u,y) = -\8$.
\end{lemma}
\begin{proof}
From \pref{le:albtd}, we know $ D(\rho^2)-\frac{D'(\rho^2)^2\rho^2}{D'(0)} \sim -\frac32D''(0)\rho^4$ as $\rho\to0+$.
For any $\eps>0$ and $\rho\in(0,\eps)$, we may find $c_\eps$ such that
\begin{align*}
  \psi_*(\rho,u,y)-\Psi_*(y) & \le -\frac{(u-\frac{\mu\rho^2}{2}+\frac{\mu D'(\rho^2) \rho^2}{D'(0) } )^2}{ 2 (D(\rho^2)-\frac{D'(\rho^2)^2 \rho^2}{D'(0)} )} -\frac{\mu^2\rho^2}{2D'(0)}+\log \rho\\
  &\le -\frac{(\frac{u}{\rho^2}-\frac{\mu}{2}+\frac{\mu D'(\rho^2)}{D'(0)})}{-3c_\eps D''(0)}-\frac{\mu^2 \rho^2}{2D'(0)} +\log \rho .
\end{align*}
The right-hand side clearly tends to $-\8$ as $\rho\to0+$.

Since $\frac{-2D''(0)}{-2D''(0)-\bt^2}\ge 1$, from the definition it is clear to see $\lim_{|y|\to \8}\psi_*(\rho,u,y) = -\8$ for fixed $\rho$ and $u$.
\end{proof}
Let $\llbracket \ell \rrbracket =\{i_1,...,i_\ell\}\subset [N-1]$. For any 1-Lipschitz function $f$, we have
\begin{align}
  &\Big|\frac1{N-1} \sum_{j=1}^{N-1}f(\la_j) -\frac1{N-1-\ell}\sum_{j\in[N-1]\setminus \llbracket \ell \rrbracket }f(\la_j)\Big| \notag\\
  &\le \frac1{(N-1)(N-1-\ell)}\sum_{j\in[N-1]\setminus \llbracket \ell \rrbracket} |(N-1-\ell)f(\la_j)+\sum_{i\in \llbracket \ell \rrbracket}f(\la_i) -(N-1)f(\la_j) |\notag\\
%  &\le \frac{1}{(N-1)(N-2)}\sum_{j=1,j\neq i}^{N-1}|f(\la_i) -f(\la_j)|\\
  &\le \frac{\ell}{N-1}\max_{i,j}|\la_i-\la_j|. \label{eq:flanl}
\end{align}

\subsection{Upper bound}\label{se:ub}
\begin{proposition}\label{pr:ubi2}
    Suppose $\bar E$ is compact and $0\le R_1<R_2<\8$. Under Assumptions I, II and IV, we have
    \begin{align*}
        \limsup_{N\to\8} \frac1N\log I_2(E,(R_1,R_2)) \le \frac12 \log[-4D''(0)] -\frac12\log D'(0) -\frac12\log(2\pi)+\sup_{(\rho,u,y)\in F}\psi_*(\rho,u,y),
    \end{align*}
    where  $F=\{(\rho,u,y):y\in\rz, \rho \in (R_1, R_2),u\in  \bar E\}$ and $\psi_*(\rho,u,y)$ is given as in \pref{eq:psifunction}.
\end{proposition}

\begin{proof}
    Since
    \begin{align*}
        I_2(E,(R_1,R_2))&=I_2(E,(R_1,R_2),\{L(\la_1^{N-1})\in B_K(\si_{{\rm sc}},\de), |z_3'-\ez(z_3')|\le K\}) \\
        &\ \ +I_2(E,(R_1,R_2),\{L(\la_1^{N-1})\notin B_K(\si_{{\rm sc}},\de)\}\cup \{|z_3'-\ez(z_3')|> K\}),
    \end{align*}
    by Lemmas \ref{le:goecpt} and \ref{le:goecpt2}, we can always choose $K$ large enough so that the second term is exponentially negligible as $N\to\8$, provided the first term yields a finite quantity in the limit. We only need to consider the first term.

    Using \pref{eq:flanl}, if $L(\la_{i=1}^{N-1})\in B_K(\si_{\rm sc}, \de) $, we may choose $N$ large enough so that $L((\frac{N-1}{N})^{1/2}\la_{j=1,j\neq i}^{N-1})\in  B_K(\si_{\rm sc}, 2\de)$. It follows that for any $i\in[N-1]$,
    \begin{align}\label{eq:jile}
    \prod_{j=1, j\neq i}^{N-1} |(\frac{N-1}{N})^{1/2}\la_j-z_3' | \indi\{ L(\la_{i=1}^{N-1})\in B_K(\si_{\rm sc}, \de) \}\le e^{(N-2)\sup_{\nu\in B_K(\si_{\rm sc},2\de)} \Psi(\nu,z_3')}.
    \end{align}
    By \pref{le:albtd} and \pref{eq:asmp1}, we have $c_{D,R_2}:=\inf_{R_1<\rho<R_2}-2D''(0)-\bt^2>0$. It follows that
    \begin{align*}
        &\frac{\sqrt{-4ND''(0)}}{\sqrt{2\pi(-2D''(0)-\bt^2)}} \exp\Big(-\frac{-2ND''(0)(y+\frac{m_2}{\sqrt{-4D''(0)}})^2}{-2D''(0)-\bt^2} \Big)\\
        &\le \frac{\sqrt{-4ND''(0)}}{\sqrt{2\pi c_{D,R_2}}} \exp\Big(-\frac{-2ND''(0)(y+\frac{m_2}{\sqrt{-4D''(0)}})^2}{-2D''(0)-\bt^2} \Big).
    \end{align*}
     Let
    \begin{align}\label{eq:fdeset}
        F(\de)=\Big\{ (\nu,\rho,u,y): \nu\in B_{K}(\si_{\rm sc},\de), y\in\Big [-\frac{m_2}{\sqrt{-4D''(0)}}-K, -\frac{m_2}{\sqrt{-4D''(0)}}+K \Big],\notag \\
         \rho \in(R_1,R_2), u\in  \bar E \Big \}.
    \end{align}
    Using $\rho^2\le R_2^2$ and the fact that all summands of $\psi(\nu,\rho,u,y)$ in \pref{eq:psidef} are bounded from above on $F(\de)$, we deduce from \pref{le:apest}
    \begin{align*}
        &I_2(E,(R_1,R_2),\{L(\la_1^{N-1})\in B_K(\si_{{\rm sc}},\de), |z_3'-\ez(z_3')|\le K\})\\
        &\le [-4D''(0)]^{N/2}\int_{R_1}^{R_2}\int_E \ez\Big[e^{(N-2)\sup_{\nu\in B_K(\si_{\rm sc},2\de)} \Psi(\nu,z_3')}\indi\{|z_3'-\ez(z_3')|\le K\} \Big]\\
        &\ \ \frac{ \sqrt{N} e^{-\frac{N(u-m_Y)^2}{2\big[D(\rho^2)-\frac{D'(\rho^2)^2\rho^2}{D'(0)} \big]}}}{\sqrt{2\pi(D(\rho^2)-\frac{D'(\rho^2)^2\rho^2}{D'(0)} )}} \frac{e^{-\frac{N \mu^2 \rho^2}{2D'(0)}}}{(2\pi)^{N/2} D'(0)^{N/2}}  \rho^{N-1}   \dd u  \dd\rho\\
        &\le \frac{C_{D,R_2, |E|} N [-4D''(0)]^{\frac{N+1}2}}{(2\pi)^{\frac{N+2}{2}} D'(0)^{\frac{N}2} } \exp\Big[ (N-3) \sup_{(\nu,\rho,u,y)\in F(2\de)} \psi(\nu,\rho,u,y)\Big],
    \end{align*}
    where $|E|$ is the Lebesgue measure of $E$.
    Since $\psi(\nu,\rho,u,y)$ is an upper semi-continuous function on $F(2\de)$ and attains its maximum on the closure $\overline{F(2\de)}$, we have
    \begin{align*}
        &\limsup_{\de\to0+} \sup_{(\nu,\rho,u,y)\in F(2\de)}  \psi(\nu,\rho,u,y)  \le \sup_{(\rho,u,y)\in F(0)}\psi_*(\rho,u,y).
    \end{align*}
    By Lemmas \ref{le:goecpt} and \ref{le:psirho0}, the continuous function $\psi_*(\rho,u,y)$ attains its maximum in $\bar F$ at some point $(\rho_*,u_*,y_*)$ with $\rho_*>0$. Therefore we may choose $K$ large enough in the beginning so that
    \[
    \sup_{(\rho,u,y)\in F(0)}\psi_*(\rho,u,y) = \psi_*(\rho_*,u_*,y_*).
    \]
    This justifies that $\sup_{(\rho,u,y)\in F}\psi_*(\rho,u,y)>-\8$ and the proof is complete.
\end{proof}

%\subsection{Upper bound for $I_1$}\label{se:ubi}

%We remark that the assumptions on compactness of $\bar E$ and $R_1>0$ are essential for the above two lemmas, which differ from the upper bound for $I_2$.

\begin{proposition}\label{pr:ubi1}
    Suppose $\bar E$ is compact and $0\le R_1<R_2<\8$. Under Assumptions I, II and IV, we have
    \begin{align*}
        \limsup_{N\to\8} \frac1N\log I_1(E,(R_1,R_2)) \le \frac12 \log[-4D''(0)] -\frac12\log D'(0) -\frac12\log(2\pi)+\sup_{(\rho,u,y)\in F}\psi_*(\rho,u,y),
    \end{align*}
    where  $F=\{(\rho,u,y):y\in\rz, \rho \in (R_1, R_2),u\in  \bar E\}$ and $\psi_*(\rho,u,y)$ is given as in \pref{eq:psifunction}.
\end{proposition}

\begin{proof}
    By the remark after \pref{le:exptt}, we know
    \begin{align*}
      \limsup_{N\to\8}\frac1N\log I_1(E,(0,R_2))& = \limsup_{N\to\8}\frac1N\log I_1(E,(\eps,R_2))
    \end{align*}
    by choosing $\eps>0$ small enough. Hence, we may assume $R_1>0$.    Similar to the proof of \pref{pr:ubi2}, since
    \begin{align*}
        I_1(E,(R_1,R_2))&=I_1(E,(R_1,R_2),\{L(\la_1^{N-1})\in B_K(\si_{{\rm sc}},\de), |z_3'-\ez(z_3')|\le K\}) \\
        &\ \ +I_1(E,(R_1,R_2),\{L(\la_1^{N-1})\notin B_K(\si_{{\rm sc}},\de)\}\cup \{|z_3'-\ez(z_3')|> K\}),
    \end{align*}
    thanks to Lemmas \ref{le:exptti11} and \ref{le:exptti12}, by choosing $K$ large enough, it suffices to consider the first term. Since $0<R_1<R_2<\8$, using continuity of functions in question, conditioning with \pref{eq:z13con0} and \pref{le:apest} for $\si_Y$,
    \begin{align*}
      &I_1(E,(R_1,R_2),\{L(\la_1^{N-1})\in B_K(\si_{{\rm sc}},\de), |z_3'-\ez(z_3')|\le K\})\\
      &\le [-4D''(0)]^{\frac{N-1}{2}} \sup_{R_1\le  \rho\le R_2, u\in \bar E, |y+\frac{m_2}{\sqrt{-4D''(0)}}|\le K} (\sfb+|m_1|+|m_2|+\sqrt{-4D''(0)}|y|) \int_{R_1}^{R_2} \int_{E}  \notag\\
  &\ \  \ez\Big[e^{(N-1)\sup_{\nu\in B_K(\si_{\rm sc},\de)} \Psi(\nu,z_3')}\indi\{|z_3'-\ez(z_3')|\le K\} \Big] \frac{ e^{-\frac{(u-m_Y)^2}{2\si_Y^2}}}{\sqrt{2\pi}\si_Y} \frac{e^{-\frac{N \mu^2 \rho^2}{2D'(0)}}}{(2\pi)^{N/2} D'(0)^{N/2}}  \rho^{N-1}   \dd u  \dd\rho\\
  &\le \frac{C_{R_1,R_2,D,K,\bar E} [-4D''(0)]^{\frac{N}{2}}} {(2\pi)^{\frac{N+2}{2}} D'(0)^{\frac{N}2}} \exp\Big[(N-3)\sup_{(\nu,\rho,u,y)\in F(\de)}\psi(\nu,\rho,u,y) \Big],
    \end{align*}
    where $F(\de)$ is given as in \pref{eq:fdeset} and the supremum of $|m_1|+|m_2|$ may depend on $R_1$. The assertion follows from the upper semi-continuity of $\psi(\nu,\rho,u,y)$ on $F(\de)$ by sending $N\to\8$ and $\de\to0+$.
\end{proof}

\subsection{Lower bound}

\begin{proposition}\label{pr:lbd*}
    Suppose $E$ is an open set and $0\le R_1<R_2<\8$. Under Assumptions I, II and IV, we have
    \begin{align*}
        \liminf_{N\to\8} \frac1N\log \ez\Crt_N(E,(R_1,R_2)) \ge \frac12+ \frac12 \log[-4D''(0)] -\frac12\log D'(0) +\sup_{(\rho,u,y)\in F}\psi_*(\rho,u,y),
    \end{align*}
    where  $F=\{(\rho,u,y):y\in\rz, \rho \in (R_1, R_2),u\in  \bar E\}$ and $\psi_*(\rho,u,y)$ is given as in \pref{eq:psifunction}.
\end{proposition}
\begin{proof}
Using \pref{eq:absgau} and \pref{eq:z13con0}, we know
\begin{align}\label{eq:z13con}
&\ez\big[|z_1' - h(z_3') |  |  z_3'=y \big] \ge \sqrt{\frac2\pi}\Big[ \frac{-4D''(0)}{N} +\frac{2D''(0)\al^2\rho^4}{N(-2D''(0)-\bt^2)}\Big]^{1/2},
\end{align}
where $h(z_3')$ only depends on $z_3'$. By conditioning, using \pref{eq:schur} and \pref{eq:ayg}, %and the formula of determinant for rank one updates, we have
\begin{align*}
&\ez(|\det G| )= \ez(|\det G_{**}| | z_1'-\xi^\sfT G_{**}^{-1} \xi |  )\\
&=[-4D''(0)]^{\frac{N-1}{2}} \ez[ |\det ((\frac{N-1}{N})^{1/2}\GOE_{N-1}-z_3' I_{N-1}) | \ez(  | z_1'-\xi^\sfT G_{**} ^{-1} \xi |  | \GOE_{N-1}, \xi, z_3')]
\\
&\ge  [-4D''(0)]^{\frac{N-1}{2}} \sqrt{\frac2\pi}\Big[ \frac{-4D''(0)}{N} +\frac{2D''(0)\al^2\rho^4}{N(-2D''(0)-\bt^2)} \Big]^{1/2} \\
&\quad  \frac{\sqrt{N(-4D''(0))}}{\sqrt{2\pi (-2D''(0)-\bt^2)} }  \int_{\rz^{N-1}}    \prod_{i=1}^{N-1} \int_\rz |(\frac{N-1}{N})^{1/2}x_i-y| \exp\Big[-\frac{-4ND''(0)(y +\frac{m_2}{\sqrt{-4D''(0)}})^2}{2(-2D''(0)-\bt^2)} \Big]  \dd y\\
&\quad   p_{\GOE}(x_1,...,x_{N-1}) \prod_{i=1}^{N-1} \dd x_i
\end{align*}
where $p_{\GOE}(x_1,...,x_{N-1})$ is the joint density of the unordered eigenvalues of GOE.

Without loss of generality we assume $E$ is non-empty. Choose $(\rho_*,u_*,y_*)$ as that in the proof of \pref{pr:ubi2}; i.e., it is a maximum of $\psi_*(\rho,u,y)$ on $[R_1,R_2]\times \bar E\times \rz$. If there are multiple points for $\psi_*$ to attain its maximum, we just choose one to be $(\rho_*,u_*,y_*)$. Recall that $\rho_*>0$. Then $(\rho_*-\de', \rho_*+\de')\cap [R_1,R_2]$ and $(u_*-\de',u_*+\de')\cap \bar E$ must be non-empty for any $\de'>0$. If $\rho_*$ and $u_*$ are both interior points, we choose $\de'>0$ small enough so that $(\rho_*-\de', \rho_*+\de')\subset (R_1,R_2)$ and $(u_*-\de',u_*+\de')\subset E$. If either $\rho_*$ or $u_*$ is a boundary point, by abuse of notation we still write $(\rho_*-\de', \rho_*+\de')$ and $(u_*-\de',u_*+\de')$ with the understanding that one endpoint should be replaced by $\rho_*$ or $u_*$ so that we always have $(\rho_*-\de', \rho_*+\de')\subset (R_1,R_2)$ and $(u_*-\de',u_*+\de')\subset E$. Using \pref{eq:asmp1}, the right-hand side of \pref{eq:z13con} attains strictly positive minimum for $\rho\in[\rho_*-\de', \rho_*+\de']$.
By restricting to small intervals, we find
\begin{align*}
&\int_{R_1}^{R_2}\int_{E} \ez(|\det G |) \frac1{\sqrt{2\pi}\si_Y} e^{-\frac{(u-m_Y)^2}{2\si_Y^2}}  \frac1{(2\pi)^{N/2} D'(0)^{N/2}} e^{-\frac{N \mu^2 \rho^2}{2D'(0)}} \rho^{N-1} \dd u  \dd\rho\\
&\ge  [-4D''(0)]^{\frac{N-1}{2}} \sqrt{\frac2\pi} \int_{\rho_*-\de'}^{\rho_*+\de'}\int_{u_*-\de'}^{u_*+\de'} \int_{y_*-\de_1}^{y_*+\de_1}   \\
&\quad  \Big[ \frac{-4D''(0)}{N} +\frac{2D''(0)\al^2\rho^4}{N(-2D''(0)-\bt^2)}\Big]^{1/2}
 \frac{\sqrt{N(-4D''(0))}}{\sqrt{2\pi (-2D''(0)-\bt^2)} } \\
 &\quad  \int_{\rz^{N-1}}   \prod_{i=1}^{N-1} |(\frac{N-1}{N})^{1/2} x_i-y| \exp\Big[-\frac{-4ND''(0)(y+\frac{m_2}{\sqrt{-4D''(0)}})^2}{2(-2D''(0)-\bt^2)} \Big]   \\
&\quad   p_{\GOE}(x_1,...,x_{N-1}) \prod_{i=1}^{N-1} \dd x_i \frac1{\sqrt{2\pi}\si_Y} e^{-\frac{(u-m_Y)^2}{2\si_Y^2}}  \frac1{(2\pi)^{N/2} D'(0)^{N/2}} e^{-\frac{N \mu^2 \rho^2}{2D'(0)}} \rho^{N-1} \dd y \dd u  \dd\rho\\
&=: \ex(\de',\de_1),
\end{align*}
where $\de_1>0$ will be specified in the following. We consider two cases.

\emph{Case 1}: $y_* \notin[-\sqrt2,\sqrt2]$. In this case, there exist $\eps_1>0$ small enough so that $y_*\notin [-\sqrt2-3\eps_1, \sqrt2+3\eps_1]$. We can choose $\de_1$ small enough so that $y_*+\de_1< -\sqrt2-2\eps_1$ if  $y_* < -\sqrt2$ or $y_*-\de_1 > \sqrt2+2\eps_1$ if $y_* > \sqrt2$. According to our choice, if $x\in(y_*-\de_1, y_*+\de_1)$, then $x\notin [-\sqrt2-2\eps_1,\sqrt2+2\eps_1]$.  With these considerations in mind, by restricting the empirical measure of GOE eigenvalues to $B_{\sqrt2+\eps_1}(\si_{{\rm sc}},\de)$ first, we find
\begin{align*}
&\ex(\de',\de_1)\ge  [-4D''(0)]^{\frac{N-1}{2}} \sqrt{\frac2\pi} \pz( L((\frac{N-1}{N})^{1/2}\la_1^{N-1})\in B_{\sqrt2+\eps_1}(\si_{{\rm sc}},\de)) \\
&\quad  \int_{\rho_*-\de'}^{\rho_*+\de'}\int_{u_*-\de'}^{u_*+\de'} \int_{y_*-\de_1}^{y_*+\de_1}  e^{(N-1) \inf_{\nu\in B_{\sqrt2+\eps_1} (\si_{\rm sc}, \de)} \Psi(\nu,y)} \exp\Big[-\frac{-4ND''(0)(y+\frac{m_2}{\sqrt{-4D''(0)}})^2}{2(-2D''(0)-\bt^2)} \Big] \\
&\quad  \Big[\frac{-4D''(0)}{N} +\frac{2D''(0)\al^2\rho^4}{N(-2D''(0)-\bt^2)} \Big]^{1/2}
\frac{\sqrt{N(-4D''(0))}}{\sqrt{2\pi (-2D''(0)-\bt^2)} } \\
&\quad   \frac{\sqrt{N} (2\pi)^{-(N+1)/2} D'(0)^{-N/2} }{\sqrt{D(\rho^2)-\frac{D'(\rho^2)^2 \rho^2}{D'(0)} }} \exp\Big(- \frac{N (u - \frac{\mu\rho^2}2+\frac{\mu D'(\rho^2) \rho^2}{D'(0)})^2}{2(D(\rho^2)-\frac{D'(\rho^2)^2 \rho^2}{D'(0)})}\Big)  e^{-\frac{N \mu^2 \rho^2}{2D'(0)}} \rho^{N-1}     \dd y \dd u  \dd\rho.
\end{align*}
Since $\Psi(\nu, y)$ is continuous in $\px[-\sqrt2-\eps_1,\sqrt2+\eps_1] \times (-\sqrt2-2\eps_1, \sqrt2+2\eps_1)^c$, we have
\begin{align*}
\lim_{\de\to0+}\inf_{\nu\in B_{\sqrt2+\eps_1} (\si_{\rm sc},\de)} \Psi(\nu,y) & = \Psi_*( y)
\end{align*}
for all $y\in[y_*-\de_1, y_*+\de_1]$. By Wigner's semicircle law with the distance \pref{eq:measd} and the LDP of the largest eigenvalue of GOE, we have
\begin{align*}
&\liminf_{N\to\8 }\pz( L((\frac{N-1}{N})^{1/2}\la_1^{N-1})\in B_{\sqrt2+\eps_1}(\si_{{\rm sc}},\de)) \\
&\ge \liminf_{N\to\8 } [\pz ( L((\frac{N-1}{N})^{1/2}\la_1^{N-1})\in B(\si_{{\rm sc}},\de))- \pz(\max_{i=1,...,N-1}|(\frac{N-1}{N})^{1/2}\la_i|>\sqrt2+\eps_1)]=1.
\end{align*}
Recall the function $\psi$ as in \eqref{eq:psidef}.  Since the functions in question are all continuous and thus attain strictly positive minimum in $\rho\in[\rho_*-\de',\rho_*+\de'], u\in[u_*-\de',u_*+\de'],y\in[y_*-\de_1, y_*+\de_1]$, using \pref{eq:snlim} and \pref{eq:krerr} we deduce  that
\begin{align}
\liminf_{N\to\8} &\frac1N\log \ez\Crt_{N}(E,(R_1, R_2))  \ge  \liminf_{\de'\to0+, \atop\de_1\to0+}  \liminf_{N\to\8} \frac1N\log \ex(\de',\de_1) + \frac12+\frac12\log(2\pi) \notag\\
&\ge  \frac12 +\frac12 \log [-4D''(0)] -\frac12 \log D'(0) \notag\\
& \ \ + \liminf_{\de\to0+,\de'\to0+, \atop\de_1\to0+}  \inf_{\rho\in[\rho_*-\de',\rho_*+\de'],\atop u\in[u_*-\de',u_*+\de'], y\in[y_*-\de_1, y_*+\de_1] }[ \psi_*(\rho,u,y)-\Psi_*(y)+\inf_{\nu\in B_{\sqrt2+\eps_1} (\si_{\rm sc},\de)} \Psi(\nu,y)] \notag\\
&= \frac12 +\frac12 \log [-4D''(0)] -\frac12 \log D'(0) +\psi_*(\rho_*,u_*,y_*).\label{eq:lbdt1}
\end{align}

\emph{Case 2}: $y_*
\in[-\sqrt2,\sqrt2]$. In this case, we can choose $\de_1>0$ small such that %$(y_*-\de_1, y_*+\de_1)\subset (-\sqrt2-\eps_0, \sqrt2+\eps_0)$ for some $\eps_0>0$ and
$G(\de_1):=(y_*-\de_1, y_*+\de_1)\cap (-\sqrt2,\sqrt2) \neq \emptyset$.  %Note that
Choosing $K$ large we find
\begin{align*}
& \int_{G(\de_1)}	\ez[e^{(N-1)\Psi(L((\frac{N-1}{N})^{1/2}\la_1^{N-1}), y)}] \exp\Big[-\frac{-4ND''(0)(y+\frac{m_2}{\sqrt{-4D''(0)}})^2}{2(-2D''(0)-\bt^2)} \Big]  \dd y\\
&\ge \frac1{Z'_{N-1}}\int_{G(\de_1)} \int_{[-(\frac{N}{N-1})^{1/2}K,(\frac{N}{N-1})^{1/2}K]^{N-1}}  \exp\Big[-\frac{-4ND''(0)(y +\frac{m_2}{\sqrt{-4D''(0)}})^2}{2(-2D''(0)-\bt^2)} \Big]   \\
&\qquad \prod_{i=1}^{N-1}|(\frac{N-1}{N})^{1/2} x_i-y| \prod_{1\le i<j\le N-1} |x_i-x_j| e^{-\frac{N-1}{2}\sum_{i=1}^{N-1}x_i^2} \prod_{i=1}^{N-1} \dd x_i \dd y\\
&\stackrel{(\frac{N-1}{N})^{1/2} x_i\mapsto x_i}{\scalebox{7}[1]{=}} \frac1{Z'_{N-1}} \Bigl(\frac{N}{N-1}\Bigr)^{\frac{N(N-1)}{4}}\int_{x_N\in G(\de_1)} \exp\Big[-\frac{-4ND''(0)(x_N+ \frac{m_2}{\sqrt{-4D''(0)}})^2}{2(-2D''(0)-\bt^2)} \Big] \\
&\qquad  \int_{[-K,K]^{N-1}}  \prod_{1\le i<j\le N} |x_i-x_j| e^{-\frac{N}2\sum_{i=1}^N x_i^2}  e^{\frac{N}{2} x_N^2} \prod_{i=1}^N \dd x_i\\
&\ge \frac{Z'_N}{Z'_{N-1}}\frac{1}{Z'_N} \Big(\frac{N}{N-1}\Big)^{\frac{N(N-1)}{4}}  \exp\Big[ N \min_{x\in G(\de_1)}\Big(\frac{ x ^2}{2}   -\frac{-4 D''(0)(x +\frac{m_2}{\sqrt{-4D''(0)}})^2}{2(-2D''(0)-\bt^2)}  \Big)\Big] \\
&\qquad  \int_{x_N\in  G(\de_1)} \int_{[-K,K]^{N-1}} \prod_{1\le i< j\le N} |x_i-x_j| e^{-\frac{N}2\sum_{i=1}^N x_i^2 }\prod_{i=1}^N \dd x_i \\
&= \frac{Z'_N}{Z'_{N-1}} \Big(\frac{N}{N-1}\Big)^{\frac{N(N-1)}{4}}  \exp\Big[ N \min_{x\in G(\de_1)}\Big(\frac{ x ^2}{2}   -\frac{-4ND''(0)(x+\frac{m_2}{\sqrt{-4D''(0)}})^2}{2(-2D''(0)-\bt^2)}  \Big)\Big] \\
&\qquad \ez\Big[\frac1N \#\{ i\in[N]: \tilde \la_i^N \in  G(\de_1)\} \indi\{\max_{i=1,...,N} |\tilde\la_i^N| \le  K\}\Big].
\end{align*}
Here $Z_N'=N!Z_N$ is the normalizing constant for the p.d.f.~of unordered eigenvalues of $\GOE_N$ matrix. By Stirling's formula,
\[
\lim_{N\to\8} \frac1N\log \Big[\frac{Z'_N}{Z'_{N-1}} \Big(\frac{N}{N-1}\Big)^{\frac{N(N-1)}{4}}\Big]=-\frac12-\frac12\log2.
\]
From Wigner's semicircle law we deduce
\begin{align*}
&\liminf_{N\to\8} \frac1N\log\ez\Big[\frac1N \#\{ i\in[N]: \tilde \la_i^N \in G(\de_1)\} \indi\{\max_{i=1,...,N} |\la_i| \le K\}\Big]\\
& = \lim_{N\to\8}\frac1N\log \si_{\rm sc}[G(\de_1)]=0.
\end{align*}
Since the functions in question are all continuous and thus attains strictly positive minimum in $\rho\in[\rho_*-\de',\rho_*+\de'], u\in[u_*-\de',u_*+\de'],y\in[y_*-\de_1, y_*+\de_1]$, using \pref{eq:snlim}  and \pref{eq:krerr} we deduce that
\begin{align}
\liminf_{N\to\8} &\frac1N\log \ez\Crt_{N}(E,(R_1, R_2))  \ge  \liminf_{\de'\to0+, \atop\de_1\to0+}  \liminf_{N\to\8} \frac1N\log \ex(\de',\de_1) + \frac12+\frac12\log(2\pi) \notag \\
&\ge  \frac12 +\frac12 \log [-4D''(0)] -\frac12 \log D'(0)-\frac12-\frac12\log2 + \liminf_{\de'\to0+, \atop\de_1\to0+}  \inf_{\rho\in[\rho_*-\de',\rho_*+\de'],\atop u\in[u_*-\de',u_*+\de'], x\in G(\de_1)}    \notag \\
& \quad  \Big[ \frac{ x ^2}{2} -\frac{-4ND''(0)(x +\frac{m_2}{\sqrt{-4D''(0)}})^2}{2(-2D''(0)-\bt^2)} -\frac{(u-\frac{\mu\rho^2}{2}+\frac{\mu D'(\rho^2) \rho^2}{D'(0) } )^2}{ 2\Big(D(\rho^2)-\frac{D'(\rho^2)^2 \rho^2}{D'(0)} \Big)} -\frac{\mu^2\rho^2}{2D'(0)}+\log \rho  \Big]\notag \\
&= \frac12 +\frac12 \log [-4D''(0)] -\frac12 \log D'(0) +\psi_*(\rho_*,u_*,y_*).\label{eq:lbdt2}
\end{align}
Here in the last step, we used the fact \pref{eq:phi*} that $\Psi_*(y_*) = \frac12 y_*^2-\frac12-\frac12\log2$ as $y_*\in[-\sqrt2,\sqrt2]$.
\end{proof}

%We restate it here with the explicit functions.
%
% Together with \pref{eq:snlim} we have proved the following:
% \begin{theorem}\label{th:cpx4}
%Let $0\le R_1<R _2\le \8$ and $E$ be an open set of $\rz$.  	Suppose Assumptions I, II, IV hold and $|\mu|+\frac1{R_2} > 0$.  Then
% 	\[
% 	\lim_{N\to\8} \frac1N \log\ez \Crt_{N}(E, (R_1,R_2)) = \frac12 \log[-4D''(0)] -\frac12\log D'(0) +\frac12+\sup_{(\rho,u,y)\in F}\psi_*(\rho,u,y)
% 	\]
% 	where  $F=\{(\rho,u,y):y\in\rz, \rho \in (R_1, R_2),t\in  \bar E\}$ and $\psi_*$ is given as in \pref{eq:psifunction}.
% \end{theorem}
 \begin{proof}[Proof of \pref{th:cpsublevel}]
 	If $\bar E$ is compact and $0\le R_1<R_2<\8$, the assertion follows from \pref{eq:snlim}, Propositions \ref{pr:ubi2}, \ref{pr:ubi1} and \ref{pr:lbd*}.
 	
 	%Suppose $\bar E$ is compact and $0=R_1<R_2<\8$. Note that $\psi_*(\rho,u,y ) \to -\8$ as $\rho\to \8$. Then the maximizer $(\rho_*,t_*,y_*)$ of $\psi_*(\rho,u,y)$ on $F$ is the same as that on $\{(\rho,u,y):y\in\rz, \rho \in[\eps, R_2],t\in  \bar E\}$ for all $\eps>0$ small enough. On one hand, we have
% 	\begin{align*}
% 	\liminf_{N\to\8} \frac1N \log\ez \Crt_{N}(E, [0,R_2])& \ge \liminf_{N\to\8} \frac1N \log\ez \Crt_{N}(E, [\eps ,R_2])\\
% 	&=\frac12 \log[-4D''(0)] -\frac12\log D'(0) +\frac12+\sup_{(\rho,u,y)\in F}\psi_*(\rho,u,y).
% 	\end{align*}
% 	On the other hand, for any $\eps'>0$, we may find $N_0(\eps')>0$ such that for $N>N_0(\eps')$ we have
% 	\begin{align*}
% 	\frac1N \log\ez \Crt_{N}(E, [\eps,R_2]) < \frac12 \log[-4D''(0)] -\frac12\log D'(0) +\frac12+\sup_{(\rho,u,y)\in F}\psi_*(\rho,u,y)+\eps'
% 	\end{align*}
% 	and the right-hand side does not depend on $\eps$. Using monotone convergence theorem, and sending $\eps\to 0$, $N\to\8$ and $\eps'\to 0$ sequentially, we find
% 	\[
% 	\limsup_{N\to\8}\frac1N \log\ez \Crt_{N}(E, [\eps,R_2]) \le \frac12 \log[-4D''(0)] -\frac12\log D'(0) +\frac12+\sup_{(\rho,u,y)\in F}\psi_*(\rho,u,y).
% 	\]
 	Suppose $\bar E$ is not compact or $R_2=\8$.
 Thanks to Lemmas \ref{le:psirho0} and \ref{le:exptt}, we may choose $R<\8$ and $T<\8$ large enough such that
 	\begin{align*}
 	  &\lim_{N\to\8} \frac1N \log\ez \Crt_{N}(E, (R_1,R_2)) =	\lim_{N\to\8} \frac1N \log\ez \Crt_{N}(E \cap(-T,T), (R_1,R_2)\cap [0,R])\\
 &=\frac12 \log[-4D''(0)] -\frac12\log D'(0) +\frac12+\sup_{y\in \rz, R_1< \rho<R\wedge R_2, u\in \bar E\cap [-T, T], }\psi_*(\rho,u,y)\\
 &=\frac12 \log[-4D''(0)] -\frac12\log D'(0) +\frac12+\sup_{(\rho,u,y)\in F}\psi_*(\rho,u,y),
 	\end{align*}
  which completes the proof.
 \end{proof}

We finish this section by showing how to recover Theorem \ref{th:ttcpx} from Theorem \ref{th:cpsublevel} when the domain of field is confined in a shell.
\begin{example}\label{ex:2}
	\rm
Let $0\le R_1<R_2\le\8$ and $E=\rz$. This removes restriction on the range of the random field. Let $J=\sqrt{-2D''(0)}$. Using \pref{eq:msialbt} and \pref{eq:uvcov},
%\[
%v=\frac{u-\frac{\mu\rho^2}{2}+\frac{\mu D'(\rho^2) \rho^2}{D'(0) } }{\sqrt{D(\rho^2)-\frac{D'(\rho^2)^2 \rho^2}{D'(0)} }}, \ \ \bt=\frac{D'(\rho^2)-D'(0)}{\sqrt{D(\rho^2)-\frac{D'(\rho^2)^2 \rho^2}{D'(0)} }},
%\]
we rewrite
\begin{align*}
\psi_*(\rho,u,y )= \Psi_*(y) -\frac{J^2}{J^2-\bt^2} \Big(y+\frac{\mu}{\sqrt2J}+\frac{\bt v}{\sqrt2 J}\Big)^2 -\frac{v^2}{2} -\frac{\mu^2\rho^2}{2D'(0)}+\log \rho.
\end{align*}
From \pref{eq:phi*}, we calculate
\begin{align*}
\partial_y \psi_*&=\frac{-(\bt^2+J^2)y-\sqrt2J (\mu+\bt v)}{J^2-\bt^2}-\sgn(y)\sqrt{y^2-2}\indi\{|y|>\sqrt2\},\\
\partial_{yy}\psi_*&=-\frac{J^2+\bt^2}{J^2-\bt^2} -\frac{|y|}{\sqrt{y^2-2}}\indi\{|y|>\sqrt2\}, \\
\partial_v\psi_*&= \frac{-J^2v -\bt (\sqrt2Jy +\mu )}{J^2-\bt^2},\  \
\partial_{yv}\psi_* =-\frac{\sqrt2 J\bt}{J^2-\bt^2},\  \
\partial_{vv}\psi_* = -\frac{J^2}{J^2-\bt^2}.
\end{align*}
Using the relation $\partial_v\psi_*=0$ we find
\begin{align}\label{eq:musbt}
    v=-\frac{\bt(\sqrt2Jy+\mu)}{J^2}, \ \ \sqrt2Jy+\mu+ \bt v = \frac{(\sqrt2Jy +\mu)(J^2-\bt^2)}{J^2}.
\end{align}
Together with \pref{eq:phi*}, we can eliminate $v$ and simplify
\begin{align}\label{eq:psids0}
  \psi_*(\rho,u,y )= -\frac12y^2-\frac12-\frac12\log2 -J_1(-|y|)\indi\{|y|>\sqrt2\}-\frac{\sqrt2 \mu y}{J}-\frac{\mu^2}{2J^2}-\frac{\mu^2\rho^2}{2D'(0)}+\log \rho.
\end{align}

\emph{Case 1}: $\mu\neq0$.  Solving $\partial_y \psi_*=0, \partial_v\psi_*=0$ gives (after removing an extraneous solution)
\begin{align*}
\begin{cases}
y=-\frac{\sqrt2 \mu}{J},\ \ v=\frac{\mu \bt}{J^2}, & |\mu|\le J,\\
y=-\frac1{\sqrt2}(\frac\mu{J}+\frac{J}\mu), \ \ v=\frac\bt\mu, &|\mu|>J.
\end{cases}
\end{align*}
From \pref{eq:asmp1} we know $J^2-\bt^2>0$ for $\rho>0$. By the second derivative test, this critical point is the unique global maximum. Moreover, plugging in the critical point reveals that
$$\Psi_*(y) -\frac{J^2}{J^2-\bt^2} \Big(y+\frac{\mu}{\sqrt2J}+\frac{\bt v}{\sqrt2 J}\Big)^2 -\frac{v^2}{2}$$ does not depend on $\rho$. As a result, we choose $\rho$ by optimizing $-\frac{\mu^2\rho^2}{2D'(0)}+\log \rho$. Let us consider $R_1<\sqrt{D'(0)}/|\mu|$ only; the other case is similar. Choose
\begin{align}\label{eq:rho*}
\rho_*=\begin{cases}
  \frac{\sqrt{D'(0)}}{|\mu|},& \text{ if } R_2>\frac{\sqrt{D'(0)}}{|\mu|},\\
  R_2,&\text{ otherwise}.
\end{cases}
\end{align}
If $|\mu|\le \sqrt{-2D''(0)}$, we take $y_*=-\mu/\sqrt{-D''(0)}$,  and 	
\begin{align}\label{eq:u*1}
u_*= \frac{\mu[D'(\rho_*^2)-D'(0)]}{-2D''(0)}+\frac{\mu\rho_*^2}{2}-\frac{\mu D'(\rho_*^2) \rho_*^2}{D'(0) } .
\end{align}
Then we find
\begin{align}\label{eq:psi*1}
	\psi_*(\rho_*,u_*,y_* )=\begin{cases}
  \frac{\mu^2}{-4D''(0)}-1-\frac12\log2 +\frac12\log D'(0)-\log |\mu|, &\mbox{if }  R_2>\frac{\sqrt{D'(0)}}{|\mu|},\\
  \frac{\mu^2}{-4D''(0)}-\frac12-\frac12\log2 +\log R_2- \frac{\mu^2R_2^2}{2D'(0)}, & \text{otherwise}.
\end{cases}
\end{align}
If $|\mu|>\sqrt{-2D''(0)}$, we take $y_*=-\frac{\mu}{\sqrt{-4D''(0)}} -\frac{\sqrt{-D''(0)}}{\mu}$,
\begin{align}\label{eq:u*2}
u_*=\frac{D'(\rho_*^2)-D'(0)}{\mu }+\frac{\mu\rho_*^2}{2}-\frac{\mu D'(\rho_*^2) \rho_*^2}{D'(0) }.
\end{align}
Then we find
\begin{align}\label{eq:psi*2}
&\psi_*(\rho_*,u_*,y_*) \notag\\
&=
\begin{cases}
  -\frac12\log2-\log\sqrt{-2D''(0)}-\frac12+\frac12\log D'(0), & \mbox{if } R_2>\frac{\sqrt{D'(0)}}{|\mu|}, \\
  -\frac12\log2-\log\sqrt{-2D''(0)}+\log|\mu| +\log R_2- \frac{\mu^2R_2^2}{2D'(0)}, & \mbox{otherwise}.
\end{cases}
\end{align}
Since $B_N=\{x\in \rz^N: \sqrt N R_1< \|x\|<\sqrt N R_2\}$, using Cramer's theorem for the chi-square distribution, we have
\[
-\Xi=\begin{cases}
      -\frac{\mu^2 R_2^2}{2D'(0)}+\frac12+\log R_2+\log |\mu|-\frac12\log D'(0) , & \mbox{if }  R_2<\frac{\sqrt{D'(0)}}{|\mu|},\\
       0, & \mbox{otherwise}.
     \end{cases}
\]
where $\Xi$ is defined as in \pref{eq:bnasp1}.

\emph{Case 2}: $\mu=0$. We have to assume $R_2<\8$. Then the above computations show that $\psi_*$ is optimized at $y_*=u_*=0$ and $\rho_*=R_2$ which gives
\[
	\lim_{N\to\8} \frac1N \log\ez \Crt_{N}(\rz, (R_1,R_2)) = \frac12 \log[-2D''(0)] -\frac12\log D'(0) +\log R_2.
\]
In addition, $\Theta=\lim_{N\to\8} \frac1N \log|B_N|= \log R_2+\frac12\log(2\pi)+\frac12$.

Our results here match all the three cases in \pref{th:ttcpx}. Therefore, this example explains the seemingly very different forms of the three phases, whose origin is hard to understand without the general \pref{th:cpsublevel}. Moreover, this suggests that the critical points around the value $u_*$ and variable $\rho_*$ given above dominate all other places.
%we take $x=\sqrt2$, $\rho=\sqrt{D'(0)}/\mu$ and
%\[
%t=\frac{(2\sqrt{-2D''(0)}-\mu)[D'(\rho^2)-D'(0)]}{-2D''(0)}+\frac{\mu\rho^2}{2}-\frac{\mu D'(\rho^2) \rho^2}{D'(0) } .
%\]
%Then we have
%\[
%\psi_*(\rho,u,y )=-\Big(\frac{\mu}{\sqrt{-4D''(0)}}-\sqrt2\Big)^2 + \frac12-\frac12\log2 +\frac12\log D'(0)-\log \mu.
%\]
%These results match that in \pref{th:ttcpx} for $\mu\neq0$ and $\Xi=0$.
\end{example}

\begin{appendix}
\section{Covariance function and its derivatives}
Let $D_N(r)=D(r/N)$. For $x,y\in \rz^N$, let $\varphi(x,y) =\frac12(D_N(\|x\|^2)+D_N(\|y\|^2)-D_N(\|x-y\|^2))$.
Under $X_N(0)=0$, isotropic increments imply that $\ez X_N(x)=0$; see \cite{Ya87}*{p.439}. We have
\[
\Cov[H_N(x), H_N(y)]=\Cov[X_N(x),X_N(y)]=\ez[X_N(x)X_N(y)]=\varphi(x,y).
\]
\begin{lemma}\label{le:cov}
  Assume Assumptions I and II. Then for $x\in \rz^N$,
  \begin{align*}
    \Cov[H_N(x), \partial_i H_N(x)]&= D'\left(\frac{\|x\|^2}N\right)x_i,\\
  \Cov[\partial_i H_N(x),\partial_j H_N(x)] &= D'(0)\de_{ij},\\
  \Cov[H_N(x),\partial_{ij} H_N(x)]&= 2D''\left(\frac{\|x\|^2}N\right)\frac{x_ix_j}N +\left[D'\left(\frac{\|x\|^2}N\right)-D'(0)\right]\de_{ij}\\
  \Cov[\partial_k H_N(x), \partial_{ij} H_N(x)]&= 0,\\
  \Cov[\partial_{lk} H_N(x), \partial_{ij} H_N(x)]&= -2D''(0)[\de_{jl}\de_{ik}+\de_{il}\de_{kj} +\de_{kl}\de_{ij}]/N,
  \end{align*}
  where $\de_{ij}$ are the Kronecker delta function.
\end{lemma}
\begin{proof}
By \cite{AT07}*{Theorem 1.4.2}, $X_N(x)$ is smooth. We can differentiate inside expectation as in \cite{AT07}*{(5.5.4)} and find
\begin{align*}
  \ez[X_N(x)\partial_i X_N(y)]/N& = \partial_{y_i} \ez(X_N(x) X_N(y))/N=D_N'(\|y\|^2)y_i +D_N'(\|x-y\|^2)(x_i-y_i),\\
  \ez[\partial_i X_N(x)\partial_j X_N(y)]/N&= \partial_{x_i} [ D_N'(\|x-y\|^2)(x_j-y_j)]\\
  &=2D_N''(\|x-y\|^2)(x_i-y_i)(x_j-y_j)+ D_N'(\|x-y\|^2)\de_{ij},\\
  \ez[ X_N(x)\partial_{ij} X_N(y)]/N&= \partial_{y_i} [ D_N'(\|y\|^2) y_j +D_N'(\|x-y\|^2) (x_j-y_j)]\\
  &=2D_N''(\|y\|^2)y_i y_j+ D_N'(\|y\|^2)\de_{ij} -2D_N''(\|x-y\|^2)(x_i-y_i)(x_j-y_j) \\ &\hspace{5ex}  -D_N'(\|x-y\|^2)\de_{ij},\\
  \ez[\partial_k X_N(x)\partial_{ij} X_N(y)]/N& =-4D_N'''(\|x-y\|^2)(x_k-y_k)(x_i-y_i)(x_j-y_j) \\
  &\hspace{-19ex} -2D_N''(\|x-y\|^2)(x_j-y_j)\de_{ki} -2D_N''(\|x-y\|^2)(x_i-y_i)\de_{kj} -2D_N''(\|x-y\|^2)(x_k-y_k)\de_{ij},\\
  \ez[\partial_{lk} X_N(x)\partial_{ij} X_N(y)]/N & =-8D_N^{(4)}(\|x-y\|^2)(x_l-y_l)(x_k-y_k)(x_i-y_i)(x_j-y_j) \\
  &\hspace{-19ex} -4D_N'''(\|x-y\|^2)[(x_i-y_i)(x_j-y_j)\de_{kl} +(x_k-y_k)(x_j-y_j)\de_{il}+(x_k-y_k)(x_i-y_i)\de_{jl}\\
  & +(x_l-y_l)(x_j-y_j)\de_{ki} +(x_l-y_l)(x_i-y_i)\de_{kj} +(x_l-y_l)(x_k-y_k)\de_{ij}]\\
  & -2D_N''(\|x-y\|^2)[\de_{jl}\de_{ik}+\de_{il}\de_{kj} +\de_{kl}\de_{ij}].
\end{align*}
Substituting $x=y$,
\begin{align*}
  \ez[X_N(x)\partial_i X_N(x)]/N&=D'_N(\|x\|^2)x_i,\\
  \ez[\partial_i X_N(x)\partial_j X_N(x)]/N &= D'_N(0)\de_{ij},\\
  \ez[X_N(x)\partial_{ij} X_N(x)]/N&= 2D_N''(\|x\|^2)x_ix_j +D_N'(\|x\|^2)\de_{ij}-D'_N(0)\de_{ij}\\
  \ez[\partial_k X_N(x)\partial_{ij} X_N(x)]/N&= 0,\\
  \ez[\partial_{lk} X_N(x)\partial_{ij} X_N(x)]/N&= -2D_N''(0)[\de_{jl}\de_{ik}+\de_{il}\de_{kj} +\de_{kl}\de_{ij}].
\end{align*}
Then we note that $D'_N(r)=D'(r/N)/N$ and $D_N''(r)=D''(r/N)/N^2$.
\end{proof}

\section{Auxiliary Lemmas}\label{se:aux}

For the integral $\ez \int_\rz \exp\big(-\frac12(N+1)x^2 -\frac{\sqrt{N(N+1)}\mu x}{\sqrt{-D''(0)}}\big)  L_{N+1}(\dd x)$, we have the following elementary fact which is used in \pref{se:whole}.
\begin{lemma}\label{le:repl}
Let $\nu_N$ be probability measures on $\rz$ and $\mu\neq0$. Suppose
$$\lim_{N\to\8} \frac1N\log \int_\rz e^{-\frac12(N+1)x^2 -\frac{(N+1)\mu x}{\sqrt{-D''(0)}} }\nu_{N+1}(\dd x) >-\8.$$
Then we have
\begin{align*}
\lim_{N\to \8} \frac1N \Big(\log \int_\rz e^{-\frac12(N+1)x^2 -\frac{\sqrt{N(N+1)}\mu x}{\sqrt{-D''(0)}} } \nu_{N+1}(\dd x)
- \log \int_\rz e^{-\frac12(N+1)x^2 -\frac{(N+1)\mu x}{\sqrt{-D''(0)}} }\nu_{N+1}(\dd x)\Big)=0.
\end{align*}
\end{lemma}
\begin{proof}
Let
\begin{align*}
 a_N &=  \int_\rz e^{-\frac12(N+1)x^2 -\frac{(N+1)\mu x}{\sqrt{-D''(0)}} }\nu_{N+1}(\dd x),\\ %a_N^t= \int_{-t}^t e^{-\frac12(N+1)x^2 -\frac{(N+1)\mu x}{\sqrt{-D''(0)}} }\rho_{N+1}(x) \dd x,\\
 b_N &=  \int_\rz e^{-\frac12(N+1)x^2 -\frac{\sqrt{N(N+1)}\mu x}{\sqrt{-D''(0)}} } \nu_{N+1}(\dd x) ,
\\ %b_N^t =  \int_{-t}^t e^{-\frac12(N+1)x^2 -\frac{\sqrt{N(N+1)}\mu x}{\sqrt{-D''(0)}} } \rho_{N+1}(x) \dd x,\\
c_N&=\int_\rz e^{-\frac12Nx^2 -\frac{N\mu x}{\sqrt{-D''(0)}} }\nu_{N+1}(\dd x).% c_N^t= \int_{-t}^t e^{-\frac12Nx^2 -\frac{N\mu x}{\sqrt{-D''(0)}} }\rho_{N+1}(x) \dd x.
\end{align*}
We claim $\lim_{N\to \8} \frac1N \log \frac{ a_N}{c_{N}}=0$. Indeed, by Jensen's inequality,
\begin{align*}
  \log\frac{ c_{N}}{ a_{N}}\le \log \frac{ a_N^{N/(N+1)}}{ a_{N}}=-\frac1{N+1}\log  a_N.
\end{align*}
But
\[
 a_N =\int_\rz e^{-\frac12(N+1)(x+\frac{\mu}{\sqrt{-D''(0)}})^2 +\frac{(N+1)\mu^2}{-2D''(0)}} \nu_{N+1}(\dd x)\le e^{\frac{(N+1)\mu^2}{-2D''(0)}}.
\]
Then the claim follows. From the elementary inequality $a\wedge b \le (a+b)/2 \le a\vee b$, we have $\lim_{N\to\8} \frac1N (\log ( a_{N} + c_{N})-\log  a_N )=0$. It remains to prove that
\[
\lim_{N\to\8} \frac1N (\log ( a_{N} +  c_{N}) - \log  b_N)=0.
\]
Note that
\begin{align*}
  b_N & \le \int_{-\8}^{0} e^{-\frac12(N+1)x^2 -\frac{(N+1)\mu x}{\sqrt{-D''(0)}} } \nu_{N+1}(\dd x) + \int_{0}^{\8} e^{-\frac12 Nx^2 -\frac{N\mu x}{\sqrt{-D''(0)}} } \nu_{N+1}(\dd x) \le a_N + c_N.
\end{align*}
%Note that
%\[
%\log (2b_N) -\log ( a_{N} + c_N) \le \log \Big(1+\frac{| b_N- a_N |+|  b_N-  c_{N} |}{ a_{N}+ c_N} \Big).
%\]
%Let $W_{N+1}$ be an $(N+1)\times (N+1)$ GOE matrix. Then we know from \cite{BDG01} that
%\begin{align}
%\pz(\|W_{N+1}\|_{\rm op}> t) \le e^{-(N+1)t^2/9}.
%\end{align}
Let $t $ be a large constant (independent of $N$) such that
\[
\lim_{N\to\8} \frac1N\log \int_\rz e^{-\frac12(N+1)x^2 -\frac{(N+1)\mu x}{\sqrt{-D''(0)}} }\nu_{N+1}(\dd x) > - \frac{t^2}{8}
\]
and that
\[
\int_{|x|>t} e^{-\frac12(N+1)x^2 -\frac{(N+1)\mu x}{\sqrt{-D''(0)}} } \nu_{N+1}(\dd x) \le e^{-(N+1)t^2/4}.
\]
It follows that
\[
\int_{|x|>t}  e^{-\frac12 Nx^2 -\frac{N\mu x}{\sqrt{-D''(0)}} } \nu_{N+1}(\dd x) \le e^{-Nt^2/4},
\]
and since $\frac1N \log \frac{a_N}{c_N}\to 0$ as $N\to\8$,
\begin{align*}
&\lim_{N\to\8} \frac1N \log\int_{-t}^t e^{-\frac12 Nx^2 -\frac{N\mu x}{\sqrt{-D''(0)}} } \nu_{N+1}(\dd x)\\
&= \lim_{N\to\8} \frac1N \log \int_{-\8}^\8 (1-\indi\{|x|>t\})e^{-\frac12 Nx^2 -\frac{N\mu x}{\sqrt{-D''(0)}} } \nu_{N+1}(\dd x)\\
&=\lim_{N\to\8} \frac1N \log \int_{-\8}^\8  e^{-\frac12 Nx^2 -\frac{N\mu x}{\sqrt{-D''(0)}} } \nu_{N+1}(\dd x).
\end{align*}
Note that
\begin{align*}
b_N &\ge e^{-\frac{t^2}2}\int_{-t}^{0} e^{-\frac12 Nx^2 -\frac{N\mu x}{\sqrt{-D''(0)}} } \nu_{N+1}(\dd x) + e^{-\frac{t^2}2 -\frac{\mu t}{\sqrt{-D''(0)}}}\int_{0}^{t} e^{-\frac12 Nx^2 -\frac{N\mu x}{\sqrt{-D''(0)}} } \nu_{N+1}(\dd x)\\
&\ge  e^{-\frac{t^2}2 -\frac{\mu t}{\sqrt{-D''(0)}}} \int_{-t}^t e^{-\frac12 Nx^2 -\frac{N\mu x}{\sqrt{-D''(0)}} } \nu_{N+1}(\dd x).
\end{align*}
Since
\[
\lim_{N\to\8} \frac1N\Big( \log (a_N+c_N)-  \log\int_{-t}^t e^{-\frac12 Nx^2 -\frac{N\mu x}{\sqrt{-D''(0)}} } \nu_{N+1}(\dd x)\Big)=0,
\]
we have $\lim_{N\to\8} \frac1N (\log ( a_{N} +  c_{N}) - \log  b_N)=0$.
%
%\begin{align*}
%  &|  b_N-  a_N |+|  b_N-  c_{N} | \le |b_N^t-a_N^t|+|b_N - c_{N}^t| \\
%  &+2\int_{-\8}^{-t} e^{-\frac12(N+1)x^2 -\frac{\sqrt{N(N+1)}\mu x}{\sqrt{-D''(0)}} } \rho_{N+1}(x) \dd x +2\int_{t}^{+\8} e^{-\frac12(N+1)x^2 -\frac{\sqrt{N(N+1)}\mu x}{\sqrt{-D''(0)}} } \rho_{N+1}(x) \dd x\\
%  & +\int_{-\8}^{-t} e^{-\frac12(N+1)x^2-\frac{(N+1)\mu x}{\sqrt{-D''(0)}} } \rho_{N+1}(x) \dd x +\int_{t}^{+\8} e^{-\frac12(N+1)x^2-\frac{(N+1)\mu x}{\sqrt{-D''(0)}} } \rho_{N+1}(x) \dd x\\
%  &+ \int_{-\8}^{-t} e^{-\frac12 Nx^2-\frac{N\mu x}{\sqrt{-D''(0)}} } \rho_{N+1}(x) \dd x +\int_{t}^{+\8} e^{-\frac12Nx^2-\frac{N\mu x}{\sqrt{-D''(0)}} } \rho_{N+1}(x) \dd x\\
%  &\le |b_N^t-a_N^t|+|b_N - c_{N}^t|  + 2a_N+ 2c_N + a_N+c_N.
%\end{align*}
%Note that by the mean value theorem
%\[
%|b_N^t-a_N^t| \le \int_{-t}^{t} e^{-\frac12(N+1)x^2}(e^{\frac{\mu x}{\sqrt{-D''(0)}}} - e^{\frac{\sqrt N \mu x}{\sqrt{-D''(0)(N+1)}}})
%\]
\end{proof}

The following discussion is about Assumption IV.

\begin{proof}[Proof of \pref{le:dgeab}]
    1. Since $y\mapsto D'(y)$ is a strictly decreasing convex function and $D'''(y)>0$ for any $y>0$, $|D''(y)|< \frac{D'(0)-D'(y)}y $.  By assumption,
    \[
    (\al \rho^2)^2 = \frac{4D''(\rho^2)^2 \rho^4}{D(\rho^2)-\frac{\rho^2 D'(\rho^2)^2}{D'(0)}} \le -\frac{8 D''(\rho^2)^2 D''(0)}{ 3[D'(\rho^2)-D'(0)]^2/\rho^4}< -\frac83 D''(0).
    \]
    It follows that
    \begin{align*}
    (\al \rho^2+\bt)\bt &< \sqrt{-\frac23 D''(0)} \sqrt{-\frac83 D''(0)}  -\frac23 D''(0) =-2D''(0),\\
    (\al \rho^2+\bt)\al \rho^2 &< -\frac83 D''(0)+ \sqrt{-\frac23 D''(0)} \sqrt{-\frac83 D''(0)}=-4D''(0).
    \end{align*}

    2. We verify \pref{eq:btbd}. If \pref{eq:btinc} holds, then $y\mapsto \bt(y)^2$ is a decreasing function and \pref{eq:btbd} follows from \pref{le:albtd}.

    3. By item 1, it suffices to check \pref{eq:btbd}. Consider the function
    \[
    f(y)=-D''(0)[D'(0)D(y)-D'(y)^2y]-\frac32D'(0)[D'(y)-D'(0)]^2.
    \]
    Condition \pref{eq:btbd} is equivalent to $f(y)\ge0$. Note that $f(0)=0$ and that
    \[
    f'(y)= [D'(0)-D'(y)][D'(0)D''(y)-D''(0)D'(y)]+2D''(y)(D''(0)D'(y)y -D'(0)[D'(y)-D'(0)]).
    \]
    By convexity, $  \frac{D'(y)-D'(0)}y \le  D''(y)\le 0$. If \pref{eq:btbd2} holds, $D''(0)D'(y)y -D'(0)[D'(y)-D'(0)]\le 0$ and
    \[
    \frac{D'(y)}{D'(0)}-\frac{D''(y)}{D''(0)}\ge 0.
    \]
    Then \pref{eq:btbd} follows from here since $D'(0)\ge D'(y)$ and we have $f'(y)\ge0$.

    4. By Cauchy's mean value theorem, condition \pref{eq:btbd2} is equivalent to \pref{eq:btbd3}.

    5. Direct calculation yields
    \[
    \frac{\dd}{\dd y} \frac{D'(y)}{- D''(y)} =\frac{-D''(y)^2+D'''(y)D'(y)}{D''(y)^2}.
    \]
    Then \pref{eq:btbd4} implies \pref{eq:btbd3}.

    6. By the representation \pref{eq:tbfcn} of Thorin--Bernstein functions, we have
    \begin{align*}
    D''(x)= -\int_{(0,\8)}  \frac1{(x+t)^2} \si(\dd t), \qquad D'''(x)= \int_{(0,\8)}  \frac{2}{(x+t)^3} \si(\dd t).
    \end{align*}
    By the Cauchy--Schwarz inequality, we have
    \begin{align*}
    2D''(x)^2\le D'(x)D'''(x).
    \end{align*}
    It follows that $\frac{\dd}{\dd y} \frac{D'(y)}{- D''(y)} \ge1$ and \pref{eq:btbd4} holds. %Therefore, \pref{eq:btbd2} holds.
    \end{proof}

    If $A=0$ in the representation \pref{eq:drep}, using the Cauchy--Schwarz inequality, we can see
    \[
    \frac{\dd}{\dd y} \frac{D'(y)}{- D''(y)} =\frac{-D''(y)^2+D'''(y)D'(y)}{D''(y)^2}\ge 0,
    \]
    compared with \pref{eq:btbd4}. It is easy to check that for any $\eps>0, 0<\ga<1$, our major examples $D(r)=\log(1+r/\eps)$ and $D(r)=(r+\eps)^\ga-\eps^{\ga}$ satisfy \pref{eq:btbd3}. With more work, one can check that these functions satisfy \pref{eq:btinc}.

    On the other hand, according to \cite{SSV}*{p.~332},
    $$D(x)=\frac{\sqrt{x}\sinh^2(\sqrt x)}{\sinh(2\sqrt x)}$$
    is  a complete  Bernstein function which is not Thorin--Bernstein.  One can check (at least numerically) that it violates \pref{eq:btbd3} but still verifies \pref{eq:btbd}.
    We suspect that \pref{eq:asmp1} and \pref{eq:asmp2} always hold for any structure function $D$. The following shows that this is the case at least in a neighborhood of $0$.

    \begin{lemma}
      Assume $A=0$ in \pref{eq:drep}. We have
      \begin{align*}
        \lim_{y\to 0+} \frac{\dd }{\dd y} [\al(y) y+\bt(y)]\bt(y) & <0, \\
        \lim_{y\to 0+} \frac{\dd }{\dd y} [\al(y) y+\bt(y)]\al(y)y & <0.
      \end{align*}
      Consequently, there exists $\de>0$ such that $-2D''(0)>[\al(y) y+\bt(y)]\bt(y)$ and $-4D''(0)>[\al(y) y+\bt(y)]\al(y)y$ for $y\in(0,\de)$.
    \end{lemma}
    \begin{proof}
      We only prove the first inequality as the second is similar. Write
      \[
      (\al y+\bt)\bt = \frac{[2D''(y)+\frac{D'(y)-D'(0)}{y}] \frac{D'(y)-D'(0)}{y}}{\frac{D(y)}{y^2}-\frac{D'(y)^2}{D'(0)y}} =:\frac{T}{B}.
      \]
      Since $[(\al y+\bt)\bt]'=\frac{T'B-B'T}{B^2}$ and $\lim_{y\to 0+} B=-\frac32 D''(0) \neq 0$, it suffices to show that $\lim_{y\to 0+}T'B-B'T<0$. By calculation, we have $\lim_{y\to0+} T= 3D''(0)^2$ and
      \begin{align*}
        T'&= [2D'''(y) +\frac{D''(y)y-D'(y) +D'(0)}{y^2}] \frac{D'(y)-D'(0)}{y}\\
        &\ \ +[2D''(y)+\frac{D'(y)-D'(0)}{y}] \frac{D''(y)y-D'(y) +D'(0)}{y^2},\\
        B'&=\frac{D'(0)D'(y)y-2D'(0)D(y)-2D'(y)D''(y)y^2+y D'(y)^2}{D'(0) y^3}.
      \end{align*}
      After some tedious computation, we find $\lim_{y\to0+} T'=4D'''(0) D''(0)$ and $\lim_{y\to0+} B'=-\frac56 D'''(0)-\frac{D''(0)^2}{D'(0)}$. Then
      \[
      \lim_{y\to 0+} T'B-B'T = D''(0)^2 \Big[\frac{3D''(0)^2}{D'(0)}-\frac72 D'''(0)\Big].
      \]
      By the Cauchy--Schwarz inequality,
      \[
      D''(0)^2 = \Big(\int_0^\8 t^4 \nu(\dd t) \Big)^2\le \int_0^\8 t^2 \nu(\dd t) \int_0^\8 t^6 \nu(\dd t)=D'(0)D'''(0).
      \]
      From here the conclusion follows.
    \end{proof}

\end{appendix}

\bibliographystyle{plain}
\bibliography{gficrev}

% \bib, bibdiv, biblist are defined by the amsrefs package.
\begin{bibdiv}
\begin{biblist}

\bib{ABA13}{article}{
      author={Auffinger, Antonio},
      author={Ben~Arous, Gerard},
       title={Complexity of random smooth functions on the high-dimensional
  sphere},
        date={2013Nov},
     journal={Ann. Probab.},
      volume={41},
      number={6},
       pages={4214\ndash 4247},
  url={https://urldefense.proofpoint.com/v2/url?u=https-3A__doi.org_10.1214_13-2DAOP862&d=DwIGaQ&c=KXXihdR8fRNGFkKiMQzstu-8MbOxd1NuZkcSBymGmgo&r=xgjcP5L9SZD7lyUnYxs7d005TSYbpcjMlIILZCrBUmY&m=ik53W52nlHKGpSmqYxuDPEAn4pA9Q4ERYBfJxF8OAyE&s=hM1xHbrjv-FKAI3ZMBx4SAhbApuDHWe_3As9gWmhuJo&e=},
}

\bib{ABC13}{article}{
      author={Auffinger, Antonio},
      author={Ben~Arous, G\'erard},
      author={Cerny, Jiri},
       title={Random matrices and complexity of spin glasses},
        date={2013},
        ISSN={1097-0312},
     journal={Communications on Pure and Applied Mathematics},
      volume={66},
      number={2},
       pages={165\ndash 201},
  url={https://urldefense.proofpoint.com/v2/url?u=http-3A__dx.doi.org_10.1002_cpa.21422&d=DwIGaQ&c=KXXihdR8fRNGFkKiMQzstu-8MbOxd1NuZkcSBymGmgo&r=xgjcP5L9SZD7lyUnYxs7d005TSYbpcjMlIILZCrBUmY&m=ik53W52nlHKGpSmqYxuDPEAn4pA9Q4ERYBfJxF8OAyE&s=BBRDT0iZjZVhejUwxVHrFRzS5R9p39n4GWbkJGi27LU&e=},
}

\bib{AT07}{book}{
      author={Adler, Robert},
      author={Taylor, Jonathan},
       title={Random fields and geometry},
      series={Springer Monographs in Mathematics},
   publisher={Springer, New York},
        date={2007},
        ISBN={978-0-387-48112-8},
      review={\MR{2319516}},
}

\bib{AZ22}{article}{
      author={{Auffinger}, Antonio},
      author={{Zeng}, Qiang},
       title={{Complexity of locally isotropic Gaussian random fields: critical
  points with given indices}},
        date={2022-06},
     journal={arXiv e-prints},
}

\bib{BDG01}{article}{
      author={Ben~Arous, G.},
      author={Dembo, A.},
      author={Guionnet, A.},
       title={Aging of spherical spin glasses},
        date={2001},
        ISSN={0178-8051},
     journal={Probab. Theory Related Fields},
      volume={120},
      number={1},
       pages={1\ndash 67},
  url={https://urldefense.proofpoint.com/v2/url?u=https-3A__doi.org_10.1007_PL00008774&d=DwIGaQ&c=KXXihdR8fRNGFkKiMQzstu-8MbOxd1NuZkcSBymGmgo&r=xgjcP5L9SZD7lyUnYxs7d005TSYbpcjMlIILZCrBUmY&m=ik53W52nlHKGpSmqYxuDPEAn4pA9Q4ERYBfJxF8OAyE&s=TyomswT8F3sVRP6PhcS9D0kK6tvHqEZLkr9ICj7fSRs&e=},
      review={\MR{1856194}},
}

\bib{BG97}{article}{
      author={Ben~Arous, G.},
      author={Guionnet, A.},
       title={Large deviations for {W}igner's law and {V}oiculescu's
  non-commutative entropy},
        date={1997},
        ISSN={0178-8051},
     journal={Probab. Theory Related Fields},
      volume={108},
      number={4},
       pages={517\ndash 542},
  url={https://urldefense.proofpoint.com/v2/url?u=https-3A__doi.org_10.1007_s004400050119&d=DwIGaQ&c=KXXihdR8fRNGFkKiMQzstu-8MbOxd1NuZkcSBymGmgo&r=xgjcP5L9SZD7lyUnYxs7d005TSYbpcjMlIILZCrBUmY&m=ik53W52nlHKGpSmqYxuDPEAn4pA9Q4ERYBfJxF8OAyE&s=lMHez1gAxmw-m4NU51PmIEWU1kcp8TgylcHp-TgcNG4&e=},
      review={\MR{1465640}},
}

\bib{Mihai}{article}{
      author={Ben~Arous, G\'erard},
      author={Mei, Song},
      author={Montanari, Andrea},
      author={Nica, Mihai},
       title={The landscape of the spiked tensor model},
        date={2019},
     journal={Communications on Pure and Applied Mathematics},
      volume={72},
      number={11},
       pages={2282\ndash 2330},
      eprint={https://onlinelibrary.wiley.com/doi/pdf/10.1002/cpa.21861},
         url={https://onlinelibrary.wiley.com/doi/abs/10.1002/cpa.21861},
}

\bib{BBMsd2}{article}{
      author={{Ben Arous}, Gerard},
      author={{Bourgade}, Paul},
      author={{McKenna}, Ben},
       title={{Exponential growth of random determinants beyond invariance}},
        date={2021-06},
     journal={arXiv e-prints},
}

\bib{BBMsd}{article}{
      author={{Ben Arous}, Gerard},
      author={{Bourgade}, Paul},
      author={{McKenna}, Ben},
       title={{Landscape complexity beyond invariance and the elastic
  manifold}},
        date={2021-06},
     journal={arXiv e-prints},
}

\bib{BD07}{article}{
      author={{Bray}, A.~J.},
      author={{Dean}, D.~S.},
       title={{Statistics of Critical Points of Gaussian Fields on
  Large-Dimensional Spaces}},
        date={2007-04},
     journal={Physical Review Letters},
      volume={98},
      number={15},
       pages={150201},
      eprint={cond-mat/0611023},
}

\bib{CL04}{article}{
      author={Crisanti, A.},
      author={Leuzzi, L.},
       title={Spherical $2+p$ spin-glass model: An exactly solvable model for
  glass to spin-glass transition},
        date={2004Nov},
     journal={Phys. Rev. Lett.},
      volume={93},
       pages={217203},
  url={https://urldefense.proofpoint.com/v2/url?u=https-3A__link.aps.org_doi_10.1103_PhysRevLett.93.217203&d=DwIGaQ&c=KXXihdR8fRNGFkKiMQzstu-8MbOxd1NuZkcSBymGmgo&r=xgjcP5L9SZD7lyUnYxs7d005TSYbpcjMlIILZCrBUmY&m=ik53W52nlHKGpSmqYxuDPEAn4pA9Q4ERYBfJxF8OAyE&s=QP_8qRXkXUGQ8T2zBpyA-pkOel3NSVRzNBvxI1acGY0&e=},
}

\bib{CS18}{article}{
      author={Cheng, Dan},
      author={Schwartzman, Armin},
       title={Expected number and height distribution of critical points of
  smooth isotropic {G}aussian random fields},
        date={2018},
        ISSN={1350-7265},
     journal={Bernoulli},
      volume={24},
      number={4B},
       pages={3422\ndash 3446},
  url={https://urldefense.proofpoint.com/v2/url?u=https-3A__doi.org_10.3150_17-2DBEJ964&d=DwIGaQ&c=KXXihdR8fRNGFkKiMQzstu-8MbOxd1NuZkcSBymGmgo&r=xgjcP5L9SZD7lyUnYxs7d005TSYbpcjMlIILZCrBUmY&m=ik53W52nlHKGpSmqYxuDPEAn4pA9Q4ERYBfJxF8OAyE&s=mupjrLOOPtNbc8f6p8pbW7Krw8foi2xb6Xr2ttIs94w&e=},
      review={\MR{3788177}},
}

\bib{En93}{article}{
      author={Engel, A.},
       title={Replica symmetry breaking in zero dimension},
        date={1993},
        ISSN={0550-3213},
     journal={Nuclear Physics B},
      volume={410},
      number={3},
       pages={617 \ndash  646},
  url={https://urldefense.proofpoint.com/v2/url?u=http-3A__www.sciencedirect.com_science_article_pii_055032139390531S&d=DwIGaQ&c=KXXihdR8fRNGFkKiMQzstu-8MbOxd1NuZkcSBymGmgo&r=xgjcP5L9SZD7lyUnYxs7d005TSYbpcjMlIILZCrBUmY&m=ik53W52nlHKGpSmqYxuDPEAn4pA9Q4ERYBfJxF8OAyE&s=Vpc_Kk0LuSFlCY-XbAiteWiykfh7ubBZieCY09WCxJ4&e=},
}

\bib{FB08}{article}{
      author={Fyodorov, Yan~V.},
      author={Bouchaud, Jean-Philippe},
       title={Statistical mechanics of a single particle in a multiscale random
  potential: {P}arisi landscapes in finite-dimensional {E}uclidean spaces},
        date={2008},
        ISSN={1751-8113},
     journal={J. Phys. A},
      volume={41},
      number={32},
       pages={324009, 25},
  url={https://urldefense.proofpoint.com/v2/url?u=https-3A__doi.org_10.1088_1751-2D8113_41_32_324009&d=DwIGaQ&c=KXXihdR8fRNGFkKiMQzstu-8MbOxd1NuZkcSBymGmgo&r=xgjcP5L9SZD7lyUnYxs7d005TSYbpcjMlIILZCrBUmY&m=ik53W52nlHKGpSmqYxuDPEAn4pA9Q4ERYBfJxF8OAyE&s=0r3zlQzs5ERk9BBYBjrJ4M0IVXiDC3Soe7FygXBCr_Y&e=},
      review={\MR{2425780}},
}

\bib{FN12}{article}{
      author={{Fyodorov}, Y.~V.},
      author={{Nadal}, C.},
       title={{Critical Behavior of the Number of Minima of a Random Landscape
  at the Glass Transition Point and the Tracy-Widom Distribution}},
        date={2012-10},
     journal={Physical Review Letters},
      volume={109},
      number={16},
       pages={167203},
      eprint={1207.6790},
}

\bib{FS07}{article}{
      author={Fyodorov, Yan~V.},
      author={Sommers, H.-J.},
       title={Classical particle in a box with random potential: exploiting
  rotational symmetry of replicated {H}amiltonian},
        date={2007},
        ISSN={0550-3213},
     journal={Nuclear Phys. B},
      volume={764},
      number={3},
       pages={128\ndash 167},
  url={https://urldefense.proofpoint.com/v2/url?u=https-3A__doi.org_10.1016_j.nuclphysb.2006.11.029&d=DwIGaQ&c=KXXihdR8fRNGFkKiMQzstu-8MbOxd1NuZkcSBymGmgo&r=xgjcP5L9SZD7lyUnYxs7d005TSYbpcjMlIILZCrBUmY&m=ik53W52nlHKGpSmqYxuDPEAn4pA9Q4ERYBfJxF8OAyE&s=pBYZ3HoZBLmBmjiTUOLkZJKwg7WjphqLqNIx53oTxuY&e=},
      review={\MR{2293452}},
}

\bib{FW07}{article}{
      author={Fyodorov, Yan~V.},
      author={Williams, Ian},
       title={Replica symmetry breaking condition exposed by random matrix
  calculation of landscape complexity},
        date={2007},
        ISSN={0022-4715},
     journal={J. Stat. Phys.},
      volume={129},
      number={5-6},
       pages={1081\ndash 1116},
  url={https://urldefense.proofpoint.com/v2/url?u=https-3A__doi.org_10.1007_s10955-2D007-2D9386-2Dx&d=DwIGaQ&c=KXXihdR8fRNGFkKiMQzstu-8MbOxd1NuZkcSBymGmgo&r=xgjcP5L9SZD7lyUnYxs7d005TSYbpcjMlIILZCrBUmY&m=ik53W52nlHKGpSmqYxuDPEAn4pA9Q4ERYBfJxF8OAyE&s=qJU_uTiSob65TI3kF6KxiMO_G42_1H1GkuH5F_54vHA&e=},
      review={\MR{2363390}},
}

\bib{Fy04}{article}{
      author={Fyodorov, Yan~V.},
       title={Complexity of random energy landscapes, glass transition, and
  absolute value of the spectral determinant of random matrices},
        date={2004},
        ISSN={0031-9007},
     journal={Phys. Rev. Lett.},
      volume={92},
      number={24},
       pages={240601, 4},
  url={https://urldefense.proofpoint.com/v2/url?u=https-3A__doi.org_10.1103_PhysRevLett.92.240601&d=DwIGaQ&c=KXXihdR8fRNGFkKiMQzstu-8MbOxd1NuZkcSBymGmgo&r=xgjcP5L9SZD7lyUnYxs7d005TSYbpcjMlIILZCrBUmY&m=ik53W52nlHKGpSmqYxuDPEAn4pA9Q4ERYBfJxF8OAyE&s=poRSJyQYA_9IdjBGzxWbTzBZ1HmEWZVsU3YEQhj-Tdc&e=},
      review={\MR{2115095}},
}

\bib{Kli12}{article}{
      author={Klimovsky, Anton},
       title={High-dimensional {G}aussian fields with isotropic increments seen
  through spin glasses},
        date={2012},
        ISSN={1083-589X},
     journal={Electron. Commun. Probab.},
      volume={17},
       pages={no. 17, 14},
  url={https://urldefense.proofpoint.com/v2/url?u=https-3A__doi.org_10.1214_ECP.v17-2D1994&d=DwIGaQ&c=KXXihdR8fRNGFkKiMQzstu-8MbOxd1NuZkcSBymGmgo&r=xgjcP5L9SZD7lyUnYxs7d005TSYbpcjMlIILZCrBUmY&m=ik53W52nlHKGpSmqYxuDPEAn4pA9Q4ERYBfJxF8OAyE&s=EKeVUYGZBClnvD-MIUVKFt2euBFnDuTA_hlqIy3crRU&e=},
      review={\MR{2915663}},
}

\bib{Ko41}{article}{
      author={Kolmogoroff, A.},
       title={The local structure of turbulence in incompressible viscous fluid
  for very large {R}eynold's numbers},
        date={1941},
     journal={C. R. (Doklady) Acad. Sci. URSS (N.S.)},
      volume={30},
       pages={301\ndash 305},
      review={\MR{0004146}},
}

\bib{MP91}{article}{
      author={{Marc M\'ezard}},
      author={{Giorgio Parisi}},
       title={Replica field theory for random manifolds},
        date={1991},
     journal={J. Phys. I France},
      volume={1},
      number={6},
       pages={809\ndash 836},
  url={https://urldefense.proofpoint.com/v2/url?u=https-3A__doi.org_10.1051_jp1-3A1991171&d=DwIGaQ&c=KXXihdR8fRNGFkKiMQzstu-8MbOxd1NuZkcSBymGmgo&r=xgjcP5L9SZD7lyUnYxs7d005TSYbpcjMlIILZCrBUmY&m=ik53W52nlHKGpSmqYxuDPEAn4pA9Q4ERYBfJxF8OAyE&s=oVrUw3pQBF7u-G1tv7FNIp4PIxhgl6XtthAsjewY3Wg&e=},
}

\bib{MPV}{book}{
      author={Mezard, M},
      author={Parisi, G},
      author={Virasoro, M},
       title={Spin glass theory and beyond},
   publisher={WORLD SCIENTIFIC},
        date={1986},
  url={https://urldefense.proofpoint.com/v2/url?u=https-3A__www.worldscientific.com_doi_abs_10.1142_0271&d=DwIGaQ&c=KXXihdR8fRNGFkKiMQzstu-8MbOxd1NuZkcSBymGmgo&r=xgjcP5L9SZD7lyUnYxs7d005TSYbpcjMlIILZCrBUmY&m=ik53W52nlHKGpSmqYxuDPEAn4pA9Q4ERYBfJxF8OAyE&s=-4HVs53ZmAR9JP6wi0YwzzJNwsoGmZrLsC3QGrsQJ90&e=},
}

\bib{Sch38}{article}{
      author={Schoenberg, I.~J.},
       title={Metric spaces and completely monotone functions},
        date={1938},
        ISSN={0003-486X},
     journal={Ann. of Math. (2)},
      volume={39},
      number={4},
       pages={811\ndash 841},
  url={https://urldefense.proofpoint.com/v2/url?u=https-3A__doi.org_10.2307_1968466&d=DwIGaQ&c=KXXihdR8fRNGFkKiMQzstu-8MbOxd1NuZkcSBymGmgo&r=xgjcP5L9SZD7lyUnYxs7d005TSYbpcjMlIILZCrBUmY&m=ik53W52nlHKGpSmqYxuDPEAn4pA9Q4ERYBfJxF8OAyE&s=2AiaaIuRneR5f_SYMHpzsgr0Y9kkwQwSmnnQ8N4JB_Y&e=},
      review={\MR{1503439}},
}

\bib{SSV}{book}{
      author={Schilling, Ren\'{e}~L.},
      author={Song, Renming},
      author={Vondra\v{c}ek, Zoran},
       title={Bernstein functions},
     edition={Second},
      series={De Gruyter Studies in Mathematics},
   publisher={Walter de Gruyter \& Co., Berlin},
        date={2012},
      volume={37},
        ISBN={978-3-11-025229-3; 978-3-11-026933-8},
  url={https://urldefense.proofpoint.com/v2/url?u=https-3A__doi.org_10.1515_9783110269338&d=DwIGaQ&c=KXXihdR8fRNGFkKiMQzstu-8MbOxd1NuZkcSBymGmgo&r=xgjcP5L9SZD7lyUnYxs7d005TSYbpcjMlIILZCrBUmY&m=ik53W52nlHKGpSmqYxuDPEAn4pA9Q4ERYBfJxF8OAyE&s=6PqpYPwiw9w4Zms4hZCo427_QkntYSqRYQIUDM1nsl0&e=},
        note={Theory and applications},
      review={\MR{2978140}},
}

\bib{subag2017complexity}{article}{
      author={Subag, Eliran},
       title={The complexity of spherical {$p$}-spin models---a second moment
  approach},
        date={2017},
        ISSN={0091-1798},
     journal={Ann. Probab.},
      volume={45},
      number={5},
       pages={3385\ndash 3450},
  url={https://urldefense.proofpoint.com/v2/url?u=https-3A__mathscinet.ams.org_mathscinet-2Dgetitem-3Fmr-3D3706746&d=DwIGaQ&c=KXXihdR8fRNGFkKiMQzstu-8MbOxd1NuZkcSBymGmgo&r=xgjcP5L9SZD7lyUnYxs7d005TSYbpcjMlIILZCrBUmY&m=ik53W52nlHKGpSmqYxuDPEAn4pA9Q4ERYBfJxF8OAyE&s=VUkt5859Ulv7gp_sFJU787qTY172RFctWU7t23fkBNY&e=},
      review={\MR{3706746}},
}

\bib{Ya57}{article}{
      author={Yaglom, A.~M.},
       title={Certain types of random fields in {$n$}-dimensional space similar
  to stationary stochastic processes},
        date={1957},
        ISSN={0040-361x},
     journal={Teor. Veroyatnost. i Primenen},
      volume={2},
       pages={292\ndash 338},
      review={\MR{0094844}},
}

\bib{Ya87}{book}{
      author={Yaglom, A.~M.},
       title={Correlation theory of stationary and related random functions.
  {V}ol. {I}},
      series={Springer Series in Statistics},
   publisher={Springer-Verlag, New York},
        date={1987},
        ISBN={0-387-96268-9},
        note={Basic results},
      review={\MR{893393}},
}

\bib{YV18}{article}{
      author={Yamada, Masaki},
      author={Vilenkin, Alexander},
       title={Hessian eigenvalue distribution in a random gaussian landscape},
        date={2018Mar},
        ISSN={1029-8479},
     journal={Journal of High Energy Physics},
      volume={2018},
      number={3},
       pages={29},
  url={https://urldefense.proofpoint.com/v2/url?u=https-3A__doi.org_10.1007_JHEP03-282018-29029&d=DwIGaQ&c=KXXihdR8fRNGFkKiMQzstu-8MbOxd1NuZkcSBymGmgo&r=xgjcP5L9SZD7lyUnYxs7d005TSYbpcjMlIILZCrBUmY&m=ik53W52nlHKGpSmqYxuDPEAn4pA9Q4ERYBfJxF8OAyE&s=obJd0xppKR55eH9IOq3pVhi07ofKWu_dtLOCEwFLITc&e=},
}

\end{biblist}
\end{bibdiv}

\end{document}